\documentclass[reqno,11pt]{amsart}
\usepackage{setspace}
\makeatletter

\usepackage[T1]{fontenc}
\usepackage{lmodern}

\usepackage[a4paper,top=1in,bottom=1in,left=1in,right=1in,marginparwidth=1.75cm]{geometry}
\usepackage{microtype}
\usepackage{textcmds}  
\usepackage{amsmath, amssymb, amsfonts, amstext, verbatim, amsthm, mathrsfs, stmaryrd}
\usepackage{thmtools}
\usepackage{appendix}
\usepackage{mathdots}
\usepackage{ytableau}
\usepackage[all,cmtip,2cell]{xy}
\usepackage{pgf,tikz,pgfplots}
\pgfplotsset{compat=1.15}
\usetikzlibrary{arrows}
\usetikzlibrary{positioning}
\usetikzlibrary{quotes}
\usepackage{tikz-cd}
\usepackage{quiver}
\usepackage{mathtools}
\usepackage[dvipsnames]{xcolor}

\usepackage{graphicx}
\graphicspath{ {images/} }

\tikzset{
  symbol/.style={
    draw=none,
    every to/.append style={
      edge node={node [sloped, allow upside down, auto=false]{$#1$}}}
  }
}

\usepackage{extarrows}

\numberwithin{equation}{section}

\usepackage{enumerate}
\usepackage{enumitem}
\setlist{topsep=1em, itemsep=1em}
\usepackage[colorlinks=true,linkcolor=blue,citecolor=blue,urlcolor=blue,citebordercolor={0 0 1},urlbordercolor={0 0 1},linkbordercolor={0 0 1}]{hyperref} 
\usepackage[shortalphabetic]{amsrefs}
\usepackage[nameinlink, capitalize]{cleveref}

\usepackage{enotez}
\setenotez{backref=true}

\def\mydefc#1{\expandafter\def\csname c#1\endcsname{\mathcal{#1}}}
\def\mydefallc#1{\ifx#1\mydefallc\else\mydefc#1\expandafter\mydefallc\fi}
\mydefallc ABCDEFGHIJKLMNOPQRSTUVWXYZ\mydefallc

\def\mydefb#1{\expandafter\def\csname b#1\endcsname{\mathbf{#1}}}
\def\mydefallb#1{\ifx#1\mydefallb\else\mydefb#1\expandafter\mydefallb\fi}
\mydefallb ABCDEFGHIJKLMNOPQRSTUVWXYZ\mydefallb

\def\mydeffrac#1{\expandafter\def\csname frac#1\endcsname{\mathscr{#1}}}
\def\mydeffracall#1{\ifx#1\mydeffracall\else\mydeffrac#1\expandafter\mydeffracall\fi}
\mydeffracall ABCDEFGHIJKLMNOPQRSTUVWXYZ\mydeffracall

\theoremstyle{plain}
\newtheorem{thm}{Theorem}[section]
\newtheorem{cor}[thm]{Corollary}
\newtheorem{lem}[thm]{Lemma}

\newtheorem{prop}[thm]{Proposition}

\newtheorem{heur}[thm]{Heuristic}

\theoremstyle{definition}
\newtheorem{rem}[thm]{Remark}
\newtheorem{defn}[thm]{Definition}

\newtheorem{setup}[thm]{Setup}

\newtheorem{ex}[thm]{Example}

\newtheorem{introthm}{Theorem}

\crefname{thm}{theorem}{theorems}
\crefname{cor}{corollary}{corollaries}
\crefname{lem}{lemma}{lemmas}
\crefname{prop}{proposition}{propositions}
\crefname{defn}{definition}{definitions}
\crefname{rem}{remark}{remarks}
\crefname{claim}{claim}{claims}
\crefname{conj}{conjecture}{conjectures}
\crefname{ex}{example}{examples}
\crefname{fact}{fact}{facts}
\crefname{quest}{question}{questions}

\Crefname{thm}{Theorem}{Theorems}
\Crefname{cor}{Corollary}{Corollaries}
\Crefname{lem}{Lemma}{Lemmas}
\Crefname{prop}{Proposition}{Propositions}
\Crefname{defn}{Definition}{Definitions}
\Crefname{rem}{Remark}{Remarks}
\Crefname{claim}{Claim}{Claims}
\Crefname{conj}{Conjecture}{Conjectures}
\Crefname{ex}{Example}{Examples}
\Crefname{fact}{Fact}{Facts}
\Crefname{quest}{Question}{Questions}

\def\rm{\mathrm}
\def\d{\mathrm{d}}
\def\K{\mathrm{K}}

\DeclareMathOperator{\Amp}{Amp}
\DeclareMathOperator{\Aut}{Aut}

\newcommand{\DCoh}{\mathcal{D}^{b}}

\DeclareMathOperator{\Coh}{Coh}

\DeclareMathOperator{\Cone}{Cone}

\DeclareMathOperator{\Ext}{Ext}

\DeclareMathOperator{\GL}{GL}
\DeclareMathOperator{\SL}{SL}
\DeclareMathOperator{\PSL}{PSL}

\DeclareMathOperator{\Hom}{Hom}

\newcommand{\Knum}{\mathrm{K}^{\mathrm{num}}}

\DeclareMathOperator{\NS}{NS}
\DeclareMathOperator{\Num}{Num}

\DeclareMathOperator{\rank}{rank}

\DeclareMathOperator{\Roots}{Roots}

\newcommand{\norm}[1]{\left\lVert#1\right\rVert}

\DeclareMathOperator{\Stab}{Stab}

\DeclareMathOperator{\Pic}{Pic}

\DeclareMathOperator{\Ob}{Ob}

\DeclareMathOperator{\im}{im}

\DeclareMathOperator{\logZ}{logZ}
\DeclareMathOperator{\gr}{gr}

\def\bf{\mathbf}

\def\Astab{\mathcal{A}\Stab}

\newcommand{\kn}{{\mathcal{N}}}
\newcommand{\ko}{\mathcal{O}}

\newcommand{\Ku}{\mathrm{Ku}}

\newcommand{\ch}{\mathrm{ch}}

\newcommand{\Zforg}{\mathcal{Z}}
\newcommand{\Bridgeland}{\mathrm{Bri}}
\newcommand{\ArcaraBertram}{\mathrm{AB}}

\newcommand{\hooklongrightarrow}{\lhook\joinrel\longrightarrow}

\usepackage{pbox}
\usepackage[normalem]{ulem}

\makeatletter

\usepackage{babel}

\begin{document}

\title[Degenerate complex structures on surfaces and stability conditions]{Degenerate complex structures on surfaces and stability conditions}

\author[A. Robotis]{Antonios-Alexandros Robotis}
\address{Department of Mathematics, Imperial College London}
\email{a.robotis@ic.ac.uk}

\author[N. Vilches]{Nicol\'{a}s Vilches}
\address{Mathematical Sciences Institute, Australian National University}
\email{nicolas.vilches@anu.edu.au}

\begin{abstract}
    Physical heuristics suggest that the space of Bridgeland stability conditions of a projective Calabi--Yau manifold $X$ should be related to the moduli space of complex structures of its mirror $X^\vee$. On the other hand, recent work in the literature introduces a modular partial compactification, $\Astab(X)$, of a quotient of the space of stability conditions on $\DCoh(X)$. 
    
    We provide some mathematical evidence for the aforementioned heuristics by proving that for curves and surfaces, certain boundary strata in $\Astab(X)$ describe degenerate complex structures on $X^\vee$. Along the way, we introduce a criterion for producing boundary points of $\Astab(\cD)$ when $\cD$ is an arbitrary $k$-linear triangulated category. As a byproduct, we deduce the existence of Bridgeland stability conditions on certain quotients of the derived category of K3 and Enriques surfaces.
\end{abstract}

\maketitle

\tableofcontents

\addtocontents{toc}{\protect\setcounter{tocdepth}{1}}

\linespread{1.15}\selectfont

\section{Introduction}

The study of stability conditions on triangulated categories was initiated by Bridgeland \cite{Br07} as an attempt to mathematize the theory of $\Pi$-stability in physics, as in the work of Douglas \cite{DouglasPi}. In spite of, or perhaps because of, its physical origin, the theory of stability conditions has proven to be a useful tool in resolving problems in algebraic geometry and beyond. 

The basic construction associates to a triangulated category $\cD$ a complex manifold $\Stab(\cD)$ of stability conditions \cite{Br07}. When $\cD = \DCoh(X)$, the bounded derived category of coherent sheaves on a variety $X$, it is frequently possible to construct moduli spaces $M_\sigma(v)$ of $\sigma$-stable objects with fixed numerical class $v \in \Knum(X)$, where $\sigma \in \Stab(X)$. The study of the dependence of the moduli space $M_\sigma(v)$ on the parameter $\sigma \in \Stab(X)$ often gives rise to wall-crossing phenomena, which have recently shed light on many problems in algebraic geometry --- see \cites{BMunreasonable,BMMMPK3,ProjbiratBM,MukaiWCF} for several examples.

In spite of this plethora of applications to algebraic geometry, many of the original heuristics about $\Stab(X)$ come from the physical theory of mirror symmetry. For a Calabi--Yau manifold $X$, mirror symmetry predicts an isomorphism $\cM_{K}(X) \cong \cM_{\bf{C}}(X^\vee)$, where $\cM_K(X)$ is the physically defined \emph{stringy K\"ahler moduli space} of $X$ and $\cM_{\bf{C}}(X^\vee)$ is the moduli space of complex structures on the mirror manifold $X^\vee$. Unfortunately, the stringy K\"ahler moduli space seems to evade mathematical definition; nevertheless, it is posited that $\cM_K(X)$ should embed in the double quotient $\Aut(\cD)\backslash \Stab(X)/\bC$, at least when $X$ is a Calabi--Yau manifold --- cf. \cite{BrSpacesofstability}*{Rem. 3.9}. The mirror to a compact Calabi--Yau manifold $X$ is itself a Calabi--Yau manifold $X^\vee$ of the same dimension and the heuristics above predict an embedding 
\begin{equation}
\label{E:embeddingSKMS}
    \cM_{\bC}(X^\vee) \xlongrightarrow{\sim} \cM_K(X) \hooklongrightarrow \Aut(\DCoh(X))\backslash \Stab(X)/\bC.
\end{equation}

The spaces in \eqref{E:embeddingSKMS} are in general non-compact as seen, for instance, in the example of $X$ an elliptic curve \cite{BrSpacesofstability}*{\S 4.1}. On the other hand, the recent work \cite{augmented} introduces a modular partial compactification $\Astab(X)$ of ${\Stab(X)}/\bC$ to which the $\Aut(\DCoh(X))$-action extends. The points of $\Astab(X)$ are called \emph{augmented stability conditions}. 

In light of these developments, it is reasonable to ask if certain boundary points of $\Astab(X)$, or $\Aut(\DCoh(X))\backslash \Astab(X)$, can be interpreted as boundary points of $\cM_{\bC}(X^\vee)$, i.e. degenerate complex structures on the Calabi--Yau manifold $X^\vee$. We summarize this as:

\begin{heur}
\label{Heur:CY}
    Let $X$ be a Calabi--Yau manifold. There should exist certain boundary strata of $\Astab(X)$ which correspond to degenerate complex structures on the mirror $X^\vee$.
\end{heur}

The purpose of the present work is to develop tools to investigate Heuristic~\ref{Heur:CY} in examples. More precisely, our main results, summarized below, give effective criteria for constructing paths in $\Stab(X)/\bC$ converging to boundary points in $\Astab(X)$. We apply these criteria to study the cases of categories $\cD$ with rank two numerical K-theory, such as $\DCoh(X)$ for $X$ a smooth projective curve, and the derived categories of K3 and Enriques surfaces.

\subsection{Main results}

We study boundary points of $\Astab(\cD)$ arising as limits of quasi-convergent paths in $\Stab(\cD)/\bC$ as introduced in \cite{quasiconvergence}. In Section~\ref{sec:ranktwo}, we consider the situation where the rank of $\Knum(\cD)$ is two. A notable special case covered in this section is when $X$ is a Calabi--Yau curve, i.e. an elliptic curve. 

In this case, the space $\Stab(X)$ is completely understood, being biholomorphic to $\bH\times \bC$ by \cite{Br07}*{Thm. 9.1}. Thus, $\Stab(X)/\bC$ is isomorphic to $\bH$ and it is observed in \cite{BrSpacesofstability}*{\S 4.1} that $\Aut(\DCoh(X))$ acts through the standard action of $\SL_2(\bZ)$ on $\bH$, the spherical twists $T_{\cO_X}$ and $T_{\cO_x}$ being mapped to generators. Thus, the double quotient $\Aut(\DCoh(X))\backslash \Stab(X)/\bC$ is identified with $\SL_2(\bZ)\backslash \bH$, the moduli stack $\cM_{1,1}$ of elliptic curves, which can be interpreted as $\cM_{\bC}(X^\vee)$. Our first result gives a description of $\Astab^\dagger(X)$, i.e. the closure of $\Stab(X)/\bC$ in $\Astab(X)$.

\begin{introthm} \label{thm:intro_elliptic}
\label{T:introthmelliptic}
    The space $\Astab^\dagger(X)$ is homeomorphic to $\bH\cup \bQ\cup \{\infty\} \subset \bP^1$ such that $\Stab(X)/\bC$ corresponds to $\bH$. Each $\mu \in \bQ \cup \{\infty\}$ corresponds to a limit point where stable sheaves of slope $\mu$ become massless. In particular, 
    \[
        \Aut(\DCoh(X))\backslash \Astab(X) \cong \overline{\cM}_{1,1}. 
    \]
\end{introthm}
Under homological mirror symmetry for $X$ \cite{PolishchukZaslow}, the sheaves $\cO_X$ and $\cO_x$ correspond to linear Lagrangians $L_1,L_2$ in the torus $X^\vee$ whose classes form a basis for $H_1(X^\vee,\bZ)$. The boundary point in $\Astab^\dagger(X)$ corresponding to $0$ can be interpreted as a limit of complex structures $\lim_{t\to\infty} \cJ_t$ in the complex moduli of $X^\vee$, where 
\begin{equation}
\label{E:ratio_of_charges}
    \frac{Z_t(\cO_X)}{Z_t(\cO_x)} = \frac{\int_{L_1} \Omega_t}{\int_{L_2}\Omega_t} \to 0.
\end{equation}
Here, $\Omega_t$ is a non-zero $1$-form on $X^\vee$ holomorphic for $\cJ_t$; note that the ratio of charges \eqref{E:ratio_of_charges} is independent of the choice of $\Omega_t$. In the limit, $L_1$ becomes massless relative to $L_2$ and we can interpret this limit as contributing a limit point where $L_1$ is shrunk to a point, i.e. a nodal torus. Quotienting by $\SL_2(\bZ)$ removes the particular choice of basis for $H_1(X^\vee,\bZ)$ used to construct the limit. Thus, Theorem~\ref{T:introthmelliptic} provides the first evidence for Heuristic~\ref{Heur:CY}.

In Section~\ref{sec:qcpaths}, we study conditions on paths $\sigma_\bullet\colon[0,\infty) \to \Stab(\cD)$ which ensure convergence in $\Astab(\cD)$ as $t\to\infty$. The first condition, \ref{axiom1}, requires that $\partial_t \Im(Z_t) = 0$. This has the key consequence that the category $\cA_t:=\cP_{\sigma_t}(0,1]$ is independent of $t$ --- see Lemma~\ref{L:constantheart}. We say that such paths have \emph{constant heart}. 

The second condition, \ref{axiom2}, requires that $\lim_{t\to\infty}Z_t = Z_\infty$ exists, and also that $Z_t$ depends analytically on $t$ near $\infty$. This can be regarded as a strengthening of Definition~\ref{D:quasiconvergent}(2). We further consider a pair of formulations \ref{axiom3} and \ref{axiom4} of the support property for paths as in \cite{quasiconvergence}*{Defn. 2.37}. Assuming \ref{axiom1} and \ref{axiom2}, these conditions are equivalent --- see Lemma~\ref{L:equivsupportproperty}. Our main technical result is:

\begin{introthm}
[= Theorem~\ref{T:strongquasiconvergence}] \label{thm:intro_qcpaths}
    If $\sigma_\bullet\colon[0,\infty)\to \Stab(\cD)$ is given such that \ref{axiom1}, \ref{axiom2}, and \ref{axiom3} hold then the induced path in $\Stab(\cD)/\bC$ is strongly quasi-convergent and thus converges in $\Astab(\cD)$. 
\label{T:secondintrothm}
\end{introthm}

The rest of the paper is dedicated to applying Theorem~\ref{T:secondintrothm} in the stability manifolds of K3 and Enriques surfaces. Using the now-classical constructions of stability conditions on surfaces in \cites{ABLStabilitysurf,BridgelandK3} by tilting, it is fairly easy to construct paths $\sigma_\bullet$ satisfying \ref{axiom1} and \ref{axiom2}. The condition \ref{axiom4}, and thus \ref{axiom3}, can be checked using formal computations with the Euler form in a 2-Calabi--Yau or Enriques category and we exploit this fact. 

\begin{introthm}[= Theorem \ref{thm:pathsK3_main}, simplified] 
\label{thm:intro_K3}
Let $S$ be a K3 surface. \vspace{-2mm}
\begin{enumerate}
    \item There exists a path $\sigma_\bullet$ in $\Stab(S)/\bC$ which converges in $\Astab(S)$ to a boundary point with associated filtration $0\subset \langle \cO_S\rangle \subset \DCoh(S)$. \vspace{-2mm}
    \item Assume that $S$ has Picard rank 1 and let $E$ be a spherical vector bundle. There exists a path $\sigma_\bullet$ in $\Stab(S)/\bC$ which converges in $\Astab(S)$ to a boundary point with associated filtration $0\subset \langle E\rangle \subset \DCoh(S)$. \vspace{-2mm}
    \item Let $f\colon S \to T$ be a crepant resolution of a projective normal surface $T$ with ADE singularities with exceptional curves $C_1, \dots, C_n$. There exists a path $\sigma_\bullet$ in $\Stab(S)/\bC$ which converges in $\Astab(S)$ to a boundary point with  associated filtration 
    \[
        0\subset \langle \cO_{C_1}(-1), \dots, \cO_{C_n}(-1) \rangle \subset \DCoh(S).\vspace{-2mm}
    \]
\end{enumerate}
\end{introthm}

In particular, as a consequence of the quasi-convergent paths formalism, we produce Bridgeland stability conditions in $\DCoh(S)/\langle \ko_S\rangle$ for an arbitrary K3 surface $S$ and $\DCoh(S)/\langle E\rangle$ when $S$ has Picard rank 1. In the third case, $\DCoh(S)/\langle \cO_{C_1}(-1), \dots, \cO_{C_n}(-1)\rangle \simeq \DCoh(T)$, which was known to admit stability conditions --- cf. \cite{Langernormal} and the references therein. Nevertheless, we provide an alternative construction of stability conditions on $T$ as limits of stability conditions on its resolution $S$.

According to Heuristic~\ref{Heur:CY}, some of the quasi-convergent paths in Theorem~\ref{thm:intro_K3} should converge to boundary strata in $\Astab(S)$ that correspond to degenerate complex structures on a K3 surface. The objects generating the subcategories of $\DCoh(S)$ in Theorem~\ref{thm:intro_K3} are 2-spherical, and thus correspond to Lagrangian spheres $L$ in the mirror. The spherical objects here are \emph{massless} for $\sigma_\bullet$ as defined in Section~\ref{sec:qcpaths}, and we expect the corresponding variation of complex structure to result in 
\begin{equation}
\label{E:chargenormal}
    Z_t(L) = \int_L \Omega_t\to 0,
\end{equation}
where $\Omega_t$ is a $\cJ_t$-holomorphic 2-form. 

In the context of Theorem~\ref{thm:intro_K3}(3), the mirror to $T$ is a nodal K3 surface $X_0$ and the mirror to the contraction $f\colon S\to T$ is a smoothing of the nodes of $X_0$, i.e. a family of K3 surfaces $\cX\to B$ with smooth generic fiber and singular fiber $X_0$ --- see \cite{STYConifold}*{p. 2} and the references therein. Thus, mirror to the spherical objects $\cO_{C_1}(-1),\ldots, \cO_{C_n}(-1)$ are Lagrangian vanishing cycles $L_1,\ldots, L_n$ in $\cX$ as $t\to 0$ in $B$ and the charges $Z_t(L_i)$ as in \eqref{E:chargenormal} describe smoothing parameters of the nodal singularities of $X_0$.\footnote{We thank Richard Thomas for pointing out this interpretation to us.} Thus, Heuristic~\ref{Heur:CY} in realized as:

\begin{introthm}
[= Theorem \ref{T:stratum_description}, simplified]
    Let $f\colon S\to T$ be as in Theorem~\ref{thm:intro_K3} and suppose that the fibers of $f$ are irreducible. There is a stratum $\cM_f$ of $\partial\Astab(S)$ which admits the structure of a holomorphic fiber bundle over $\Stab(T)/\bC$. The associated paths $\sigma_\bullet$ as in Theorem~\ref{thm:intro_K3}(3) converge to points of $\cM_f$.
\end{introthm}

See Remark~\ref{R:parameters} for a further discussion. Finally, we apply the same type of ideas to construct quasi-convergent paths in the stability manifold of Enriques surfaces:

\begin{introthm}[= Theorem \ref{thm:pathsEnriques_main}] \label{thm:intro_Enriques}
Let $T$ be an Enriques surface. \vspace{-2mm}
\begin{enumerate}
    \item There exists a path $\sigma_\bullet$ in $\Stab(T)/\bC$ which converges in $\Astab(T)$ to a boundary point with associated filtration $0\subset \langle \cO_T, \cO_T(K_T) \rangle \subset \DCoh(T)$. \vspace{-2mm}
    \item Assume that $T$ is unnodal, and fix a Fano polarization $\Delta$. Let $F_1^\pm, \dots, F_{10}^\pm$ be the half-fibers associated to the ten nef classes $f_1, \dots, f_{10}$ satisfying $f_i^2=0$ and $f_i.\Delta=3$. There exists a path $\sigma_\bullet$ in $\Stab(T)/\bC$ which converges in $\Astab(T)$ to a boundary point with associated filtration $0\subset \langle \cO_T(F_i^+), \cO_T(F_i^-): 1 \leq i \leq 10\rangle \subset \DCoh(T)$.\vspace{-2mm}
\end{enumerate}
\end{introthm}

These quasi-convergent paths induce Bridgeland stability conditions on $\DCoh(T)/\langle \ko_T, \ko_T(K_T)\rangle$ on any Enriques surface, and on $\DCoh(T)/\langle \ko_T(F_i^+), \ko_T(F_i^-): 1 \leq i \leq 10\rangle$ on any unnodal Enriques surface.

\subsection{Relations to other work}
The subject of (partial) compactifications of spaces of stability conditions has received a fair amount of attention in recent years, with several different groups providing their own viewpoints on the problem. We refer to \cite{augmented}*{p. 8} for a brief survey of some of these approaches. In the work \cite{BDLThurston}, the so-called \emph{Thurston compactification} of $\Stab(\cD)$ is introduced and studied in the 2-Calabi--Yau categories associated to the $A_2$ and Kronecker quivers. Subsequently, this compactification was studied for derived categories of generic analytic K3 surfaces in \cite{DeopurkarThurstonk3}. These examples all fit into the rank two regime described in Section~\ref{sec:ranktwo}, and we hope that our results will allow for a precise comparison of the Thurston compactification to the space $\Astab(\cD)$ in these cases.

In the case of an algebraic surface $S$, paths in $\Stab(S)$ where $Z(\cO_S)\to 0$, as in Theorem~\ref{thm:intro_K3}, are related to dimension bounds on global sections of semistable bundles supported on an embedded curve $C\hookrightarrow S$. We refer to \cite{FL21}*{p. 4} and the references therein for related discussions.

In \cite{CGL+23}, an abstract notion of a \emph{nodal category} $\cD$ is introduced. Examples of these are obtained by taking  $\DCoh(S)/\langle E \rangle$ for a K3 surface $S$ and a spherical vector bundle $E$. One of the questions raised in \emph{op. cit.} relates to whether $\cD$ is \emph{geometric}, i.e. equivalent to an admissible subcategory of $\DCoh(Y)$ for $Y$ a normal quasi-projective variety. Theorem~\ref{thm:intro_K3} supports this interpretation by producing stability conditions on $\DCoh(S)/\langle E\rangle$, when $S$ has Picard rank 1. 

A result similar to Theorem~\ref{thm:intro_K3}(3) was obtained by Chou \cite{Cho24} when $f\colon S \to T$ is a crepant resolution of a normal projective surface $T$ with a single ADE singularity. He studies a degeneration to a \emph{weak} stability condition as opposed to an \emph{augmented} stability condition. Additionally, his result invokes Langer's construction of stability conditions on singular surfaces \cite{Langernormal}, whereas our techniques produce elements of $\Stab(T)$ by studying stability conditions on the smooth category $\DCoh(S)$.

Finally, any Enriques surface $T$ admits a semiorthogonal decomposition 
\[
    \DCoh(T) =\langle \cO_T(F_1), \dots, \cO_T(F_{10}), \Ku(T)\rangle,
\]
cf. \cite{LNSZ21}*{\textsection 3.1}. When $T$ is unnodal, the divisors $F_1, \dots, F_{10}$ can be picked as ten of the half-fibers $F_1^+, \dots, F_{10}^+$ associated to a Fano polarization, cf. \cite{Zub97}*{\textsection 5}. It is currently unknown whether or not $\Ku(T)$ admits Bridgeland stability conditions. The projections of the other ten half-fibers $\ko_T(F_1^-), \dots, \ko_T(F_{10}^-)$ into $\Ku(T)$ are 3-spherical with zero class in numerical K-theory, hence they cannot be semistable for any stability condition in $\Ku(T)$. This could provide a possible obstruction to existence of stability conditions on $\Ku(T)$. A beautiful proposal of Ingalls and Kuznetsov \cite{IK15} relates $\Ku(T)$ to the derived category of a crepant resoution of an Artin--Mumford double solid in the case where $T$ is nodal --- cf. \citelist{\cite{LNSZ21}\cite{HT20}}. Thus, the fact that such resolutions are not projective should be related to $\Ku(T)$ not admitting Bridgeland stability conditions. 

Theorem~\ref{thm:intro_Enriques} constructs stability conditions on the Verdier quotient $\Ku(T)/\langle E_1, \dots, E_{10}\rangle$, where $E_i$ is the projection of $\ko_T(F_i^-)$ into $\Ku(T)$ and $T$ is unnodal. In the nodal case, this quotient is expected to be related to a semiorthogonal component of the Artin--Mumford double solid $Y$. Although singular, these varieties are projective, and thus $\Stab(Y)$ is non-empty by \citelist{\cite{Li26} \cite{LLL26}}; this suggests that $\Ku(T)/\langle E_1, \dots, E_{10}\rangle$ should admit stability conditions in the nodal case also.

\subsection{Structure of the paper}
We review some preliminaries in Section \ref{S:preliminaries}, including Bridgeland stability conditions, quasi-convergent paths, and augmented stability conditions. Section \ref{sec:ranktwo} is devoted to the study of $\Astab(\cD)$ for rank two categories $\cD$. As a very special case, we obtain a proof of Theorem \ref{thm:intro_elliptic}. In Section~\ref{sec:qcpaths}, we prove Theorem~\ref{thm:intro_qcpaths} which is the technical backbone of this work. We apply the criterion of Theorem~\ref{thm:intro_qcpaths} in Section \ref{sec:K3} to K3 surfaces, and in Section \ref{sec:Enriques} to Enriques surfaces.

\subsection*{Notation and conventions}
We collect here some notation used throughout the body of the paper. Given an Abelian group $A$ and a subgroup $B$, we define the \emph{saturation} of $B$ in $A$ to be 
\[
    B^{\rm{sat}} := \left\{ v\in A: n\cdot v\in B \text{ for some }n \ge 1 \right\}.
\]
Given a pair of real-valued functions $f,g : I\subset \bR\to \bR$ and a distinguished point $a$ in $\overline{I} \subset \bR\cup\{\pm \infty\}$ we say $f\sim_a g$ if $\lim_{t\to a} f(t)/g(t) = 1$ and $f\approx_a g$ if $\lim_{t\to a} f(t) - g(t) = 0$. Usually, $a$ is clear from the context and we omit it from the notation. 

If $\cD$ is a triangulated category and $\{\cA_\alpha\}_{\alpha \in A}$ a collection of subcategories we let $\langle \cA_\alpha:\alpha \in A\rangle$ denote the smallest strict full subcategory of $\cD$ containing $\bigcup_{\alpha \in A}\Ob(\cA_\alpha)$. In the same notation, $\langle \cA_\alpha:\alpha \in A\rangle_{\rm{ext}}$ denotes the smallest extension closed subcategory containing $\bigcup_{\alpha \in A}\Ob(\cA_\alpha)$. 

We write $G$ for the group $\GL_2^+(\bR)^{\sim}$, which acts on $\Stab(\cD)$ by \cite{Br07}*{Lem. 8.2}. The symbol $\mathtt{i}$ denotes a fixed complex number satisfying $\mathtt{i}^2 = -1$. Given a complex number $z$, $\Re(z)$ denotes its real part and $\Im(z)$ its imaginary part.

\subsection*{Acknowledgements} It is our pleasure to thank Arend Bayer, Daniel Halpern-Leistner, Elise LePage, Emanuele Macrì, Richard Thomas, and Xiaolei Zhao for helpful conversations related to the contents of this paper.

N.V. thanks the Mathematisches Forschungsinstitut Oberwolfach and the organizers of the workshop \emph{Moduli Spaces in Algebraic Geometry} for providing an exciting environment. The results in Theorem \ref{thm:intro_Enriques} were inspired by a lunch discussion of N.V. with Arend Bayer, Emanuele Macrì, and Xiaolei Zhao at the institute.

A.R. was supported by NSF grant DMS-250340 and Royal Society University Research Fellowship URF/R1/231191. N.V. gratefully acknowledges support by the Australian Research Council grant DP240101084.

A.R. thanks Maria Teresa for her love and support during the preparation of this work. 

\subsection*{AI Statement}
LLMs and other AI systems were not used in the preparation of this work, besides to help produce code for the several TikZ figures featured below.

\section{Preliminaries}
\label{S:preliminaries}

In this section, we review background material from the works \cites{quasiconvergence,augmented} in the hopes of making the present work more self-contained. Throughout the paper, $\cD$ denotes a $k$-linear triangulated category with a fixed homomorphism $v\colon \K_0(\cD) \twoheadrightarrow \Lambda$, where $\Lambda$ is a finite rank free Abelian group. To fix notation, we review the definition of a Bridgeland stability condition:

\begin{defn}
[\cite{Br07}] A \emph{pre-stability condition} on $\cD$ is a pair $(Z,\cP)$ where $Z\in \Hom(\K_0(\cD),\bC)$ is called the \emph{central charge} and $\cP$ is a collection $\{\cP(\phi)\}_{\phi \in \bR}$ of full additive subcategories of $\cD$ called a \emph{slicing} such that: \vspace{-2mm}
\begin{enumerate}
    \item for all $\phi \in \bR$ one has $\cP(\phi)[1] = \cP(\phi+1)$; \vspace{-2mm}
    \item if $\psi > \phi$ then $\Hom(\cP(\psi),\cP(\phi)) = 0$; and \vspace{-2mm}
    \item every non-zero object $E$ of $\cD$ admits a \emph{Harder--Narasimhan (HN) filtration}, i.e. a sequence of morphisms $0 = E_0 \to E_1\to \cdots\to E_{n} = E$ such that $H_i = \Cone(E_{i-1}\to E_i) \in \cP(\phi_i)$ for each $i=1,\ldots, n$ and such that $\phi_1>\cdots>\phi_n$. \vspace{-2mm}
\end{enumerate}
Non-zero objects of $\cP(\phi)$ are said to be semistable of phase $\phi$. The objects $H_i$ arising from the HN filtration are called HN \emph{factors} of $E$. The pair $(Z,\cP)$ is required to satisfy a compatibility condition: for all non-zero $E\in \cP(\phi)$ we have 
\[
    Z(E) = m(E)\cdot \exp(\mathtt{i}\pi \phi)
\]
where $m(E) \in \bR_{>0}$ is called the \emph{mass} of $E$. We also write $\cP$ to mean the union $\bigcup_{\phi \in \bR} \cP(\phi)$. 

A pre-stability condition on $\cD$ is a \emph{stability condition} if in addition its charge $Z$ factors through $v\colon \K_0(\cD)\twoheadrightarrow \Lambda$ and it satisfies the support property of Kontsevich--Soibelman \cite{KS08}: for some (equiv. any) norm $\lVert\:\cdot\:\rVert$ on $\Lambda_{\bR}$ we have
\[
    \inf_{0\ne E \in \cP} \frac{\lvert Z(E)\rvert}{\lVert v(E)\rVert} > 0.
\]
Let $\Stab(\cD)$ denote the set of stability conditions on $\cD$.
\label{D:stabilitycondition}
\end{defn}

\begin{ex}
    In most of what follows, $\cD = \DCoh(X)$, the bounded derived category of coherent sheaves on $X$. In this case, we take $\Lambda = \Knum(X)$ and take $v$ to be the quotient map $\K_0(X) \twoheadrightarrow \Knum(X)$. The resulting stability conditions are called \emph{numerical}. We denote by $\Stab(X)$ the space of numerical stability conditions on $\DCoh(X)$.
\end{ex}

The most important fact about $\Stab(\cD)$ is:

\begin{thm}
[Bridgeland, \cites{Bayer_short,Br07}] The set $\Stab(\cD)$ admits a natural metric topology such that the map $\cZ\colon \Stab(\cD) \to \Hom(\Lambda,\bC)$ given by $(Z,\cP)\mapsto Z$ is a local homeomorphism onto an open subspace. 
\label{T:deformation}
\end{thm}

An immediate consequence of Theorem~\ref{T:deformation} is that $\Stab(\cD)$ has a canonical structure of a  complex manifold of dimension equal to the rank of $\Lambda$. Before proceeding, we recall several other functions defined using (pre-)stability conditions. 

First, in the notation of Definition~\ref{D:stabilitycondition}, given a stability condition $\sigma$ and a non-zero object $E$, we put $\phi^+(E) = \phi_1$ and $\phi^-(E) = \phi_n$. We sometimes write $\phi^{\pm}_\sigma$ to emphasize the choice of $\sigma$. Next, we define the \emph{mass} of $E$ to be $m(E) = \sum_{i=1}^n m(H_i)$. Finally, it is also useful to define the average phase of an object $E$ as 
\[
    \phi_\sigma(E) := \frac{1}{m_\sigma(E)} \sum_{i=1}^n m_\sigma(H_i)\cdot \phi_\sigma(H_i). 
\]
The average phase is defined for any non-zero object and when $E$ is semistable it agrees with the usual phase. When considering a path of stability conditions $\{\sigma_t\}_{t\in I\subset \bR}$ we write $\phi_t = \phi_{\sigma_t}$, etc.

Finally, the space $\Stab(\cD)$ admits a smooth action by $G := \GL_2^+(\bR)^{\sim}$ which lifts the right action $\Hom(\Lambda,\bC) \times G \to \Hom(\Lambda,\bC)$ given by $(f,g) = f\circ g$, where we identify $\bC = \bR^2$. The action of $\bC^* = \GL_1(\bC)\subset G$ lifts to a free holomorphic action of $\bC$ on $\Stab(\cD)$.

\subsection{Quasi-convergent paths} The work \cite{quasiconvergence} relates paths in $\Stab(\cD)$ to decomp\-ositions of $\cD$. The philosophy of \emph{loc. cit.} is that in the presence of a suitably tame path $\sigma_\bullet\colon [0,\infty)\to \Stab(\cD)$ we can filter $\cD$ by studying the asymptotic behavior of the functions $m_t(-), \phi_t^{\pm}(-):\Ob(\cD)\to \bR$ as $t\to\infty$. The requisite notion of tameness is:

\begin{defn}
[\cite{quasiconvergence}*{Defns. 2.6 \& 2.8}] 
A non-zero object $E$ is called \emph{limit semistable} with respect to $\sigma_\bullet$ if $\lim_{t\to\infty} \phi_t^+(E) - \phi_t^-(E) = 0$.  The path $\sigma_\bullet$ is said to be \emph{quasi-convergent} if \vspace{-2mm}
\begin{enumerate}
    \item for any $E\in \Ob(\cD)$ there exists a sequence of morphisms $0 = E_0\to E_1\to \cdots \to E_n = E$ in $\cD$ such that $G_i = \Cone(E_{i-1}\to E_i)$ is limit semistable and 
    \[
        \liminf_{t\to\infty} \phi_t(G_i) - \phi_t(G_{i+1}) > 0
    \]
    for each $i=1,\ldots, n$; and \vspace{-2mm}
    \item for each pair of limit semistable objects $E$ and $F$, the limit
    \begin{equation}
    \label{E:ratiolimit}
        \lim_{t\to\infty}  \frac{\logZ_t(E) - \logZ_t(F)}{1+\lvert \logZ_t(E) - \logZ_t(F)\rvert}
    \end{equation}
    exists. Here, since $\phi_t^+(E) - \phi_t^-(E)\to 0$ for all limit semistable objects $E$, there is a canonical logarithm of $Z_t(E)$ for all $t \gg 0$, uniquely defined such that $\phi_t^-(E)\le \Im (\logZ_t(E)) \le \phi_t^+(E)$. We write $\mathfrak{n}_t(E,F)$ for the ratio in \eqref{E:ratiolimit}. \vspace{-2mm}
\end{enumerate}
\label{D:quasiconvergent}
\end{defn}

Much of the remainder of the present work is dedicated to studying quasi-convergent paths in $\Stab(X)$, where $X$ is a projective curve or surface. The reader is encouraged to consult \cite{quasiconvergence} for basic properties of quasi-convergent paths. We will restrain ourselves here to only stating the main result of the theory. 

Let $\sigma_\bullet$ be a quasi-convergent path and  $\cP_{\sigma_\bullet}$ its class of limit semistable objects. There exists a total preorder $\preceq^i$ on $\cP_{\sigma_\bullet}$ such that $E\prec^i F$ if and only if $\lim_{t\to\infty} \phi_t(F) - \phi_t(E) = \infty$. We denote the induced equivalence relation on $\cP_{\sigma_\bullet}$ by $\sim^i$. The relation $\sim^i$ declares two limit semistable objects equivalent if and only if the difference of their phase functions converges \cite{quasiconvergence}*{Lem. 2.16}. The resulting quotient $\mathcal{I} := {\cP_{\sigma_\bullet}}/{\sim^i}$ can be regarded as a set of asymptotics of $\Im \logZ_t(-)$ as $t\to\infty$. 

We also use a finer total preorder $\preceq$ on $\cP_{\sigma_\bullet}$ such that $E\prec F$ if and only if $E\prec^i F$ or $E\sim^i F$ and $\lim_{t\to\infty} m_t(E)/m_t(F) = 0$. The preorder $\preceq$ should be regarded as a lexicographic ordering on $\logZ_t(E)$ for $E\in \cP_{\sigma_\bullet}$: first, one compares the growth of the imaginary part and then that of the real part. We define $\cT := {\cP_{\sigma_\bullet}}/{\sim}$. Since $\sim$ is a refinement of $\sim^i$,  there is a canonical map $\varpi\colon \cT \twoheadrightarrow \cI$ and for any $\mathfrak{i} \in \cI$, $\varpi^{-1}(\mathfrak{i})$ is totally ordered by the order induced by $\preceq$, which compares order of growth of masses. Given $\rho \in \mathfrak{i}$, we let $
\mathfrak{i}_{\le \rho} = \{\psi \in \mathfrak{i}:\psi \preceq \rho\}$.

\begin{thm}
[Summary of \cite{quasiconvergence}*{\S2}] 
    Let $\sigma_\bullet\colon [0,\infty)\to \Stab(\cD)$ be a quasi-convergent path. For each $\mathfrak{i} \in \cI$, let $\cD_\mathfrak{i} := \langle E: E\in \mathfrak{i}\rangle$. There is a semiorthogonal decomposition 
    \begin{equation}
    \label{E:SODqc}
        \cD = \left\langle \cD_{\mathfrak{i}}:\mathfrak{i}\in \cI, \text{ ordered by }\preceq^i\right\rangle.
    \end{equation}
    Furthermore, each $\cD_{\mathfrak{i}}$ is filtered by the thick triangulated subcategories $\cD_{\le \rho} = \langle E: E\in \mathfrak{i}_{\le \rho} \rangle$ where $\rho \in \varpi^{-1}(\mathfrak{i})$. Additionally, given $E\in \rho$, the category $\cD_{\le \rho}/\cD_{<\rho}$ admits a pre-stability condition $\sigma_\rho^E = (Z_E,\cP_E)$ where 
    \[
        Z_E(-) = \lim_{t\to\infty} \frac{Z_t(-)}{Z_t(E)}
    \]
    and $\cP_E(\phi)$ is the essential image of the class of those $F\in \mathfrak{i}_{\le \rho}$ such that $\lim_{t\to\infty} \phi_t(F) - \phi_t(E) = \phi$. If, in addition, $\sigma_\bullet$ satisfies the support property \cite{quasiconvergence}*{Defn. 2.37} then each $\sigma_\rho^E$ satisfies the support property with respect to a canonically defined sub-quotient lattice $\Lambda_\rho$ of $\Lambda$. 
\label{T:quasiconvergencesummary}
\end{thm}

\begin{rem}
\label{R:modC}
    Note that for any $\rho$ as in Theorem~\ref{T:quasiconvergencesummary} and any pair $E,F\in \rho$ the pre-stability conditions $\sigma_\rho^E$ and $\sigma_\rho^F$ are in the same orbit for the $\bC$-action. Thus, we obtain canonically defined elements $\sigma_\rho \in \Stab^{\rm{pre}}(\cD_{\le \rho}/\cD_{<\rho})/\bC$. Observe also that the conditions defining quasi-convergence are all $\bC$-invariant. Thus, it is reasonable to interpret quasi-convergent paths as naturally living in the quotient $\Stab(\cD)/\bC$. 
\end{rem}

\begin{ex}
    In the present work, we are primarily interested in the case where the semi\-orthogonal decomposition \eqref{E:SODqc} is trivial, i.e. when $\cI$ is a singleton. Suppose that $\sigma_\bullet$ is a quasi-convergent path with this property and that for each $E \in \cP_{\sigma_\bullet}$ we have $m_t(E) \sim C_E \cdot t^{k}$ for some $k \in \bZ$ --- cf. Lemma~\ref{L:vanishing_order}. Then, as totally ordered sets 
    \[
        \cT \cong \left\{k\in \bZ: \text{ there exists }E \in \cP_{\sigma_\bullet} \text{ with } m_t(E)\sim C_E\cdot t^k\right\}.
    \]
    Theorem~\ref{T:quasiconvergencesummary} thus gives a filtration of $\cD$ by categories 
    \[
        \cD_{\le k} = \left\langle E \in \cP_{\sigma_\bullet}:m_t(E) \sim C_E\cdot t^\ell \text{ for } \ell \le k\right\rangle
    \]
    along with (pre-)stability conditions on the associated graded categories $\cD_{\le k}/\cD_{<k}$. Note that in this case there is a canonical representative of the $\bC$-orbit obtained by defining $Z_{(k)}(E) = \lim_{t\to\infty} Z_t(E)/t^k$. 
\end{ex}

\subsection{Augmented stability conditions} Theorem~\ref{T:quasiconvergencesummary} and Remark~\ref{R:modC} suggest that quasi-convergent paths should be regarded as convergent in an enlargement of $\Stab(\cD)/\bC$. The resulting partial compactification constructed in \cite{augmented} is called the space of \emph{augmented} stability conditions, denoted $\Astab(\cD)$. The boundary points of $\Astab(\cD)$ parametrize the data of a \emph{multi-scale} decomposition of $\cD$, a generalization of a filtration of $\cD$, equipped with stability conditions in a suitable sense. 

The full definition of $\Astab(\cD)$ in \cite{augmented} is somewhat involved, but we attempt here to summarize the relevant information for what follows. We present a simplified version of the main definitions of \emph{loc. cit.} which is sufficient for the present work.

\begin{defn}
\label{D:multiscaleline}
    A \emph{$\le$2-level multi-scale line} is determined by a triple $(\Sigma,p_\infty,\omega_\bullet)$ where $\Sigma$ is a proper nodal genus zero complex curve equipped with a smooth marked point $p_\infty$ and a collection $\omega_\bullet$ of meromorphic 1-forms, one per irreducible component of $\Sigma$, such that:\vspace{-2mm}
    \begin{enumerate}
        \item The point $p_\infty$ lies on a distinguished irreducible component $\Sigma_\infty$, which gives the dual graph $\Gamma(\Sigma)$ of $\Sigma$ the structure of a rooted tree. The two-level condition forces requires that all non-root components are connected to $\Sigma_\infty$ by a node. We call the non-root components \emph{terminal}. \vspace{-2mm}
        \item If $\Sigma$ is reducible, we require that there be at least $2$ terminal components. \vspace{-2mm}
    \end{enumerate}
    We write $V(\Sigma)$ for the set of vertices of $\Gamma(\Sigma)$, i.e. the set of irreducible components of $\Sigma$. \vspace{-2mm}
    \begin{enumerate}[resume]
        \item In the collection $\omega_\bullet = \{\omega_v\}_{v\in V(\Sigma)}$, each $1$-form $\omega_v$ has a unique pole of order $2$ at the node connecting $\Sigma_v$ to $\Sigma_\infty$ (resp. $p_\infty$ when $\Sigma_v = \Sigma_\infty$) and no other zeros or poles.\vspace{-2mm}
    \end{enumerate}
\end{defn}

\begin{ex}
    It is not hard to classify the $\le$2-level multi-scale lines. There is the $1$-level case, where $\Sigma$ is irreducible. In this case, choosing an isomorphism of $\Sigma$ with $\bP^1$ taking $p_\infty$ to $\infty$, the form $\omega$ is identified with $\lambda \: \d z$ for some $\lambda \in \bC^*$. In the 2-level case, $\Sigma$ consists of a root component $\Sigma_r$ and $p\ge2$ terminal components. The curve $\Sigma$ is thus specified by a configuration of points 
    
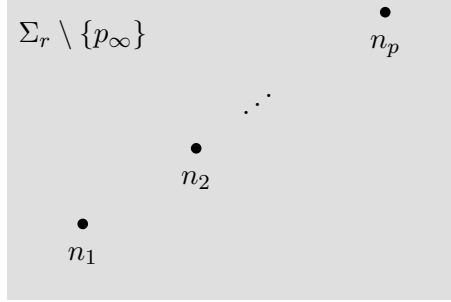
\[
    \begin{tikzpicture}
        \fill[gray!25] (0,0) rectangle (6,4);

        \fill (1,1) circle (2pt);
        \fill (2.5,2) circle (2pt);
        \fill (5,3.8) circle (2pt);

        \node[below=5pt] at (1,1) {$n_1$};
        \node[below=5pt] at (2.5,2) {$n_2$};
        \node[below=5pt] at (5,3.8) {$n_p$};
        
        \node at (1,3.5) {$\Sigma_r\setminus \{p_\infty\}$};
        \node at (3.3,2.7) {$\iddots$};
    \end{tikzpicture}
    \]
    in $\Sigma_r\setminus \{p_\infty\} \cong \bC$. Note that the form $\omega_r$ on $\Sigma_r$ makes this isomorphism canonical up to translation. The additional data consist of the forms attached to the components $\Sigma_1,\ldots,\Sigma_p$; however, these will not be needed in what follows. Consequently, it is best to think of $\Sigma$ as being specified by a configuration of points in $\Sigma_r\setminus \{p_\infty\}$, as visualized above.
\label{ex:twolevel}
\end{ex}

We consider a 2-level multi-scale line $\Sigma$ as in Example~\ref{ex:twolevel}. For any distinct pair of points $n_i$ and $n_j$, we define the period 
\[
    \Pi(n_i,n_j) := \int_{n_i}^{n_j}\omega_r\in \bC^*.
\]
To obtain smooth moduli spaces of multi-scale lines, their marked variants are defined in \cite{augmented} and considered up to  \emph{complex projective} and \emph{real oriented} equivalence. The important point is that up to real oriented equivalence, the $S^1$-valued angle function 
\[
    \mathfrak{p}(n_i,n_j) := \frac{\Pi(n_i,n_j)}{\lvert \Pi(n_i,n_j)\rvert}
\]
is well-defined. Thus, ultimately, the essential information encoded by a two-level multi-scale line $\Sigma$ is $\{\mathfrak{p}(n_i,n_j):1\le i <j\le n\}$. This allows us to state:

\begin{defn}
\label{D:multiscaledecomp}
    A \emph{$\le$2-level multi-scale decomposition} of $\cD$ is a $\le$2-level multi-scale line $\Sigma$ and a collection of full thick triangulated subcategories $\{\cD_{\le v}\}_{v\in V(\Sigma)_{\rm{term}}}$ of $\cD$ which generate $\cD$ and satisfy relations: \vspace{-2mm}
    \begin{enumerate}
        \item if $\Im (\mathfrak{p}(v,w)) > 0$ then $\Hom(\cD_{\le w},\cD_{\le v}) = 0$; and \vspace{-2mm} 
    \item if $\mathfrak{p}(v,w) = 1$ then $\cD_{\le v}\subset \cD_{\le w}$. \vspace{-2mm} 
    \end{enumerate}
    We write the multi-scale decomposition as $\cD = \langle \cD_\bullet\rangle_{\Sigma}$.
\end{defn}

\begin{ex}
\label{Ex:twolevelMSD}
    If $\Sigma$ is irreducible, then the multi-scale decomposition is trivial, consisting only of the category $\cD$ itself. Next, let $\Sigma$ have two levels. We illustrate two special cases:
    
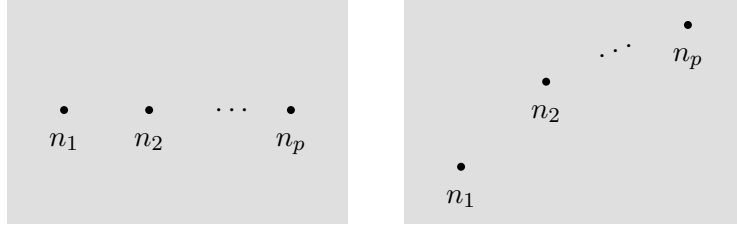
\[
    \begin{tikzpicture}[scale = .75]
        \fill[gray!25] (0,0) rectangle (6,4);

        \fill (1,2) circle (2pt);
        \fill (2.5,2) circle (2pt);
        \fill (5,2) circle (2pt);

        \node[below=5pt] at (1,2) {$n_1$};
        \node[below=5pt] at (2.5,2) {$n_2$};
        \node[below=5pt] at (5,2) {$n_p$};
        
        \node at (4,2) {$\cdots$};

        \fill[gray!25] (7,0) rectangle (13,4);

        \fill (8,1) circle (2pt);
        \fill (9.5,2.5) circle (2pt);
        \fill (12,3.5) circle (2pt);

        \node[below=5pt] at (8,1) {$n_1$};
        \node[below=5pt] at (9.5,2.5) {$n_2$};
        \node[below=5pt] at (12,3.5) {$n_p$};
        
        \node[rotate = -16] at (10.75,3.2) {$\iddots$};
    \end{tikzpicture}
    \]
    In the first case, corresponding to the configuration of points on the left, we obtain a filtration $\cD_{\le v_1}\subset \cdots \subset \cD_{\le v_p} = \cD$ of the category $\cD$. In the case on the right, we obtain a semiorthogonal decomposition $\cD = \langle \cD_{\le n_1},\ldots, \cD_{\le n_p}\rangle$. 
    
    In the general case, we obtain a semiorthogonal decomposition whose factors are indexed by the set of $\Im$-values of the configuration of points, and such that each factor is filtered according to the set of $\Re$-values of nodes with the given $\Im$-value.
\end{ex}

\begin{defn}
    Given a two-level multi-scale decomposition as in Example~\ref{Ex:twolevelMSD} and $v\in V(\Sigma)_{\rm{term}}$ we can define the \emph{associated graded category} $\gr_v(\cD) := \cD_{\le v}/\cD_{<v}$, where $\cD_{<v}$ is the thick subcategory of $\cD_{\le v}$ generated by all $\cD_{\le w}$ with $\mathfrak{p}(w,v) = 1$.
\end{defn}

Returning to Example~\ref{Ex:twolevelMSD}, in the situation corresponding to the configuration of points on the left, the associated graded categories are of the form $\gr_{v_i}(\cD) = \cD_{\le v_i}/\cD_{\le v_{i-1}}$. In the second situation, for each $v_i$ we have $\cD_{<v_i} = 0$ so that $\gr_{v_i}(\cD) = \cD_{\le v_i}$.

\begin{defn}
A \emph{($\le$2-level) augmented stability condition}, written $\sigma = \langle \cD_\bullet|\sigma_\bullet\rangle_\Sigma$, is a ($\le$2-level) multi-scale decomposition $\cD = \langle \cD_\bullet\rangle_\Sigma$ equipped with the data of stability conditions
$\sigma_\bullet = (\sigma_v) \in \prod_{v\in V(\Sigma)_{\rm{term}}} \Stab(\gr_v(\cD))/\bC$. 
\label{D:augmentedstab_special_case}
\end{defn}

Implicit in the definition is that each $\sigma_v$ satisfies the support property with respect to a subquotient lattice $\Lambda_v$ of $\Lambda$ --- see \cite{augmented}*{Defn. 3.36} for the details. We denote the set of all augmented stability conditions (without the $\le$2-level constraint) by $\Astab(\cD)$. The relevant properties of $\Astab(\cD)$ to the present work are:

\begin{thm}
[\cite{augmented}]
    Let $\cD$ be a triangulated category. \vspace{-2mm}
    \begin{enumerate}
        \item The space $\Astab(\cD)$ admits a Hausdorff topology such that $\Stab(\cD)/\bC$ embeds as an open subspace and such that the $\Aut(\cD)$-action extends continuously. \vspace{-2mm}
        \item Paths $\sigma_\bullet\colon[0,\infty)\to \Stab(\cD)/\bC$ which are \emph{strongly} quasi-convergent in the sense of \cite{augmented}*{Defn. 3.36} converge in $\Astab(\cD)$. \vspace{-2mm}
    \end{enumerate}
\end{thm}

We prove as Theorem~\ref{T:strongquasiconvergence} that subject to conditions on a path $\sigma_\bullet$ in $\Stab(\cD)$ that hold in  examples on surfaces, strong quasi-convergence is automatic.

\section{Augmented stability conditions in rank two} \label{sec:ranktwo}

We let $\cD$ be a Hom-finite $k$-linear triangulated category such that $\Knum(\cD)$ is free of rank $2$.\footnote{The results of this section carry over verbatim to the situation where the charge lattice $\Lambda$ is free of rank $2$.} We say that such a $\cD$ is of \emph{rank two}. As mentioned in Section~\ref{S:preliminaries}, we consider numerical stability conditions, i.e. those $\sigma = (Z,\cP)$ such that $Z$ factors through $v\colon \K_0(\cD)\twoheadrightarrow \Knum(\cD)$, satisfying the support property of \cite{KS08}. Then, by Bridgeland's deformation theorem \cites{Bayer_short,Br07}, if $\Stab(\cD)\ne \varnothing$ then it is a 2-dimensional complex manifold and $\Stab(\cD)/\bC$ is an open, possibly disconnected, Riemann surface. First, we need:

\begin{lem}
\label{L:rank2astab}
    If $\cD$ is a rank two category then $\sigma = \langle \cD_\bullet|\sigma_\bullet\rangle_\Sigma\in \partial\Astab(\cD)$ is specified by: \vspace{-2mm}
    \begin{enumerate}
        \item a two-level multi-scale line $\Sigma$ with two terminal components $v,w$, up to real-oriented isomorphism; and \vspace{-2mm}
        \item a thick triangulated subcategory $\cD_{\le v} \subsetneq \cD$ and 
        \begin{equation}
        \label{E:stabilityinproduct}
            (\sigma_v,\sigma_w) \in \Stab_{\Lambda_{\le v}}(\cD_{\le v})/\bC \times \Stab_{\Lambda/\Lambda_{\le v}^{\rm{sat}}}(\cD/\cD_{\le v})/\bC,
        \end{equation}
        where $\Lambda_{\le v}:= \im(\K_0(\cD_{\le v}) \to \Knum(\cD))$ is rank one. \vspace{-2mm}
    \end{enumerate}
    The real-oriented isomorphism class of $\Sigma$ is uniquely determined by $\mathfrak{p}(v,w) \in S^1$ and if $\cD_{\le v}\hookrightarrow \cD$ is not admissible then $\mathfrak{p}(v,w) = 1$.
\end{lem}

\begin{proof}
    By Definition~\ref{D:augmentedstab_special_case}, $\sigma$ has underlying multi-scale decomp\-osition $\langle \cD_\bullet\rangle_\Sigma$ as in Definition~\ref{D:multiscaledecomp}. The condition that the associated graded categories admit stability conditions forces $\# \Gamma(\Sigma)_{\rm{term}} \le \rank \Knum(\cD) = 2$. Definition~\ref{D:multiscaleline}(2) now forces $\Sigma$ to have two terminal components. 
    
    Example~\ref{Ex:twolevelMSD} describes this situation: $\langle \cD_\bullet\rangle_\Sigma$ is given by a pair of thick triangulated subcategories $\cD_{\le v}$ and $\cD_{\le w}$ of $\cD$. At least one of these categories, say $\cD_{\le v}$, must be a strict subcategory of $\cD$. If, say, $\Im (\mathfrak{p}(v,w)) > 0$, then $\cD = \langle \cD_{\le v},\cD_{\le w}\rangle$ so that $\cD_{\le v}\hookrightarrow \cD$ is admissible and $\cD_{\le w} = {}^\perp\cD_{\le v}$. If $\Im(\mathfrak{p}(v,w)) = 0$, assume without loss of generality that $\mathfrak{p}(v,w) = 1$. In that case, $\cD_{\le w} = \cD$. 

    To specify $\sigma$, we now need to give stability conditions modulo $\bC$ on $\cD_{\le v}$ and $\cD/\cD_{\le v}$, satisfying the support property with respect to $\Lambda_{\le v}:=\im(\K_0(\cD_{\le v}) \to \Knum(\cD))$ and $\Lambda/\Lambda_{\le v}^{\rm{sat}}$, respectively. This forces $\Lambda_{\le v}$ and $\Lambda/\Lambda_{\le v}^{\rm{sat}}$ to be of rank $1$. Note that in this case product space in \eqref{E:stabilityinproduct} is zero-dimensional.
\end{proof}

Next, the $G$-action on $\Stab(\cD)$ is generically free and the action of $\bC\subset G$ is free and holomorphic. In particular, the induced $G$-action gives a decomposition of $\Stab(\cD)/\bC$ into a union of open orbits $\mathscr{O}$, glued along degenerate one-dimensional orbits. These degenerate orbits, if they exist, are comprised of stability conditions with central charge $Z$ valued in a real one-dimensional subspace of $\bC$. 

So, consider an open orbit $\widetilde{\mathscr{O}}\subset \Stab(\cD)$ for the $G$-action and let $\mathscr{O}$ be its image in $\Stab(\cD)/\bC$. The group $G$ acts at the level of central charges through the homomorphism $\rho\colon G\to \GL_2^+(\bR)$ by $g\cdot Z := \rho(g)\circ Z$, where we regard $Z$ as valued in $\bR^2$. To parametrize the orbit, consider $\sigma \in \widetilde{\mathscr{O}}$ and a pair of stable objects $E$ and $F$ in $\cP_\sigma(0,1]$ where $\phi(E)<\phi(F)$. To remove the $\bC\subset G$ degree of freedom, we set $Z(F) = -1$ and $\phi(F) = 1$. This gives a lift $\mathscr{O}^+\subset \Stab(\cD)$ of $\mathscr{O}$. For all $\tau \in\mathscr{O}^+$, $E$ is $\tau$-stable with $\phi_\tau(E) \in (0,1)$ and there is a biholomorphism 
\[
    Z_{(-)}(E)\colon\mathscr{O}^+ \longrightarrow \bH
\]
given by $\tau \mapsto Z_\tau(E)$. More intrinsically, the induced coordinate on $\mathscr{O}$ is 
\begin{equation}
\label{E:identification}
    \overline{\tau}\longmapsto -Z_\tau(E)/Z_\tau(F)
\end{equation}
for any choice of $\tau \in \overline{\tau}$. Different choices of objects $E$ and $F$ result in a change of variables of $\bH\subset \bP^1$ by the standard action of $\PSL_2(\bZ)$.

It also follows from this argument that degenerate orbits are diffeomorphic to $\bR_{>0}$. Next, given an open orbit $\mathscr{O}$, we show how to compute its closure $\overline{\mathscr{O}}$ in $\Astab(\cD)$. 

\begin{lem}
\label{L:rank2_constantheart}
    Consider a $G$-orbit $\mathscr{O}\subset \Stab(\cD)/\bC$ and its lift to $\mathscr{O}^+\subset \Stab(\cD)$. For any $\sigma,\tau \in \mathscr{O}^+$, one has $\cP_\sigma(0,1] = \cP_\tau(0,1]$.
\end{lem}

\begin{proof}
    By definition of the $G$-action, the class $\cP_\sigma = \bigcup_{\phi \in \bR} \cP_\sigma(\phi)$ is independent of $\sigma \in \mathscr{O}^+$. Furthermore, for all $\sigma \in \mathscr{O}^+$ we have $\phi_\sigma(F) = 1$. Since elements of $G$ are orientation preserving, the $G$-action preserves $\bigcup_{\phi \in (0,1]}\cP_\sigma(\phi)$. On the other hand, $\cP_\sigma(0,1] = \langle \bigcup_{\phi \in (0,1]}\cP_\sigma(\phi)\rangle_{\rm{ext}}$ and the result follows.
\end{proof}

By Lemma~\ref{L:rank2_constantheart}, there is a heart $\cA_{\mathscr{O}}$ associated to the orbit $\mathscr{O}$, defined by $\cA_{\mathscr{O}} = \cP_\sigma(0,1]$ for any $\sigma \in \widetilde{\mathscr{O}}$. Furthermore, since the $G$-action does not affect (semi)stability, we say that an object $A$ is $\mathscr{O}$-(semi)stable if it is (semi)stable for some (equiv. any) element of $\mathscr{O}$. We write $\cA_{\mathscr{O}}^{\rm{ss}}$ for the set of $\mathscr{O}$-semistable objects of $\cA_{\mathscr{O}}$.

\begin{lem}
\label{L:rank2quasiconvergent}
    Any path $\sigma_\bullet$ in $\mathscr{O}$ which converges in $\Astab(\cD)$ is quasi-convergent.
\end{lem}

\begin{proof}
    If $\sigma_\bullet$ converges in $\Stab(X)/\bf{C}$, then it is tautologically quasi-convergent. Next, suppose that $\sigma_\bullet$ converges to a point $\sigma_\infty$ in $\partial \Astab(\cD)$. By definition of the $G$-action, Harder--Narasimhan filtrations are constant in $\mathscr{O}$; therefore, limit HN filtrations exist. 
    
    By the description of boundary points in Lemma~\ref{L:rank2astab}, there exist objects $A$ and $B$ of $\cA_{\mathscr{O}}$ such that $(\logZ_{\sigma_\bullet}(A),\logZ_{\sigma_\bullet}(B))$ converges in $\cA_2^{\bf{R}}$ to a boundary point. Since $\sigma_t \in \mathscr{O}$ for all $t$, $A$ is massless relative to $B$ in the limit (or vice versa). Consequently, $v(A)$ and $v(B)$ are $\bQ$-linearly independent and a calculation shows that for any pair of $\mathscr{O}$-semistable objects $E,F \in \cA_{\mathscr{O}}$, the $\lim_{t\to\infty} \mathfrak{n}_t(E,F)$ exists. Thus, $\sigma_\bullet$ is quasi-convergent.    
\end{proof}

Identifying $\mathscr{O}$ with $\bH$ as in \eqref{E:identification}, we give an explicit description of the boundary points of $\overline{\mathscr{O}} \subset \Astab(\cD)$, regarded as elements of $\overline{\bH} \subset \bP^1$.

Up to re-choosing $E$ and $F$, we may assume that $v(E)$ and $v(F)$ form a $\bZ$-basis of $\Knum(\cD)$. Thus, for any $A\in \cA_{\mathscr{O}}$ there is a unique decomposition:
\[
    v(A) = r(A)\cdot v(E) + d(A)\cdot v(F),\quad r(A)\in \bZ_{\ge 0}, d(A) \in \bZ.
\]
Indeed, since $0<\phi_\sigma(E) < 1$ for all $\sigma \in \mathscr{O}$ we have $\Im Z(E) > 0$. On the other hand, $\phi_\sigma(F) = 1$ so $\Im Z(F) = 0$. Thus, for $A\in \cA_{\mathscr{O}}$ we have $\Im Z(A) > 0$ if and only if $d > 0$. Next, we define the slope $\mu(A) = d(A)/r(A)\in \bQ\cup\{\infty\}$ and 
\[
    \mu(\cA_{\mathscr{O}}) := \left\{\mu(A):A\in \cA_{\mathscr{O}}^{\rm{ss}}\right\} \subset \bQ \cup \{\infty\}. 
\]
Next, we define $\cB_\mu$ to be the full triangulated subcategory of $\cD$ generated by the $\mathscr{O}$-semistable objects $A$ with $\mu(A) = \mu$.

\begin{prop}
\label{P:rank2_thicksubcat}
    Let $\cD$ be as above. The thick subcategories $\cD_{\le v}$ as in Lemma~\ref{L:rank2astab} that appear for $\sigma \in \overline{\mathscr{O}}\cap \partial\Astab(\cD)$ are $\cB_\mu$ for $\mu \in \mu(\cA_{\mathscr{O}})$.
\end{prop}

\begin{proof}
    It follows from Lemma~\ref{L:rank2quasiconvergent} that any point of $\overline{\mathscr{O}}\cap \partial \Astab(\cD)$ is the limit of a quasi-convergent path. Furthermore, for any such path there is a pair of objects $(A,B)$ such that $A$ is massless relative to $B$ in the limit. Consequently, $\cD_{\le v}$ is generated by the (limit) semistable objects for $\sigma_\bullet$ with numerical class in $(\bf{Z} \cdot v(A))^{\rm{sat}}\subset \Knum(\cD)$.

    It remains only to see that there are quasi-convergent paths realizing each $\cB_\mu$ for $\mu = \mu(A)$ for $A \in \cA_{\mathscr{O}}^{\rm{ss}}$. Consider a path in $\mathscr{O}$ such that $-Z_t(E)/Z_t(F) \to \mu$ in $\overline{\bH}$. We have 
    \[
        \frac{Z_t(A)}{Z_t(F)} = \frac{r(A)Z_t(E) + d(A)Z_t(F)}{Z_t(F)} = r(A)\cdot \frac{Z_t(E)}{Z_t(F)} + d(A) \to - r(A) \cdot \mu + d(A) = 0.
    \]
    This path realizes $\cB_\mu$ for $\mu \ne \infty$. When $\mu = \infty$, any path where $Z_t(F)/Z_t(E) \to 0$ will do, taking $(A,B) = (F,E)$.
\end{proof}

Next, we apply Proposition~\ref{P:rank2_thicksubcat} to smooth projective curves. We let $X$ denote a smooth projective curve and $\mathscr{O}^X$ the image in $\Stab(X)/\bC$ of the $G$-orbit of the the standard slope stability condition with heart $\Coh(X)$ and charge $Z(-) = -\deg(-) + \mathtt{i}\rank(-)$. We identify $\mathscr{O}^X$ with $\bH$ using $\tau \mapsto -Z_\tau(\cO_X)/Z_\tau(\cO_p)$ for some $p\in X$.

\begin{cor}
\label{C:caseofcurves}
    The analysis of curves splits into two cases. \vspace{-2mm} 
    \begin{enumerate}
        \item If $X = \bP^1$, then $\overline{\mathscr{O}}^{X} \cap \partial\Astab(\bP^1) \cong \bZ \cup \{\infty\}$; and \vspace{-2mm}
        \item if $g(X) \ge 1$, then $\overline{\mathscr{O}}^X := \Astab^\dagger(X)\cong  \bH \cup \bQ \cup \{\infty\}$ \vspace{-2mm}
    \end{enumerate}
    where in both cases $\mathscr{O}^X\cong \bH$ and boundary points $\mu \in\bQ \cup\{\infty\}$ correspond to a unique augmented stability condition with massless category $\cB_{\mu}$. The category $\cB_\infty$ is the category of torsion complexes on $X$.
\end{cor}

\begin{proof}
    This is a direct application of Proposition~\ref{P:rank2_thicksubcat}. Indeed, in the case of $\bP^1$ the stable objects in $\Coh(\bP^1)$ are the skyscraper sheaves $\cO_x$ for $x\in \bP^1$ and line bundles $\cO(n)$ for $n\in \bZ$. These correspond, respectively, to the point $\infty$ and the points $\bZ$ in $\partial \Astab(\bP^1)$. 
    
    When $g(X)\ge 1$, we have $\mathscr{O}^X = \Stab(X)/\bf{C}$ by \cite{Br07}*{Thm. 9.1} for $g(X) = 1$ and \cite{Macricurves} for $g(X) \ge 2$. It is well known that in these cases for any $\mu \in \bQ$ there is a stable bundle of slope $\mu$. The point $\infty$ again corresponds to skyscraper sheaves of points, which have slope $\infty$.
\end{proof}

\begin{rem}
    Consider the argument of Corollary~\ref{C:caseofcurves} for $\lambda$ with $\lambda \not \in \bZ \cup \{\infty\}$ and $X = \bP^1$. A quasi-convergent path $\sigma_\bullet$ in $\Coh(\bP^1)$ can be written down such that $Z_t(\cO_p)/Z_t(\cO_X)\to \lambda$. However, $\cB_{\lambda} = 0$ and $\lim_{t\to\infty} \sigma_t = \sigma^\lambda$ exists in $\Stab(\bf{P}^1)/\bf{C}$, as described by Okada \cite{Ok06}. 
    
    The point $\sigma^\lambda$ is a stability condition whose charge is valued in a real 1-dimensional subspace of $\bf{C}$. It is by passing through such degeneracy walls that one enters the glued regions of $\Stab(\bf{P}^1)$ as described in \cite{Ok06}. 
    
    One can ask the same question for $g(X) \ge 1$ and $\lambda \not \in \bQ \cup \{\infty\}$. In these cases, the limiting object is a pre-stability condition $\sigma^\lambda$ which does not satisfy the support property, specified by the heart $\Coh(X)$ and a charge $Z^\lambda$ valued in a real line. The point $\sigma^\lambda$ does not correspond to an element of $\Astab(X)$. See Figure~\ref{fig:astabcurves} for a visualization.
\end{rem}

\begin{cor}
\label{C:genusatleasttwoAstab}
    If $g(X)\ge 2$, then ${\Aut(\DCoh(X))}\backslash\Astab^\dagger(X)$ is isomorphic to $(\bf{H}\cup \bf{Q}\cup \{\infty\})/\bZ$ where $n\cdot z = z+n$ for all $n\in \bZ$.
\end{cor}

\begin{proof}
    By \cite{BondalOrlovreconstruction}*{Thm. 3.1}, since $g(X)\ge 2$ the autoequivalence group of $\DCoh(X)$ is $(\Pic(X)\times \bf{Z}) \rtimes \Aut(X)$. Here, $\bf{Z}$ acts by shifting and thus trivially on $\Stab(X)/\bf{C}$ and $\Astab^\dagger(X)$. Any geometric automorphism $\alpha \in \Aut(X)$ acts trivially on $\Knum(X)$ and thus trivially on the space of central charges. Therefore, $\Aut(X)$ acts trivially on $\Astab^\dagger(X)$ by connectedness. 
    
    Finally, we compute the action of $\Pic(X)$ on $\Stab(X)/\bf{C} \cong \bf{H}$, where we use the coordinate $z(\tau) = -Z_{\tau}(\cO_X)/Z_{\tau}(\cO_p)$. The action of $\Pic(X)$ factors through the N\'{e}ron--Severi group $\NS(X) = \bf{Z}$, which is generated by the class of a degree one line bundle $L$. Let $\sigma \in \Stab(X)/\bf{C}$ be given and let $\sigma' := \sigma \circ (-\otimes L)$. It follows that 
    \[
        z(\sigma') = -\frac{Z_{\sigma'}(\cO_X)}{Z_{\sigma'}(\cO_p)} = -\frac{Z_\sigma(\cO_X + \cO_p)}{Z_\sigma(\cO_p)} = z(\sigma) + 1. \qedhere
    \]
\end{proof}

The $\bZ$-action defined above is free on all points except $\{\infty\}$, which has stabilizer $\bf{Z}$, since the category of torsion sheaves is invariant under tensoring by line bundles. \medskip

Finally, we consider genus one curves $X$, which are perhaps the most interesting from the perspective of mirror symmetry. By \cite{UeharaAutoequivalences}*{\S 2.4}, there is a short exact sequence 
\[
    1 \to \Pic^0(X) \times \Aut(X) \times \bf{Z}[2] \to \Aut(\DCoh(X)) \to \SL_2(\bf{Z})\to 1.
\]

\begin{cor}
\label{C:astabofgenusone}
    Let $X$ be a curve of genus one. The quotient space $\Aut(\DCoh(X))\backslash\Astab^\dagger(X)$ is isomorphic to the orbifold projective line $\bf{P}(4,6)$.
\end{cor}

\begin{proof}
    By the proof of Proposition~\ref{C:genusatleasttwoAstab}, the $\Aut(\DCoh(X))$-action on $\Stab(X)/\bf{C}$ factors through $\Aut(\DCoh(X)) \twoheadrightarrow \SL_2(\bf{Z})$. In turn, the $\SL_2(\bZ)$-action induced on $\bH$ is the usual one by M\"obius transformations. Finally, as is well known the $j$-invariant gives an isomorphism $j\colon\SL_2(\bZ)\backslash\bf{H} \to \bf{A}^1(4,6)$. The boundary $\bf{Q}\cup \{\infty\}$ gives one additional orbit with stabilizer $\{\pm 1\}\subset \SL_2(\bf{Z})$. 
\end{proof}

As explained by Bridgeland \cite{BrSpacesofstability}, it is expected that for a Calabi--Yau variety $X$ the space $\Aut(\DCoh(X))\backslash\Stab(X)/\bf{C}$ should be related to the moduli space of complex structures $\cM_{\bf{C}}(X^\vee)$ on the mirror Calabi--Yau $X^\vee$. By work of Polishchuk--Zaslow \cite{PolishchukZaslow}, it is known that under homological mirror symmetry the mirror to $X$ is a symplectic $2$-torus, so that $\cM_{\bf{C}}(X^\vee)$ should be the corresponding moduli space of complex structures. In \cite{BrSpacesofstability}, it is explained that 
\[
    \Aut(\DCoh(X))\backslash \Stab(X)/\bf{C} \cong \SL_2(\bf{Z})\backslash \bf{H} \cong \bf{A}^1(4,6) \cong \cM_{1,1},
\]
i.e. the moduli of complex elliptic curves. Corollary~\ref{C:astabofgenusone} implies that ${\Aut(\DCoh(X))}\backslash\Astab(X)$ is isomorphic to the Deligne--Mumford compactification $\overline{\cM}_{1,1}$, with boundary point corresponding to the nodal elliptic curve. 

\begin{rem}
    There are other examples where the results of this section are applicable such as the Calabi--Yau completions of two-vertex quivers including the $A_2$ quiver and the Kronecker quiver. Compactifications of these spaces of stability conditions have already been studied by Bapat--Deopurkar--Licata \cite{BDLThurston}. Deopurkar has further studied the example of generic analytic K3 surfaces \cite{DeopurkarThurstonk3}, to which our results also apply.

    Proposition~\ref{P:rank2_thicksubcat} gives an explicit description of $\overline{\mathscr{O}} \cap \partial \Astab(\cD)$ for a given $G$-orbit $\mathscr{O}$ in these cases. In particular, when $\Stab(\cD)/\bC$ has finitely many $G$-orbits, this allows one to write 
    \begin{equation}
    \label{E:Astabdecomposition}
        \Astab(\cD) = \Stab(\cD)/\bC \sqcup \bigsqcup_{\mathscr{O} \text{ open}} (\overline{\mathscr{O}} \cap \partial\Astab(\cD)).
    \end{equation}
    This should enable a comparison of the partial compactification by augmented stability conditions to the Thurston compactification considered in \cites{BDLThurston,DeopurkarThurstonk3}. 

    Note, however, that in general the decomposition \eqref{E:Astabdecomposition} does not hold, even in rank 2. Indeed, in the case of $\DCoh(\bP^1)$ the admissible boundary points corresponding to semiorthogonal decomp\-ositions $\langle \cO(n),\cO(n+1)\rangle$ with $\Im(\mathfrak{p}(v,w)) \ne 0$ are not in the closure of any open $G$-orbit $\mathscr{O}$, of which there are infinitely many. See \cite{augmented}*{\S 6.3} for a complete description of $\Astab(\bP^1)$.
\end{rem}

\begin{figure}
    \centering
    \[
       \begin{tikzpicture}[scale=1]
    \fill[blue!8] (-6.25,0) rectangle (-1.25,3.5);

    \fill[blue!8] (1.25,0) rectangle (6.25,3.5);

    \node at (-5.75,0) {$\bullet$};
    \node at (-4.75,0) {$\bullet$};
    \node at (-3.75,0) {$\bullet$};
    \node at (-2.75,0) {$\bullet$};
    \node at (-1.75,0) {$\bullet$};

    \node at (-3.75,3.5) {$\bullet$};
    \node at (-3.75,3.7) {$\infty$};

    \node at (-5.75,-.3) {$-2$};
    \node at (-4.75,-.3) {$-1$};
    \node at (-3.75,-.3) {$0$};
    \node at (-2.75,-.3) {$1$};
    \node at (-1.75,-.3) {$2$};

    \node at (-3.75,2.2) {$\Stab(\bP^1)/\bC$};
    \node at (3.75,2.2) {$\Stab(X)/\bC$};

    \draw[dash pattern=on 1.5pt off 1pt, line width=1.5pt]
        (1.25,0) -- (6.25,0);

    \node at (3.75,-.3) {$\bQ$};
    \node at (3.75,3.7) {$\infty$};
    \node at (3.75,3.5) {$\bullet$};

    \draw[-, bend right=10, line width = 1]
    (-3.75,1.75) to (-4.25,.0);
    \node[rotate=-103] at (-4.125,.75) {$>$} ;
    \node at (-4.3, 1.3) {$\sigma_\bullet$};
\end{tikzpicture}
    \]
    \caption{We depict part of the partial compactification $\Astab(X)$ in the cases of $X = \bP^1$ and $g(X) \ge 1$, as described in Corollary~\ref{C:caseofcurves}. The black dots correspond to boundary points in $\Astab(X)$. On the left side, there is a portion of a quasi-convergent path $\sigma_\bullet$ which escapes into the glued region coming from the full exceptional collection $\langle \cO(-1),\cO\rangle$.}
    \label{fig:astabcurves}
\end{figure}
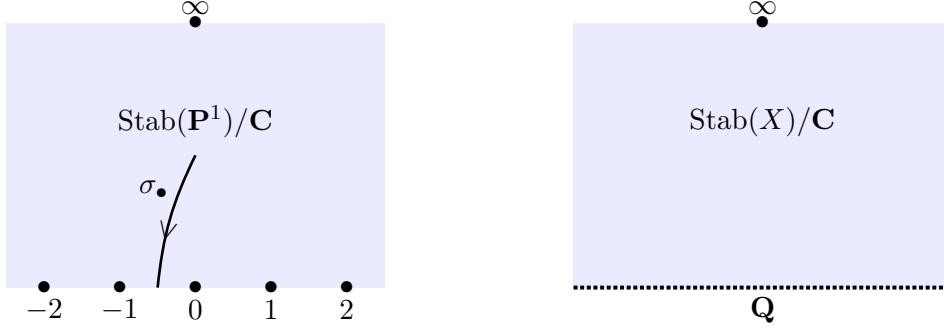

\section{Quasi-convergent paths from real variation of charge} \label{sec:qcpaths}

Let $\cD$ be a triangulated category and consider a path $\sigma_\bullet\colon [0,\infty) \to \Stab(\cD)$ of stability conditions. We denote by $\cP_{\sigma_\bullet}$ the full subcategory of limit semistable objects for $\sigma_\bullet$. Implicitly, stability conditions are defined with respect to a fixed map $v\colon\K_0(\cD)\twoheadrightarrow \Lambda$, for $\Lambda$ a finite rank lattice as in Definition~\ref{D:stabilitycondition}. In our main case of interest, where $\cD = \DCoh(X)$ for $X$ a smooth complex projective variety, we will typically take $\K_0(X) \twoheadrightarrow \Knum(X)$ as our map.

\subsection{Criteria for quasi-convergence}
We consider the following conditions:\vspace{-2mm}

\begin{enumerate}[label=(A\arabic*)]
    \item Denoting $\sigma_t = (Z_t,\cA_t)$ for all $t\in [0,\infty)$, we have $Z_t = A_t+ \mathtt{i}B$ for $A_t,B \in \Hom(\Lambda,\bR)$. In words, $\Im (Z_t)$ is constant in $t$. \vspace{-2mm}
    \label{axiom1}
    \item The germ of $A_t$ at infinity is analytic: that is, we can write $A_t(-) = \sum_{n \ge 0} a_n(-)\cdot t^{-n}$ for $a_n \in \Hom(\Lambda,\bR)$ for all $n\ge 0$. \vspace{-2mm}
    \label{axiom2}
\end{enumerate}

Note that by \ref{axiom2}, $Z_\infty := \lim_{t\to\infty} Z_t = a_0 + \mathtt{i}B$ exists. The following seems to be known to the experts cf. \cite{Bayer_short}*{p. 1606}, but we include a proof here for completeness.

\begin{lem}
\label{L:constantheart}
    If \ref{axiom1} holds then $\cA_t$ is constant in $t$.
\end{lem}

\begin{proof}
    Let $\cA := \cA_0$. We prove that the set of $t\in [0,\infty)$ for which $\cA_{t}:=\cP_{t}(0,1] = \cA$ is both open and closed. For openness, consider some $t_0$ for which $\cA_{t_0} = \cA$. For any $\epsilon>0$, there exists $\delta>0$ such that for all $t\in (t_0-\delta,t_0+\delta)$ we have $\cA \subset \cP_{\sigma_t}(-\epsilon, 1+\epsilon]$. Suppose $E$ is an object of $\cA$ such that $E \not \in \cP_{\sigma_t}(0,1]$. Thus, $\phi_{t}^+(E)> 1$ or $\phi_t^-(E)\le 0$.
    
    If $\phi_t^+(E)>1$, then by continuity of $\phi_t^+$ there exists $t_1 \in (t_0-\delta,t_0+\delta)$ such that $\phi_{t_1}^+(E) = 1$ and $\phi_{t_1+\gamma}^+(E) > 1$ for all $\gamma>0$ sufficiently small. Thus, there is a $\sigma_{t_1}$-HN factor $E'$ of $E$ with $\phi_{t_1}(E') = 1$. Consider the Jordan--H\"older factors $F_1,\ldots, F_p$ of $E'$ with respect to $\sigma_{t_1}$. By openness of stability, for all $\gamma>0$ sufficiently small, each $F_i$ is $\sigma_{t_1+\gamma}$-stable. However, it also follows that $F_i \in \cP_{t+\gamma}(\psi)$ for some $\psi \in (1,1+\epsilon]$ and $i\in \{1,\ldots, p\}$. Thus, $\Im Z_{t_1}(F_i) = 0$ while $\Im Z_{t_1+\gamma}(F_i) < 0$, contrary to \ref{axiom1}. Thus, $\cA\subseteq \cA_t$.

    The case where $\phi_t^-(E)\le 0$ is similar: there exists $t_1 \in (t_0-\delta,t_0+\delta)$ such that for all $\gamma>0$ sufficiently small we have $\phi_{t_1-\gamma}(E) > 0$, while $\phi_{t_1}^-(E) = 0$. Then, there exists a Jordan--H\"older factor $F$ of $E$ with respect to $\sigma_{t_1}$ such that $\phi_{t_1-\gamma}(F) > 0$ and $\phi_{t_1}(F) = 0$. This contradicts \ref{axiom1}. Thus, in either case we have $\cA \subseteq \cA_{t}$ so that $\cA = \cA_t$ and openness follows. 
    
    For closedness, consider a convergent sequence $(t_n)_{n\ge 1}\to t$ such that $\cA_{t_n} = \cA$ for all $n$. Then, for all non-zero objects $E$ of $\cA$, it follows that $\phi_{t_n}^+(E)\le 1$, and thus that $\phi_t^+(E)\le 1$. Next, by openness of stability, there exists $N>0$ such that $n\ge N$ implies that $\phi_{t_n}^-(E) = \phi_{t_n}(F)$ for some $\sigma_{t_n}$-stable $F$. Consequently, $\phi_{t_n}(F)>0$ implies $B(F)>0$ and thus that $\phi_{t}^-(E) = \phi_t(F) > 0$. So, $\cA \subseteq \cA_t$ and we have $\cA = \cA_t$.
\end{proof}

If the conclusion of Lemma~\ref{L:constantheart} holds, we say that $\sigma_\bullet$ has \emph{constant heart}. Note that by Lemma~\ref{L:rank2_constantheart}, the quasi-convergent paths considered in Section~\ref{sec:ranktwo} have constant heart. However, there \ref{axiom2} was not required.

\begin{lem}
\label{L:phase_one_independent}
    Given \ref{axiom1}, $\cP_{\sigma_t}(1) = \ker B \cap \cA$. In particular, the category of phase $n$ objects $\cP_{\sigma_t}(n)$ is independent of $t$ for all $n\in \bZ$.
\end{lem}

\begin{proof}
    Lemma~\ref{L:constantheart} associates a single heart of a bounded t-structure, $\cA := \cA_t$ for any $t\ge 0$, to $\sigma_\bullet$. The subcategory $\cP_{\sigma_t}(1)\subset \cA$ is characterized linear algebraically as the full subcategory of objects $E$ with $Z_t(E) \in \bR_{<0}$. In turn, this is equivalent to $B(E) = 0$ by \ref{axiom1}. 
\end{proof}

To de-emphasize the choice of $t$, we write $\cP_{\sigma_\bullet}(n)$ for the category of phase $n\in \bZ$ objects discussed in Lemma~\ref{L:phase_one_independent}. We call an object $E$ of $\cD$ \emph{massless} if $\lim_{t\to\infty} m_t(E) = 0$.

\begin{lem}
\label{L:vanishing_order}
    If $E$ is massless, then $H^n_{\cA}(E)\in \cP_{\sigma_\bullet}(1)$ for all $n\in \bZ$ and if $E$ is also limit semistable then $H^n_{\cA}(E) \in \cP_{\sigma_\bullet}(n)$ for some $n\in \bZ$. Further, for all massless $E$ there exist $\nu \ge 1$ and $C_E\in \bR_{>0}$ such that $m_t(E) \sim C_E\cdot t^{-\nu}$ as $t\to\infty$. 
\end{lem}

\begin{proof}
    Since $\cA = \cA_{\sigma_t}$ is constant by Lemma~\ref{L:constantheart}, we have $m_t(E) = \sum_{i \in \bZ} m_t(H^i_{\cA}(E))$, where the sum is finite by boundedness. So, if $E$ is massless we have that $H^i_{\cA}(E)$ is also for all $i$. However, if $\Im Z_t(H_{\cA}^i(E)) = B(H_{\cA}^i(E)) > 0$ then $Z_\infty(H^i_{\cA}(E)) \ne 0$ by \ref{axiom1}, so $B(H^i_{\cA}(E)) = 0$. Thus, $H^i_{\cA}(E)[-i] \in \cP_{\sigma_\bullet}(1-i)$ for each $i$ by Lemma~\ref{L:phase_one_independent}. If $E$ is also assumed to be limit semistable then $H^i_{\cA}(E) = 0$ for all but one $i$. Thus, $E\in \cP_{\sigma_\bullet}(n)$ for some $n\in \bZ$. 
    
    Finally, if $E$ is massless and limit semistable we have $m_t(E)\sim \lvert Z_t(E)\rvert$. Letting $\nu = \min\{n:a_n(E)\ne 0\}$ we have $\nu \ge 1$ and $Z_t(E) \sim a_\nu(E)\cdot t^{-\nu}$. The general case follows from the observation that each $H^i_{\cA}(E)$ is massless and limit semistable.
\end{proof}

\begin{defn}
    An object $E$ is of order $\nu \in \bZ_{\ge 0}$ if  $m_t(E)\sim C_E\cdot t^{-\nu}$ for some $C_E\in \bR_{>0}$. We let $\cA_{\star \nu}$ denote the full subcategory of $\cA$ generated by limit semistable objects of order $\star \nu$ where $\star\in \{\ge, >, = \}$. When $\star$ is $=$ we write $\cA_\nu$. By Lemma~\ref{L:masslesscharacterization}, massless objects are of order $\nu \ge 1$. 
\end{defn}

Consider $f\in \Hom(\Lambda,\bR)$. By a mild abuse of notation, we write $\ker f$ for the full subcategory of objects $E$ of $\cA$ such that $v(E) \in \ker f$.

\begin{lem}
\label{L:masslesscharacterization}
    For any $\nu \ge 1$, \vspace{-2mm}
    \begin{enumerate}
        \item $\cA_{\ge \nu} = \ker B \cap \bigcap_{m=0}^{\nu -1} \ker a_m$; and\vspace{-2mm}
        \item $\cA_{>\nu} = \ker B \cap \bigcap_{m=0}^\nu \ker a_m$; and  \vspace{-2mm}
        \item $\cA_{\nu} = \cA_{\ge \nu} \setminus \ker a_\nu$. \vspace{-2mm}
    \end{enumerate} 
    As a consequence, $\cA_{\ge \nu}$ and $\cA_{>\nu}$ are Serre subcategories of $\cA$.
\end{lem}

\begin{proof}
    Consider a limit semistable object $E$ in $\cA$ and write its $\sigma_0$-HN factors as $E_1,\ldots, E_n$ with $0<\phi(E_n)<\cdots<\phi(E_1)\le 1$. It follows that 
    \[
        Z_t(E) = \sum_{j=1}^n Z_t(E_j) = \sum_{j=1}^nA_t(E_j) + \mathtt{i}B(E_j).
    \]
    Therefore, $\Im Z_\infty(E)\ne 0$ unless $n=1$ and $\phi(E) = 1$. Thus, $Z_\infty(E) = 0$ implies $E\in \cP_{\sigma_\bullet}(1)$, hence that $E\in  \ker B \cap \cA$, by Lemma~\ref{L:phase_one_independent}. Furthermore, since in this case $m_t(E) = \lvert Z_t(E)\rvert$, we have $E\in \cA_{\ge \nu}$ if and only if $B(E) = a_0(E) =\cdots = a_{\nu-1}(E) = 0$. The analogous statement follows for $\cA_{>\nu}$.

    Next, we show that $\cA_{\ge \nu}$ and $\cA_{>\nu}$ are Serre subcategories of $\cA$ --- this amounts to checking the criteria of \cite{stacks-project}*{\href{https://stacks.math.columbia.edu/tag/02MP}{Tag 02MP}}. Given a subobject $E\subset F$, we have $B(F) = B(E) + B(F/E)$ so that $B(F) = 0$ if and only if $B(E) = 0$ and $B(F/E) = 0$ since $B\ge 0$ on $\cA$. Next, since $a_m \le 0$ on $\ker B \cap \cA$ the same argument shows that $\cA_{\ge\nu}$ and $\cA_{>\nu}$ are closed under subobjects and quotient objects. The same reasoning shows that $\cA_{\ge \nu}$ and $\cA_{>\nu}$ are closed under extensions.
\end{proof}

\begin{defn}[Asymptotic categories]
    For $\nu \ge 0$, we let \vspace{-2mm}
    \begin{enumerate}
        \item $\cD_{\star \nu}$ denote the full subcategory of $\cD$ whose objects $E$ satisfy $H^i_{\cA}(E) \in \cA_{\star \nu}$ for all $i\in \bZ$ where $\star \in \{\ge,>,=\}$.\vspace{-2mm}
    \end{enumerate}
    
    By \cite{quasiconvergence}*{Lem. 2.23},  the categories $\cD_{\ge \nu}$ and $\cD_{>\nu}$ are thick triangulated subcategories of $\cD$ for all $\nu \ge 1$ since $\cA_{\ge \nu}$ and $\cA_{>\nu}$ are Serre subcategories, respectively, by Lemma~\ref{L:masslesscharacterization}. Thus, it makes sense to define the associated graded category \vspace{-2mm}
    \begin{enumerate}[resume]
        \item $\cG_\nu := \cD_{\ge \nu}/\cD_{>\nu}$ for $\nu \ge 1$ and $\cG_0 := \cD/\cD_{>0}$. \vspace{-2mm}
    \end{enumerate}
    
    By definition, we have $\cD \supseteq \cD_{\ge 0} \supseteq \cD_{\ge 1}\supseteq\cdots \supseteq \cD_{\ge \nu}\supseteq \cdots$. We will deduce later on that $\cA_{\ge 0} = \cA$ and $\cD_{\ge 0} = \cD$, i.e. that every object has a filtration by limit semistable objects. Next, we replicate these definitions at the linear algebraic level. For $\nu \ge 1$, we set: \vspace{-2mm}
    \begin{enumerate}[resume]
        \item $\Lambda_{\ge \nu}:= \im(v|_{\rm{K}_0(\cD_{\ge \nu})})$ and define $\Lambda_{>\nu}$ analogously; and \vspace{-2mm}
        \item $\gr_\nu(\Lambda) := \Lambda_{\ge \nu}/\Lambda_{>\nu}^{\rm{sat}}$, \vspace{-2mm}
    \end{enumerate}
    where $\Lambda_{>\nu}^{\rm{sat}}\subset \Lambda$ denotes the saturation of the subgroup $\Lambda_{>\nu}\subset \Lambda$.
    \label{D:asymptotic_categories}
\end{defn}

Lemma~\ref{L:masslesscharacterization} has a linear algebraic analogue:

\begin{cor}
\label{C:linear_algebraic_filtration}
    For any $\nu \ge 1$, \vspace{-2mm}
    \begin{enumerate}
        \item $\Lambda_{\ge \nu} = \ker B \cap \bigcap_{m=0}^{\nu-1} \ker a_m$; \vspace{-2mm}
        \item $\Lambda_{> \nu} = \ker B \cap \bigcap_{m=0}^{\nu} \ker a_m$; and \vspace{-2mm}
        \item $\Lambda_\nu = \Lambda_{\ge \nu} \setminus \ker a_\nu$ \vspace{-2mm}
    \end{enumerate}
    Note that in the case $\nu = 0$, we have $\Lambda_{>0} = \ker Z_\infty$.
\end{cor}

\begin{proof}
    By construction, the category $\cA_{\ge \nu}$ is the heart of a bounded t-structure on $\cD_{\ge \nu}$. Conseq\-uently, $\K_0(\cD_{\ge \nu}) = \K_0(\cA_{\ge \nu})$ for each $\nu \ge 1$ and similarly $\K_0(\cD) = \K_0(\cA)$. Since the conditions defining $\cA_{\ge \nu}\subset \cA$ are defined purely in terms of vanishing of the additive functions $B,a_0,\ldots, a_{\nu-1}$ by Lemma~\ref{L:masslesscharacterization}, the claimed description of $\Lambda_{\ge \nu}$ follows. The other cases are analogous. 
\end{proof}

Next, we formulate the limiting support property for $\sigma_\bullet\colon[0,\infty)\to \Stab(\cD)$ satisfying \ref{axiom1} and \ref{axiom2}. To this end, note that the filtration
\begin{equation}
\label{E:filtrationlattice}
    \Lambda = \Lambda_{\ge 0} \supseteq \Lambda_{\ge 1}\supseteq \cdots \supseteq \Lambda_{\ge \nu}\supseteq \cdots.
\end{equation}
has finitely many steps $\nu_1,\ldots, \nu_p$ for which $\Lambda_{\ge \nu_i}\supsetneq \Lambda_{\ge \nu_{i+1}}^{\rm{sat}}$. For each $i=1,\ldots, p$, choose a section $\gr_{\nu_i}(\Lambda) \hookrightarrow \Lambda_{\ge \nu_i}$ of $\Lambda_{\ge \nu_i}\twoheadrightarrow \gr_{\nu_i}(\Lambda)$. Thus, we obtain an external direct sum decomposition 
\begin{equation}
\label{E:directsumdecomp}
    \Lambda \cong \bigoplus_{i=1}^p \gr_{\nu_i}(\Lambda).
\end{equation}
A norm $\lVert\:\cdot\:\rVert$ on $\Lambda_{\bR}$ is \emph{compatible} with \eqref{E:directsumdecomp} if it is a direct sum of norms on the spaces $\gr_{\nu_i}(\Lambda)$. \medskip

Using Corollary~\ref{C:linear_algebraic_filtration},  we are able to give a formulation of the limiting support property both in terms of norms and quadratic forms. We will consider the following two conditions:

\vspace{-2mm}
\begin{enumerate}[resume, label=(A\arabic*)]
    \item There exist $t_0 \ge 0$ and constants $\{C_\nu > 0\}_{\nu \ge 0}$ such that for all objects $E$ that are $\sigma_t$-stable for some $t\ge t_0$ we have
    \[
        \lvert Z_t(E)\rvert \ge C_\nu \cdot t^{-\nu} \cdot \lVert v(E)\rVert 
    \]
    for all $t\in [t_0,\infty]$, where $\nu = 0$ if $E$ is massive, and $\nu$ is the order of $E$ if it is massless. Here, $\lVert\:\cdot\:\rVert$ is some (equiv. any) norm on $\Lambda_{\bR}$ compatible with \eqref{E:directsumdecomp}. \label{axiom3}
    \vspace{-2mm}
    \item There exist $t_0 \in \bR_{>0}$ and quadratic forms $Q_\nu\colon\Lambda_{\ge\nu,\bR}^{\otimes 2}\to \bR$ for all $\nu \ge 0$ such that: \vspace{-2mm}
    \begin{itemize}
        \item $Q_0$ is negative definite on $\ker Z_\infty$ and $Q_0(E) \ge 0$ for all massive $\sigma_t$-stable objects $E$ for all $t\ge t_0$; and \vspace{-2mm}
        \item for all $\nu \ge 1$, $Q_\nu$ is negative definite on $\ker a_\nu \cap \Lambda_{\ge \nu,\bR}$ and satisfies $Q_\nu(E)\ge 0$ for all limit semistable objects $E$ of order $\nu$. \vspace{-2mm}
    \end{itemize} 
    \label{axiom4} 
    \vspace{-2mm}
\end{enumerate}

\begin{rem}
    Note that it suffices in \ref{axiom3} to quantify only over objects that are in $\cA$, since all $\sigma_t$-stable objects lie in $\cA[n]$ for some $n\in \bZ$. Also, note that both conditions are satisfied vacuously for all $\nu$ besides the indices $\nu_1,\ldots, \nu_p$ which give non-trivial steps of the filtration \eqref{E:filtrationlattice}.
\end{rem}

The following criterion is useful for verifying \ref{axiom3} for $\nu>0$. 

\begin{lem}
\label{L:qcpathsA4_massless}
    Let $\nu>0$ be given. If the monoid $M_\nu$ in $\Lambda$ generated by $\{v(E):E\in \cA_{\nu}\}$ is finitely generated then there exists $C_\nu>0$ as in \ref{axiom3}. 
\end{lem}

\begin{proof}
    Consider generators $v_1,\ldots, v_n$ for $M_\nu$. We have $Z_t(v_i) \in \bR_{<0}$ for all $i=1,\ldots, n$ and all $t$, since $\cA_\nu \subset \cP_{\sigma_\bullet}(1)$ by Lemma~\ref{L:phase_one_independent}. Thus, for each $i=1,\ldots, n$, we have $\lvert Z_t(v_i)\rvert \sim C_i \cdot t^{-\nu}$ as $t\to \infty$ for some $C_i \in \bR_{>0}$. So, for any non-zero $v = \sum n_i\cdot v_i\in M_\nu$ we have:
    \[
        \frac{\lvert Z_t(v)\rvert}{\lVert v\rVert} = \frac{\sum n_i \cdot \lvert Z_t(v_i)\rvert}{\lVert v\rVert} \ge \frac{\sum n_i \cdot \lvert Z_t(v_i)\rvert }{\sum n_i \cdot \lVert v_i\rVert} \ge \min_i \left\{\frac{\lvert Z_t(v_i)\rvert}{\lVert v_i\rVert}\right\}
    \]
    and the claim follows.
\end{proof}

\begin{lem}
\label{L:equivsupportproperty}
    If \ref{axiom1} and \ref{axiom2} hold then \ref{axiom3} is equivalent \ref{axiom4}.
\end{lem}

\begin{proof}
    Define $\Delta_t := \lVert Z_t - Z_\infty\rVert_{\rm{op}}$ and note that $\lim_{t\to\infty} \Delta_t = 0$. First, we show \ref{axiom4} $\Rightarrow$ \ref{axiom3}. We consider first the case of $\nu = 0$. Since $Q_0$ is negative definite on $\ker a_0 \cap \ker B = \ker Z_\infty$, there exists $C>0$ such that the function
    \begin{equation}
    \label{E:norm_qf}
        \lVert\:\cdot\:\rVert^2 := C^2\lvert Z_\infty(\:\cdot\:)\rvert^2 - Q(\:\cdot\:)
    \end{equation}
    defines a norm on $\Lambda_{\bR}$. Then, for all massive $\sigma_t$-stable objects for $t\ge t_0$ we have 
    \[
        \lvert Z_t(E)\rvert \ge \lvert Z_\infty(E)\rvert - \lvert Z_t(E) - Z_\infty(E)\rvert \ge (C-\Delta_t)\lVert v(E)\rVert 
    \]
    where we have applied \eqref{E:norm_qf} and the definition of the operator norm. We conclude by taking $t_0>0$ such that $\Delta_t<C/2$ for all $t \geq t_0$. Next, suppose that $\nu \ge 1$. In this case, an analogous argument shows that 
    \[
        \lVert\:\cdot\:\rVert^2_\nu := \lim_{t\to\infty} C_\nu^2 t^{2\nu}\lvert Z_t(\:\cdot\:)\rvert^2 - Q_\nu(\:\cdot\:) = C_\nu^2\lvert a_\nu(\:\cdot\:)\rvert^2 - Q_\nu(\:\cdot\:)
    \]
    works to establish \ref{axiom3}. 

    Next, we show that \ref{axiom3} $\Rightarrow$ \ref{axiom4}. If $E$ is massive and $\sigma_t$-stable for some $t\ge t_0$, we have 
    \[
        Q'(E) = \lvert Z_\infty(E)\rvert^2 - \frac{3C^2_0}{4}\lVert v(E)\rVert^2 \ge 0
    \]
    and negative definite on $\ker Z_\infty$. Negative definiteness is an open condition on $Z_\infty$ and thus the condition holds for all $t\ge t_0$ up to enlarging $t_0$. Finally, for $E$ limit semistable of order $\nu \ge 1$, the form 
    \[
        \lim_{t\to\infty} t^{2\nu}\lvert Z_t(E)\rvert^2 - \frac{3C_\nu^2}{4}\lVert v(E)\rVert_\nu^2 = \lvert a_\nu(E)\rvert^2 - \frac{3C_\nu^2}{4}\lVert v(E)\rVert_\nu^2 
    \]
    does the job.
\end{proof}

\begin{prop} 
\label{P:boundedHN}
Let $E$ be a non-zero object of $\mathcal{A}$ with no Harder--Narasimhan factor of phase $1$ at some (equiv. any) $t$. If \ref{axiom3} holds, then there exist $t_0,M_E\in \bR_{>0}$ such that for any $t\in [t_0,\infty]$ we have 
\begin{equation}
\label{E:normbound}
    \max_{i=1}^r\left\{\norm{v(G_i)},\norm{v(H_i)}\right\} \leq M_E,
\end{equation}
where $0 = G_0 \subsetneq G_1 \subsetneq \dots \subsetneq G_r = E$
is the $\sigma_t$-Harder--Narasimhan filtration of $E$ and $H_i = G_{i}/G_{i-1}$ are the HN factors for $i=1,\ldots,r$.
\end{prop}

\begin{proof}
Since $Z_t\to Z_\infty$, we have $\lVert Z_t - Z_\infty\rVert_{\rm{op}} \to 0$. Note also that $\lVert A_t - A_\infty\rVert_{\rm{op}}\le \lVert Z_t - Z_\infty\rVert_{\rm{op}}$. So, let $t_0$ be given such that $\lVert Z_t-Z_\infty\rVert_{\rm{op}}\le C_0/2$ for all $t\in [t_0,\infty]$. Since $\{Z_t(G_i)\}$ form the vertices of the Harder--Narasimhan polygon of $E$ as in \cite{Bayer_short}*{\S 3}, we have
\[ 
    A_t(G_i) \leq \max\{0, A_t(E)\} \leq \max \left\{0, A_{t_0}(E) + \frac{C_0}{2} \norm{v(E)}\right\}, 
\]
and so 
\begin{align*}
    A_{t_0}(G_i) & \le  A_t(G_i) + \lvert A_{t_0}(G_i) - A_t(G_i)\rvert \\
    & \le  \max \left\{0, A_{t_0}(E) + \frac{C}{2}\norm{v(E)}\right\} + \frac{C}{2} \norm{v(G_i)}.
\end{align*}
By \cite{Bayer_short}*{Lem. 5.5}, $G_i\subset E$ in $\cA$ implies $m_{t_0}(G_i) - A_{t_0}(G_i) \le m_{t_0}(E) - A_{t_0}(E)$ so
\[ 
    m_{t_0}(G_i) \leq m_{t_0}(E)-A_{t_0}(E) +\max \left\{0, A_{t_0}(E) + \frac{C_0}{2}\norm{v(E)}\right\} + \frac{C_0}{2} \norm{v(G_i)}. 
\]
Applying \ref{axiom3} and the triangle inequality as in Lemma~\ref{L:qcpathsA4_massless}, we obtain $m_{t_0}(G_i) \ge C_0\cdot \lVert v(G_i)\rVert$. Thus, 
\[ 
    \frac{1}{2}\cdot C_0 \norm{v(G_i)} \leq m_{t_0}(E)-A_{t_0}(E) +\max \left\{0, A_{t_0}(E) + \frac{C_0}{2}\norm{v(E)}\right\}
\]
and it follows that $\{\lVert v(G_i)\rVert\}$ is bounded. Next, $v(H_i) = v(G_i) - v(G_{i-1})$ and the triangle inequality imply that $\max \{\lVert v(H_i)\rVert\}\le 2 \max \{\lVert v(G_i)\rVert\} =: M_E$.
\end{proof}

\begin{cor}
\label{C:limitHN}
    If axioms \ref{axiom1}, \ref{axiom2}, and \ref{axiom3} hold, then all non-zero objects of $\cD$ have limit Harder--Narasimhan filtrations.
\end{cor}

\begin{proof}
    By Lemma~\ref{L:constantheart}, limit HN filtrations are a refinement of the $\cA$-filtration. Thus, we may consider a non-zero object $E$ of $\cA$, choose $t_0$ and $M_E$ as in Proposition~\ref{P:boundedHN}, and put  $V_E:=\{v(E) : E\in \cA \setminus \{0\} \text{ and } \lVert v(E)\rVert\le M_E\}$.
    
    Note that $\phi_t(E) = \Im \logZ_t(E) \in (0,1]$ is well-defined for all non-zero objects of $\cA$ and depends only on $v(E)\in \Lambda$. By convergence of $Z_t(E)$ and constancy of $\Im Z_t$ it follows that $\lim_{t\to\infty} \phi_t(E) = \phi_\infty(E)$ exists and is in $(0,1]$. Thus, we may unambiguously write $\phi_t(v)$ for $v = v(E)$ with $E\in \cA\setminus \{0\}$ for any $t\in [0,\infty]$. Next, consider 
    \[
        \Phi_\infty(E):= \left\{\phi_\infty(v): v\in V_E\right\}.
    \]
    By convergence of $Z_t$, for any given $\epsilon>0$ we can enlarge $t_0$ such that $t\ge t_0$ implies $\lvert \phi_t(v) - \phi_\infty(v) \rvert < \epsilon$ for all $v\in V_E$. By finiteness of $V_E$, we can choose $\epsilon>0$ small enough that the intervals $I_v(\epsilon):=(\phi_\infty(v)-\epsilon,\phi_\infty(v)+\epsilon)$ are disjoint for all pairs $v,w\in V_E$ with $\phi_\infty(v)\ne \phi_\infty(w)$. By Proposition~\ref{P:boundedHN}, we obtain a filtration $0 = E_0\subsetneq E_1\subsetneq \cdots \subsetneq E_n = E$ such that $H_i:=E_i/E_{i-1} \in \cP_{\sigma_t}(I_{v_i})$ with $\phi_\infty(v_1)>\cdots>\phi_\infty(v_n)$ for some $v_1,\ldots,v_n\in V_E$. 
    
    We claim that this is a limit HN filtration for $E$: each $H_i$ has a $\sigma_t$-HN filtration whose factors are part of the HN filtration of $E$. Consequently, any HN factor $H_{i,j}$ of $H_i$ lies in $I_v(\epsilon)$ for all $t\ge t_0$ and it follows that $\phi_t^{\pm}(H_i) \in I_v(\epsilon)$ for all $t\ge t_0$. Sending $\epsilon \to 0$, we see that $H_i$ is limit semistable with limiting phase $\phi_\infty(v_i)$ for all $i=1,\ldots, n$.
\end{proof}

\begin{cor}
    If \ref{axiom1}, \ref{axiom2}, and \ref{axiom3} hold then $m_\infty(E):=\lim_{t\to\infty}m_t(E)$ exists for all non-zero objects $E$ of $\cD$. 
\end{cor}

\begin{proof}
    By Corollary \ref{C:limitHN}, $E$ has a limit HN filtration with factors $H_1,\ldots, H_n$. Taking $t_0$ sufficiently large, we may assume that $m_t(E) = \sum_{i=1}^n m_t(H_i)$ for all $t\ge t_0$. The result then follows from the fact that $\lim_{t\to\infty} m_t(H_i) = \lvert Z_\infty(H_i)\rvert$ for any limit semistable object $H_i$.
\end{proof}

\begin{lem}
\label{L:normalizedlimit}
    If \ref{axiom2} holds then for any $E,F\in \cP_{\sigma_\bullet}$, $\lim_{t \to \infty} \mathfrak{n}_t(E, F)$ exists.
\end{lem}

\begin{proof}
Given a limit semistable object $E \in \cP_{\sigma_\bullet}$ with $E \in \cA$, then $\logZ_t(E) \approx \logZ_\infty(E)$ if $E$ is massive, and $\log Z_t(E) \approx -\nu \log t+\log\lvert a_\nu(E) \rvert$, if $E$ is massless of order $\nu$. We plug these asymptotics into \eqref{E:ratiolimit} to conclude.
\end{proof}

\begin{cor}
    If axioms \ref{axiom1}, \ref{axiom2}, and \ref{axiom3} hold, then $\sigma_\bullet$ is quasi-convergent with constant heart $\cA$. 
\end{cor}

\begin{proof}
    The constant heart property is by Lemma~\ref{L:constantheart}. The existence of limit Harder--Narasimhan filtrations is given by Corollary~\ref{C:limitHN}. The other condition in Definition~\ref{D:quasiconvergent} is Lemma~\ref{L:normalizedlimit}.
\end{proof}

Thus, by the results of \cite{quasiconvergence}*{\S2} we obtain a filtration of $\cD$, which we now describe. Since for any limit semistable object we have $\lim_{t\to\infty} \phi_t(E) = \phi_\infty(E)$ it follows that $E\sim^{\rm{i}}F$ for any $E,F\in \cP_{\sigma_\bullet}$. Thus, $\sigma_\bullet$ gives rise to the trivial semiorthogonal decomposition of $\cD$. However, it gives a filtration of $\cD$ by order of vanishing of mass, which we now explain. Let 
\[
    \cP_{\sigma_\bullet,\nu} := \left\{E\in \cP_{\sigma_\bullet} \cap \cD_{\nu}\right\}
\]
for $\nu\ge 0$, where $\cD_\nu$ is as in Definition~\ref{D:asymptotic_categories}. By Lemma~\ref{L:masslesscharacterization}, $\cP_{\sigma_{\bullet}} = \bigsqcup_{\nu\ge 0} \cP_{\sigma_\bullet,\nu}$. With this in mind, we define $\cP_{\sigma_\bullet,\ge \nu}:=\bigcup_{\mu \ge \nu} \cP_{\sigma_\bullet,\mu}$ and $\cP_{\sigma_\bullet,>\nu}$ analogously. Note that  $\cD_{\ge \nu} = \langle \cP_{\sigma_\bullet,\ge \nu}\rangle$ and similarly for $\cD_{>\nu}$ --- see Definition~\ref{D:asymptotic_categories}. The following is a strengthening of \cite{quasiconvergence}*{Thm. 2.38(1)} in the present context:

\begin{prop}
\label{P:quasiconvergentlimit}
    Suppose \ref{axiom1}, \ref{axiom2}, and \ref{axiom3} hold. The path $\sigma_\bullet$ is quasi-convergent and gives rise to a finite-step filtration 
    \[
        \cdots\subset \cD_{\ge \nu}\subset \cdots \subset \cD_{\ge 1} \subset \cD_{\ge 0} = \cD
    \]
    by thick triangulated subcategories and $\tau_\nu = (W_\nu,\cQ_\nu) \in \Stab_{\gr_\nu(\Lambda)}(\cG_\nu)$ for all $\cG_\nu \ne 0$ where: \vspace{-2mm}
    \begin{enumerate}
        \item $W_0 = \lim_{t\to\infty} Z_t$ and for all $\nu\ge 1$, $W_\nu$ is induced by $a_\nu|_{\Lambda_{\ge \nu}}$, and \vspace{-2mm}
        \item $\cQ_\nu(\phi)$ is the essential image of $\cP_{\sigma_\bullet,\ge \nu}(\phi)$ under the localization $\cD_{\ge \nu}\twoheadrightarrow \cG_\nu$. \vspace{-2mm}
    \end{enumerate}
    In particular, the essential image of $\cA_{\ge \nu}:= \cD_{\ge \nu}\cap \cA$ in $\cG_\nu$ is $\cQ_\nu(0,1]$.
\end{prop}

\begin{proof}
    As remarked above, the non-zero objects $E$ of $\cD_{\ge \nu}$ are exactly those whose limit HN factors lie in $\cP_{\sigma_\bullet,\ge \nu}$. Equivalently, such objects $E$ satisfy $m_t(E) \sim C_E\cdot t^{-\mu}$ for some $\mu\ge \nu$ and $C_E\in \bR_{>0}$. Thickness of $\cD_{\ge \nu}$ follows. The fact that $W_\nu$ descends to a map $\gr_\nu(\Lambda)\to \bC$ is a consequence of Corollary~\ref{C:linear_algebraic_filtration}.

    The claim that $\{\cQ_\nu(\phi)\}_{\phi \in \bR}$ is a slicing on $\cG_\nu$ is a consequence of the existence of limit HN filtrations --- see Lemma~\ref{C:limitHN}. To see that $\tau_0$ is a pre-stability condition, note that for all $E\in\cP_{\sigma_\bullet,0}(\phi)$
    \[
        W_0(E) = \lim_{t\to\infty}Z_t(E) \in \bR_{>0}\cdot \exp(\mathtt{i}\pi \phi)
    \]
    since $\phi:= \lim_{t\to\infty} \frac{1}{\pi}\Im \logZ_t(E)$, where the branch of the logarithm is chosen according to the value of $n \in \bZ$ such that $E\in \cA[n]$. For $\tau_\nu$ with $\nu\ge 1$, simply note that $\cQ_\nu(\phi)$ is non-empty only for $\phi\in \bZ$ and that for all non-zero $E\in \cQ_\nu(1)$, we have $W_\nu(E) = a_\nu(E) \in \bR_{<0}.$ It is immediate from this discussion that the essential image of $\cA_{\ge \nu}$ is $\cQ_\nu(0,1]$. 
    
    Finally, consider the decomposition \eqref{E:directsumdecomp} and a compatible norm $\lVert\:\cdot\:\rVert = \lVert\:\cdot\:\rVert_0 \oplus \bigoplus_{\nu > 0} \lVert\:\cdot\:\rVert_\nu$. Consider $t\ge 0$ and a non-zero object $E$ of $\cP_{\sigma_\bullet,0}$ with massive $\sigma_t$-HN factors $F_1,\ldots, F_k$. We have $\lVert v(E)\rVert_0 \le \sum_{i=1}^k \lVert v(F_i)\rVert_0 \le \sum_{i=1}^k \lVert v(F_i)\rVert$ and thus by \ref{axiom3} 
    \[
        \frac{m_t(E)}{\lVert v(E)\rVert_0}  \ge \frac{\sum_{i=1}^k\lvert Z_t(F_i)\rvert}{\sum_{i=1}^k \lVert v(F_i)\rVert} \ge \min \left\{\frac{\lvert Z_t(F_i)\rvert}{\lVert v(F_i)\rVert}\right\} \ge C_0>0.
    \]
    In particular, 
    \[
        \lim_{t\to\infty} \frac{\lvert Z_t(E)\rvert}{\lVert v(E)\rVert_0} = \lim_{t\to\infty} \frac{m_t(E)}{\lVert v(E)\rVert_0} =  \frac{\lvert W_0(E)\rvert}{\lVert v(E)\rVert_0} \ge C_0 > 0
    \] 
    and the support property for $\tau_0$ follows. The support property for $\tau_\nu$ follows from a similar but simpler argument when $\nu \ge 1$.
\end{proof}

\begin{rem}
    We note that if \ref{axiom3} is only required to hold for massive objects, then we obtain a quasi-convergent path which gives rise to a stability condition on the massive associated graded category, $\cG_0$.
\end{rem}

\subsection{Strong quasi-convergence}
Next, we show that quasi-convergent paths $\sigma_\bullet$ with constant heart converge to boundary points of $\Astab(\cD)$, as introduced in \cite{augmented}. 

\begin{setup}
\label{Setup:strongqc}
    Suppose that $\sigma_\bullet$ is a quasi-convergent path satisfying \ref{axiom1}, \ref{axiom2}, and \ref{axiom3}. By Proposition~\ref{P:quasiconvergentlimit}, there are integers $0=\nu_1<\cdots< \nu_p$ such that $\cG_\nu \ne 0$ if and only if $\nu \in \{\nu_1,\ldots, \nu_p\}$. Furthermore, there is an \emph{asymptotic indexing set} for $\sigma_\bullet$, i.e. a collection $\{E_1,\ldots, E_p\}\subset \cP_{\sigma_\bullet}$, such that $m_t(E_i)\sim C_i\cdot t^{-\nu_i}$ for $i=1,\ldots, p$. Finally, Proposition~\ref{P:quasiconvergentlimit} associates to $\sigma_\bullet$ a filtration $\{\cD_{\ge \nu_i}\}_{i=1}^p$ and $(\tau_{i})_{i=1}^p \in \prod_{i=1}^p \Stab_{\gr_{\nu_i}(\Lambda)}(\cG_{\nu_i})$.
\end{setup}

\begin{rem}
    Note that unlike in \cite{quasiconvergence}, the asymptotic indexing set can be identified with the set $\{\nu_1,\ldots, \nu_p\}$. Indeed, it is this observation that allows us to obtain genuine stability conditions $(\tau_{i})_{i=1}^p$ as opposed to their equivalence class in $\prod_{i=1}^p(\Stab_{\gr_{\nu_i}(\Lambda)}(\cG_{\nu_i})/\bC)$
\end{rem}

In \cite{augmented}*{Defn. 5.19}, a notion of \emph{strong quasi-convergence} is introduced. Strongly quasi-convergent paths converge with respect to the ``strong'' topology on $\Astab(\cD)$. Intuitively, this is the topology of uniform convergence quantified over $\cP_{\sigma_\bullet}$, as opposed to the ``weak'' topology of pointwise convergence. The strong topology recovers Bridgeland's topology on $\Stab(\cD)/\bC \subset \Astab(\cD)$. The rest of the section is dedicated to proving:

\begin{thm}
\label{T:strongquasiconvergence}
    If Setup~\ref{Setup:strongqc} holds, then $\sigma_\bullet$ is strongly quasi-convergent and thus converges in $\Astab(\cD)$ with respect to the strong topology.
\end{thm}

Strong quasi-convergence involves checking three additional conditions, which we recall and verify piecemeal in what follows. We work in Setup~\ref{Setup:strongqc} throughout. The first condition is:

\begin{lem}
\label{L:crossratio}
    For any quadruple $i,j,k,\ell \in \{1,\ldots, p\}$ of distinct indices, the limit of the cross ratios
    \[
        \lim_{t\to\infty} \frac{\logZ_t(E_i) - \logZ_t(E_j)}{\logZ_t(E_k) - \logZ_t(E_\ell)} 
    \]
    exists in $\bP^1$.
\end{lem}

\begin{proof}
    We use the asymptotics from the proof of Lemma \ref{L:normalizedlimit} to conclude.
\end{proof}

Next, we verify the second axiom of strong quasi-convergence. A $\tau_i$-stable object $E\in \cG_{\nu_i}$ is said to admit a \emph{stable lift} if there exists $t_0(E)$ and an object $\hat{E}$ of $\cD_{\ge \nu_i}$ such that $\hat{E} \cong E$ in $\cG_{\nu_i}$ and $t\ge t_0(E)$ implies that $\hat{E}$ is $\sigma_t$-stable. A strongly quasi-convergent path $\sigma_\bullet$ is required to admit \emph{stable lifts}, in that all $\tau_i$-stable objects admit stable lifts, for all $i=1,\ldots, p$.

\begin{prop}
\label{P:stablelifts}
    The path $\sigma_\bullet$ admits stable lifts.
\end{prop}

\begin{proof}
    For each $i=1,\ldots, p$, the $\tau_{i}$-stable objects are shifts of simple objects of $\cQ_{\nu_i}(\phi)$ for $\phi \in (0,1]$. First, consider a $\tau_{1}$-stable object $E$ in $\cQ_{\nu_1}(0,1]$. By definition of $\tau_{1}$-semistability and Proposition~\ref{P:quasiconvergentlimit}, it lifts to $\hat{E} \in \cP_{\sigma_\bullet,\nu_1}(\phi)$, for some $\phi \in (0,1]$.
    
    Without loss of generality, we may suppose that $\hat{E}$ has no subobject in $\cP_{\sigma_\bullet,\nu_1}(\phi)$: if such an object is massless it can be removed and if it is massive it contradicts stability of $E$. So, we assume that $\phi\ne 1$. 
    
    By Proposition~\ref{P:boundedHN}, there exists $t_0$ such that there are only finitely many possible $\sigma_t$-Jordan--H\"older filtrations of $\hat{E}$ for all $t\ge t_0$. So, let $\cJ$ denote the finite set of (iso. classes of) Jordan--H\"older factors of $\hat{E}$ for $\sigma_t$ for all $t\ge t_0$. Since $\hat{E}$ is limit semistable, $\lim_{t\to\infty} \phi_{\sigma_t}^+(\hat{E}) = \phi$ and if there exists $0\ne F\subset \hat{E}$ such that $\lim_{t\to\infty}\phi_{\sigma_t}(F) = \phi$, then $F = \hat{E}$ by stability of $E$. So, if $F\subsetneq \hat{E}$ as above is a Jordan--H\"older factor in $\cP_{\sigma_\bullet,0}$ at some $t\ge t_0$, then $\lim_{t\to\infty} \phi_t(F) = \phi_\infty(F) < \phi_\infty(E)$. By finiteness, we may define
    \[
        \epsilon := \max\left\{\phi_\infty(E) - \phi_\infty(F): F\in \cJ \text{ and } F\subsetneq \hat{E}\right\}.
    \]
    Up to increasing $t_0$, we may assume that $\lvert \phi_t(F) - \phi_\infty(F)\rvert < \epsilon/3$ for all $F\in \cJ$, by finiteness of $\cJ$. Next, suppose there exists $F\in \cJ$ such that $F\subsetneq \hat{E}$ and consider the associated JH filtration: $F = F_1 \subset \cdots \subset F_\ell = \hat{E}$ with associated graded objects $G_1,\ldots, G_\ell \in \cJ$. It follows that
    \begin{align*}
        \phi_\infty(\hat{E}) - \phi_\infty(G_i) & = \phi_\infty(E) - \phi_\infty(F) + \phi_\infty(F) - \phi_\infty(G_i) > \epsilon - \frac{2\epsilon}{3} = \frac{\epsilon}{3}.
    \end{align*}
    Thus, it follows that $\phi_\infty(G_i) < \phi_\infty(\hat{E})$ and that $\arg(Z_\infty(\hat{E})) = \arg(\sum_{i=1}^\ell Z_\infty(G_i)) < \phi_\infty(E)$, a contradiction. Thus, there exists $t_0$ such that $t\ge t_0$ implies that $\hat{E}$ is stable.

    Finally, we address the case $i=2,\ldots,p$. For this, simply observe that the $\tau_{i}$-Jordan--H\"older filtration of $E\in \cQ_{\tau_{i}}(1)$ is obtained from its (constant in $t$) $\sigma_t$-JH filtration, modulo factors lying in $\cD_{>\nu_i}$.
\end{proof}

Finally, we need to verify uniform convergence of $\phi^{\pm}_t$ and $m_t$ on limit semistable objects. This is the content of Propositions~\ref{P:phaseconvergence} and \ref{P:massconvergence}. To this end, we study $\hat{Z}_t(-):= Z_t(-)/\lVert v(-)\rVert$, which is determined by its values on the unit sphere $\bS:=\{v\in \Lambda_{\bR}:\lVert v\rVert = 1\}$. 

\begin{lem}
\label{L:uniformconvergencephase}
    We have $\hat{Z}_t(E) \to \hat{Z}_\infty(E)$ and $\phi_t(E) \to \phi_\infty(E)$ uniformly for all $E\in \cP_{\sigma_\bullet,0}$. 
\end{lem}

\begin{proof}
    Since $Z_t \to Z_\infty$ in $\Hom(\Lambda,\bC)$ by \ref{axiom2}, it follows that $\Delta_t := \lVert Z_t-Z_\infty\rVert_{\rm{op}}\to 0$. Consequently, for all $v\in \Lambda_{\bR}$ we have 
    \[
        \limsup_{t\to\infty} \left\{\frac{\lvert Z_t(v) - Z_\infty(v)\rvert}{\lVert v\rVert}:v\in \Lambda_{\bR} \setminus \{0\}\right\} = 0.
    \]
    Thus, $\hat{Z}_t(E) \to \hat{Z}_\infty(E)$ uniformly in non-zero $E\in \cP_{\sigma_\bullet,0}$ \emph{a fortiori}. Next, by \ref{axiom3}, there exists $C_0>0$ such that for all non-zero $E\in \cP_{\sigma_\bullet,0}$ we have $\lvert \hat{Z}_\infty(E)\rvert \ge C_0$. Consequently, for all non-zero $E\in \cP_{\sigma_\bullet,0}$ we have $v(E)/\lVert v(E)\rVert \in \bS\cap \hat{Z}_\infty^{-1}[C,\infty) =: \bS_C$, which is compact. Next, observe that for any $E \in \cA_0$ 
    \[
        \phi_t(E) = \frac{1}{\pi}\Im\log\hat{Z}_t(E) \in (0,1].
    \]
    We compute 
    \begin{equation} 
    \label{E:phidifference}
    \begin{split}
        \phi_t(E) - \phi_\infty(E) & = \frac{1}{\pi}\Im\left(\log \hat{Z}_t(E) - \log \hat{Z}_\infty(E)\right) = \frac{1}{\pi}\Im \log\left(1+ \frac{Z_t(E) - Z_\infty(E)}{Z_\infty(E)}\right),
    \end{split}
    \end{equation}
    where we use that $Z_\infty(E) \ne 0$. Next, 
    \begin{equation}
    \label{E:uniformconvoperatornorm}
        \left\lvert \frac{Z_t(E) - Z_\infty(E)}{Z_\infty(E)}\right \rvert = \left\lvert \frac{Z_t(E) - Z_\infty(E)}{\lVert v(E)\rVert}\cdot \frac{\lVert v(E)\rVert}{Z_\infty(E)}\right\rvert \le \Delta_t \cdot C_0^{-1}\to 0.
    \end{equation}
    Thus, applying the Taylor series expansion of $\log(1+x)$ to \eqref{E:phidifference}, we see that $\lvert \phi_t(E) - \phi_\infty(E)\rvert \to 0$ uniformly in $E\in \cP_{\sigma_\bullet,0}$. 
\end{proof} 

\begin{prop}
\label{P:phaseconvergence}
    In Setup~\ref{Setup:strongqc}, we have
    \[
        \limsup_{t\to\infty} \left\{\lvert \phi_t^+(E) - \phi_\infty^+(E)\rvert, \lvert \phi_t^-(E) - \phi_\infty^-(E)\rvert : E\in \bigcup_{i=1}^p\cD_{\nu_i} \setminus \{0\} \right\} = 0.
    \]
\end{prop}

\begin{proof}
    First, note that for $E \in \cD_{\nu_i}$ with $\nu_i>0$ we have $\phi_t^{\pm}(E) = \phi_\infty^{\pm}(E) = 1$ so there is nothing to check. Next, since $\sigma_\bullet$ has constant heart by Lemma~\ref{L:constantheart} and the limit HN filtrations constructed in Corollary~\ref{C:limitHN} are refinements of the $\cA$-filtration, it suffices to consider only objects $E\in \cA_0$. Choose such an $E$ and let $\cJ$ denote the set of $\sigma_t$-Jordan--H\"older factors of $E$ for all $t\ge 0$, all of which are massive. By definition, $\phi_t^+(E) =  \sup\{\phi_t(F):F\in \cJ\}$ and it follows that $\lim_{t\to\infty} \phi_t^+(E) = \phi_\infty^+(E) = \sup\{\phi_\infty(F):F\in \cJ\}$. It then follows that 
    \begin{equation} 
    \label{E:limsupbound}
        \limsup_{t\to \infty} \left\{\lvert \phi_t^+(E) - \phi_\infty^+(E)\rvert: E\in \cD_{\nu_1}\right\} \le \limsup_{t\to\infty} \left\{\lvert \phi_t(F) - \phi_\infty(F)\rvert: F\in \cS_{\ge t_0} \right\}, 
    \end{equation}
    where $\cS_{\ge t_0}$ is the set of objects that are $\sigma_t$-stable for some $t\ge t_0$. Thus, we must bound the right-hand side of \eqref{E:limsupbound}. Let $\epsilon>0$ be given. By Lemma~\ref{L:uniformconvergencephase}, there is a $t_0$ such that $t\ge t_0$ implies $\lvert \phi_t(H) - \phi_\infty(H)\rvert < \epsilon$ for all $H \in \cP_{\sigma_\bullet,0}$. By Corollary~\ref{C:limitHN}, any $F\in \cS_{\ge t_0}$ has a limit HN-filtration and we may assume that $F \in \cA$. Therefore, there is an inclusion $0\ne H^+\subset F$ in $\cA$ where $H^+$ is the limit HN factor of limiting phase $\phi_\infty^+(F)$. Since $F$ is $\sigma_t$-stable for some $t$, we have $\phi_t(H^+) \le \phi_t(F)$ and thus $\phi_\infty(H^+) < \phi_t(F)+\epsilon$ by definition of $t_0$. Dually, we can show that $\phi_\infty(H^-)>\phi_t(F) - \epsilon$, where $H^-$ is the limit HN factor of the least limiting phase. Thus,
    \[
         \phi_t(F) - \epsilon < \phi_\infty^-(F) \le \phi_\infty(F) \le \phi_\infty^+(F) < \phi_t(F) + \epsilon
     \]
     uniformly in $\cS_{\ge t_0}$ and \eqref{E:limsupbound} equals zero. The argument for $\phi_t^-$ is analogous. 
\end{proof}

Finally, we address the convergence of masses. 

\begin{prop}
\label{P:massconvergence}
    In Setup~\ref{Setup:strongqc}, we have
    \[
        \limsup_{t\to\infty} \left\{\left\lvert \log t^{\nu_i}\cdot m_t(E) - \log m_\infty(E)\right\rvert: E\in \bigcup_{i=1}^p\cD_{\nu_i} \setminus \{0\}\right\} = 0.
    \]  
\end{prop}

\begin{proof}
    First, we consider $0\ne E \in \cD_{\nu_1}$ and its limit HN factors $H_1,\ldots, H_n$. By Proposition~\ref{P:phaseconvergence}, the phase width of limit semistable objects tends uniformly to zero so we can choose $t_0$ such that $t\ge t_0$ implies that there exists a universal constant $C_t \in \bR_{\ge 1}$ such that 
    \[
        1\le \frac{m_t(F)}{\lvert Z_t(F)\rvert} \le C_t
    \]
    for all $F\in \cP_{\sigma_\bullet,0}$ and $\lim_{t\to\infty} C_t = 1$. Next, the triangle inequality for masses \cite{deformed}*{Thm. D} yields $m_t(E) \le \sum_{i=1}^p m_t(H_i)$ for any $t$. Therefore, one has
    \begin{equation}
    \label{E:logmass}
        \left\lvert \log \frac{m_t(E)}{m_\infty(E)}\right\rvert  \le \left\lvert \log \frac{C_t\cdot \sum_{i=1}^p \lvert Z_t(H_i)\rvert}{\sum_{i=1}^p\lvert Z_\infty (H_i)\rvert }\right\rvert
        \le \lvert \log C_t\rvert + \left\lvert \log \frac{\sum_{i=1}^p \lvert Z_t(H_i)\rvert}{\sum_{i=1}^p \lvert Z_\infty(H_i)\rvert}\right\rvert. 
    \end{equation}
    On the other hand, 
    \begin{equation}
    \label{E:ratiobound}
        \min_{i=1}^p \left\{\frac{\lvert Z_t(H_i)\rvert}{\lvert Z_\infty(H_i)\rvert}\right\} \le \frac{\sum_{i=1}^p \lvert Z_t(H_i)\rvert}{\sum_{i=1}^p \lvert Z_\infty(H_i)\rvert} \le \max_{i=1}^p \left\{\frac{\lvert Z_t(H_i)\rvert}{\lvert Z_\infty(H_i)\rvert}\right\}.
    \end{equation}
    The same computation as in \eqref{E:uniformconvoperatornorm} shows that $\log \lvert Z_t(H_i)\rvert /\lvert Z_\infty(H_i)\rvert$ tends to zero uniformly in $H \in \cP_{\sigma_\bullet,0}$ and consequently \eqref{E:logmass} tends to zero uniformly. 

    Next, we consider the case of $0\ne E \in \cD_{\nu_i}$ with $\nu_i\ne 0$. As usual, it suffices to consider non-zero objects in $\cA_{\nu_i}:= \cA \cap \cD_{\nu_i}$. By Lemma~\ref{L:masslesscharacterization}, such $E$ have constant phase $1$ and thus are semistable for all $t$. We have $m_t(E) = \lvert Z_t(E)\rvert = \lvert \sum_{\mu\ge \nu_i}a_\mu(E)\cdot t^{-\mu} \rvert$
    and consequently 
    \begin{equation}
    \label{E:massratio}
        t^{\nu_i}\cdot \frac{m_t(E)}{m_\infty(E)} = \left\lvert 1 + \sum_{\mu>\nu_i} \frac{a_\mu(E)}{a_{\nu_i}(E)} t^{\nu_i-\mu}\right\rvert. 
    \end{equation}
    Next, \ref{axiom3} implies that $a_{\nu_i}(E) \ge C_i>0$ for all $0\ne E\in \cP_{\sigma_\bullet,\nu_i}$. Thus, there exists a compact $K\subset \bS$ on which the $\bR^*$-invariant function 
    \[
        r_t(E) = \sum_{\mu > \nu_i}\frac{a_\mu(E)}{a_{\nu_i}(E)}\cdot t^{\nu_i-\mu}
    \]  
    is defined for all $E\in \cP_{\sigma_\bullet,0}$. The analyticity hypothesis \ref{axiom2}, combined with the fact that $a_{\nu_i}(E) \ge C_i$, implies that \eqref{E:massratio} converges uniformly to $1$ in $E \in \cP_{\sigma_\bullet,\nu_i}$.  
\end{proof}

\begin{proof}[Proof of Theorem~\ref{T:strongquasiconvergence}]
   The result now follows by combining Lemma~\ref{L:crossratio} and Propositions ~\ref{P:stablelifts}, \ref{P:phaseconvergence}, and \ref{P:massconvergence} to verify the conditions of strong quasi-convergence. Then, the result follows from \cite{augmented}*{Prop. 5.21}, which states that strongly quasi-convergent paths converge in $\Astab(\cD)$.  
\end{proof}

\begin{rem}
    In fact, Proposition~\ref{P:quasiconvergentlimit} shows that $\sigma_\bullet$ satisfying \ref{axiom1}, \ref{axiom2}, and \ref{axiom3} gives rise to a finer structure than an augmented stability condition. Indeed, we obtain from $\sigma_\bullet$ a multi-scale stability condition in the sense of \cite{BMSmulti-scale}. This is specified by the filtration $\{\cA_{\ge \nu}\}$ of $\cA$ described in Lemma~\ref{L:masslesscharacterization} and the induced charges $\{W_\nu\}$.

    The special feature of the quasi-convergent paths considered in the present work is the fact that phases of limit semistable objects (rather than certain differences of phases) converge, thereby defining the limiting heart $\cA$. Thus, constant heart quasi-convergent paths can be regarded as naturally belonging to $\Stab(\cD)$, rather than $\Stab(\cD)/\bC$.
\end{rem}

\section{K3 surfaces} \label{sec:K3}

The goal of this section is to produce examples of (strongly) quasi-convergent paths using the results of Section \ref{sec:qcpaths}. We will first discuss generalities on the support property on 2-Calabi--Yau categories, and then we specialize to the case of K3 surfaces, including a proof of Theorem~\ref{thm:intro_K3}.

\subsection{Generalities on 2-Calabi--Yau categories} \label{subsec:gen2cy}

Through this section we let $\cD$ be a proper $k$-linear 2-Calabi--Yau category, i.e. with Serre functor $S\cong [2]$. We assume that the Abelian group $\Lambda=\Knum(\cD)$ is finitely generated. Denote by $\lambda\colon \K(\cD) \to \Lambda$ the natural projection. We assume that the pairing
\[ (\lambda(E), \lambda(F)) := - \chi(E, F) \]
has signature $(2, \rank \Lambda-2)$.

\begin{lem} \label{lemma:gen2cy_roots}
If $E \in \cD$ has $\Hom^{< 0}(E, E)=0$ and $\Hom(E, E) \cong k$, then $(\lambda(E), \lambda(E)) = -2+2\ell$, with $\ell=\frac{1}{2}\dim \Ext^1(E, E)\geq 0$.
\end{lem}

\begin{proof}
Note that $\dim \Hom^{k}(E, E)=\dim \Hom^{2-k}(E, E)=0$ for $k\geq3$ by Serre duality. Thus, we get that $\Hom^{\geq 3}(E, E)=0$ and $\dim \Hom^2(E, E)=1$. This way, we get
\[ (\lambda(E), \lambda(E))=-\chi(E, E)=-2+\dim \Ext^1(E, E). \]
On the other hand, the Serre duality isomorphism $\Ext^1(E, E) \cong \Ext^1(E, E)^\vee$ gives a symplectic pairing $\Ext^1(E, E) \times \Ext^1(E, E) \to \Ext^2(E, E) \cong k$, which shows that $\Ext^1(E, E)$ is even dimensional.
\end{proof}

\begin{ex} \label{ex:gen2cy_K3}
Let $S$ be a K3 surface. Consider the Mukai lattice $\kn(S) := \bZ \oplus \NS(S) \oplus \bZ$, endowed with the pairing $((s_0, s_1, s_2), (s_0', s_1', s_2')) = s_1\cdot s_1'-s_0s_2'-s_2s_0'$. The Hirzebruch--Riemann--Roch theorem shows that the \emph{Mukai vector} 
\[
    s(E):=(\ch_0(E), \ch_1(E), \ch_0(E)+\ch_2(E))
\]
induces a lattice isomorphism $s\colon (\Knum(S), (-, -) )\to (\kn(S), (-, -))$.
\end{ex}

We denote by $\Roots(\Lambda)$ the collection of elements $v \in \Lambda$ such that $(a, a)=-2$. By Lemma \ref{lemma:gen2cy_roots}, if $E$ is $\sigma$-stable for some $\sigma \in \Stab(\Lambda)$, then either $v(E) \in \Roots(\Lambda)$ or $(v(E), v(E)) \geq 0$.

\begin{prop} \label{prop:gen2cy_axiom4}
Let $\sigma_t$ be a path in $\Stab(\cD)$ satisfying \ref{axiom1} and \ref{axiom2}, and such that $(-, -)|_{\ker Z_\infty}$ is negative definite. Then, \ref{axiom4} holds for $\nu=0$. 
\end{prop}

\begin{proof}
By assumption, $(-,-)|_{\ker Z_\infty}$ is negative definite. Thus, $\{a \in \ker Z_\infty\colon \norm{a} \leq -2\}$ is bounded, and so there are finitely many elements $\delta_1, \dots, \delta_r \in \ker Z_\infty \cap \Roots(\Lambda)$. 

Note that $\norm{v}^2:= A\lvert Z_\infty(v)\rvert^2 - (v, v)$ defines a norm on $\Lambda_\bR$ for some $A \gg 0$, as $(-,-)$ is negative definite on $\ker Z_\infty$. With this in mind, we claim that
\begin{equation} 
C = \inf \{ \lvert Z_\infty(\delta)\rvert : \delta \in \Roots(\Lambda) \setminus \ker Z_\infty \} >0.
\end{equation}
In fact, if $\delta \in \Lambda$ has $\lvert Z_\infty(v)\rvert \leq 1$, then $\lVert \delta\rVert \leq \sqrt{A+2}$. Thus, there are finitely many values of $\delta \in \Lambda$ with $\lvert Z(\delta) \rvert \leq 1$. 

We consider the quadratic form $Q(v) = (v, v)+\frac{2}{C^2}\lvert Z_\infty(v) \rvert^2$. By construction, $Q$ is negative definite on $\ker Z_\infty$, hence since $Z_t \to Z_{\infty}$ by \ref{axiom2} there is some $t_0$ such that $Q$ is negative definite on $\ker Z_t$ for $t \geq t_0$. 

We claim that $Q$ satisfies \ref{axiom4} for $\nu=0$. In fact, if $E$ is $\sigma_t$-stable for some $t\geq 0$ and massive, then by Lemma \ref{lemma:gen2cy_roots} there are two possibilities: \vspace{-2mm}
\begin{itemize}
    \item if $(v(E), v(E))=-2$, then $v(E) \in \Roots(\Lambda) \setminus \ker Z_\infty$ and we have:
    \[ 
        Q(v(E)) = -2 + \frac{2}{C^2}\lvert Z(E)\rvert^2 \geq 0. 
    \]\vspace{-8mm}
    \item If $(v(E), v(E)) \geq 0$, then $Q(E)\geq 0$ is automatic. \vspace{-2mm} \qedhere
\end{itemize}
\end{proof}

\subsection{Paths on K3 surfaces}
We will apply the results of Subsection \ref{subsec:gen2cy} to construct quasi-convergent paths on the stability manifold of K3 surfaces. Given a K3 surface $S$ we use the Mukai vector $s$ from Example \ref{ex:gen2cy_K3}. To simplify notation, we write $\Roots(S)$ for $\Roots(\kn(S))$.
\begin{thm} \label{thm:pathsK3_main} \label{pathnegcurves}
Let $S$ be a projective K3 surface.\vspace{-2mm}
\begin{enumerate}
    \item There exists a strongly quasi-convergent path $\sigma_\bullet$ in $\Stab(S)$ with associated filtration $0\subset \langle \cO_S \rangle \subset \DCoh(S)$. \label{pathOS}\vspace{-2mm}
    \item If $S$ has Picard rank 1, and $G \in \Coh(S)$ is a spherical vector bundle, then there exists a strongly quasi-convergent path $\sigma_\bullet$ in $\Stab(S)$ with associated filtration $0\subset \langle G \rangle \subset \DCoh(S)$. \label{pathsphvb}\vspace{-2mm}
    \item Suppose that there is a birational map $f\colon S \to T$ to a normal, projective surface $T$ with ADE singularities. Let $C_1, \dots, C_n$ be the exceptional curves, and pick a decomposition $\{1, \dots, n\} = I_1 \sqcup \dots \sqcup I_r$. Then there exists a strongly quasi-convergent path $\sigma_\bullet$ in $\Stab(S)$ with associated filtration
    \begin{equation}
    \label{E:filtration_thm}
        0\subset \langle \cO_{C_i}(-1) \rangle_{i \in I_r} \subset \langle \cO_{C_i}(-1) \rangle_{i \in I_r \cup I_{r-1}} \subset \dots \subset \langle \cO_{C_i}(-1) \rangle_{i \in I_r \cup \dots \cup I_1} \subset \DCoh(S). 
    \end{equation}
    \vspace{-2mm}
\end{enumerate}
\end{thm}

We begin by recalling the stability conditions on $\DCoh(S)$ constructed by Bridgeland \cite{BridgelandK3}*{\textsection 6}. Choose $\beta, \omega \in \NS(S)_\bQ$, with $\omega$ ample. Denote by $\mu_\omega\colon \Coh(S) \to \bR \cup \{+\infty\}$ the slope
\[
    \mu_\omega(E)=
        \begin{cases}
            \frac{\ch_1(E)\cdot\omega}{\ch_0(E)}, & \ch_0(E)>0 \\ 
            +\infty & \text{otherwise.} 
        \end{cases} 
\]
We set $\cT_{\beta, \omega}$ to be the smallest extension-closed full subcategory of $\Coh(S)$ containing all torsion sheaves and $\mu_\omega$-semistable torsion-free sheaves with $\mu_\omega>\beta\cdot \omega$. We let $\cF_{\beta, \omega}$ denote the smallest extension-closed subcategory of $\Coh(S)$ containing $\mu_\omega$-semistable torsion-free sheaves with $\mu_\omega\leq \beta\cdot\omega$. By \cite{BridgelandK3}*{Lem. 6.1}, these two subcategories form a torsion pair on $\Coh(S)$, and so 
\begin{equation} \label{eq:pathsK3_heart}
\cA_{\beta, \omega} := \{E \in \DCoh(S): H^{-1}(E) \in \cF_{\beta,\omega}, H^0(E) \in \cT_{\beta,\omega}, H^i(E)=0 \text{ otherwise} \}
\end{equation}
is the heart of a bounded t-structure on $\DCoh(S)$. We consider the central charge
\begin{align*}
    Z_{\beta,\omega}^\Bridgeland(E) &:= ( \exp(\beta+\mathtt{i}\omega), v(E) ) \\
    &= -s_2(E) + \beta\cdot s_1(E) + \frac{1}{2}s_0(E)(\omega^2-\beta^2) + \mathtt{i}\omega\cdot (s_1(E)-\beta s_0(E)) \\
    &= -\ch_2(E) + \beta \cdot \ch_1(E) + \frac{1}{2}\ch_0(E)(\omega^2-\beta^2-2) + \mathtt{i}\omega\cdot(\ch_1(E)-\beta \ch_0(E)),
\end{align*}
where $s(E)=(s_0,s_1,s_2)$. To state the main result, recall that we identify $\Knum(S)$ with the Mukai lattice $\cN(S)$. Let us write $\Roots^+(S) =\{(s_0, s_1, s_2)\in \Roots(S): s_0>0\} \subset \Roots(S)$. Let $V(S) \subset \Stab(S)$ be the set of \emph{normalized geometric} stability conditions:
\[ V(S) = \{ \sigma \in \Stab(S): \text{ for all } x \in S, \ \ko_x\text{ is stable of phase 1 and }Z(\ko_x)=-1\}. \]

\begin{thm}[\cite{BridgelandK3}*{\textsection 11}] \label{thm:pathsK3_geochamber}
The forgetful morphism $\Zforg\colon \Stab(S) \to \Hom(\Lambda, \bC)$ restricts to an embedding $V(S) \hookrightarrow \Hom(\Lambda, \bC)$. The image of $V(S)$ via $\Zforg$ is equal to the image of the set
\[ \mathscr{L}(S) = \{ (\beta, \omega) \in \NS(S)_\bR\times \Amp(S)_\bR : \forall v \in \Roots^+(S), Z_{\beta,\omega}^\Bridgeland(v)\notin \bR_{\leq 0} \} \]
under the map $(\beta,\omega) \mapsto Z_{\beta,\omega}^\Bridgeland$.
\end{thm}

We denote by $\sigma^\Bridgeland_{\beta,\omega}$ the element of $V(S)$ with central charge $Z_{\beta, \omega}^\Bridgeland$, if it exists. We point out that $\mathscr{L}(S)$ contains all pairs $(\beta, \omega)$ with $\omega$ ample and $\omega^2>2$, cf. \cite{BridgelandK3}*{Lem. 6.2}.

\begin{rem}[\cite{BridgelandK3}*{\textsection 10}] \label{remark:pathsK3_stableph1}
Given $(\beta,\omega) \in \mathscr{L}(S)$, the $\sigma^\Bridgeland_{\beta,\omega}$-stable objects of phase 1 are (i) skyscraper sheaves of points, and (ii) shifts $F[1]$ of $\omega$-stable vector bundles $F$ with $\ch_1(F).\omega=\beta.\omega$.
\end{rem}

To prove Theorem \ref{thm:pathsK3_main}, we construct paths in $V(S)$ satisfying \ref{axiom1} and \ref{axiom2}, with $(-,-)|_{Z_\infty}$ negative definite. Then, Proposition \ref{prop:gen2cy_axiom4} will guarantee that \ref{axiom4} holds, and so we can apply Theorem \ref{T:strongquasiconvergence}.

\begin{proof}[Proof of Theorem \ref{thm:pathsK3_main}(1)]
Choose an ample divisor $H$ on $S$, and write $\lambda=\sqrt{2/H^2}$. We define a path\footnote{Note that the path $\sigma_\bullet$ approaches $\lambda$, not $+\infty$.} $\sigma_\bullet=(Z_\bullet, \cA)\colon (\lambda, +\infty) \to \Stab(S)$ as follows. Set $\cA := \cA_{0, H}$, and
\[ 
    Z_t(E) = -s_2(E) + \frac{1}{2}s_0(E)t^2H^2 + \mathtt{i}H.s_1(E), 
\]
where $s(E)=(s_0(E),s_1(E),s_2(E))$. Note that $Z_t = \Re Z_{0,tH} + \frac{1}{t} \Im Z_{0,tH}$. This way, we have that $\sigma_t=g_t \cdot \sigma_{0, tH}^\Bridgeland$ for $g_t \in G$, and so $\sigma_t$ is a stability condition when $(0, tH) \in \mathscr{L}(S)$. Here $(tH)^2 =t^2H^2>2$ for $t>\lambda$, so $\sigma_t$ is a stability condition for all $t>\lambda$.

We claim that $\{\sigma_t\}_{t \to \lambda^+}$ is strongly quasi-convergent. To prove this, note that \ref{axiom1} and \ref{axiom2} follow by construction. Note that
\[ 
    Z_\lambda(E) = -s_2+s_0 + \mathtt{i} H.s_1, 
\]
so that $\ker Z_\lambda=\{(s_0, s_1, s_0): s_1.H=0\}$. Here
\[ ((s_0,s_1,s_0),(s_0,s_1,s_0)) = s_1^2-2s_0^2\leq0 \]
by the Hodge index theorem. Equality holds only for $s_0,s_1=0$, so $(-,-)|_{\ker Z_\Lambda}$ is negative definite, and \ref{axiom4} holds for $\nu=0$ by Proposition \ref{prop:gen2cy_axiom4}.

Next, we describe the massless objects of $\sigma_\bullet$, i.e. those objects $E \in \cA$ with $m_t(E)\to 0$ as $t \to \lambda^+$. By Lemma \ref{L:vanishing_order}, all such objects have constant phase 1. By Remark \ref{remark:pathsK3_stableph1}, the stable objects $E \in \cA$ of phase 1 are (i) skyscraper sheaves of points, or (ii) $E=F[1]$, for $F$ an $H$-stable vector bundle with $H.\ch_1(E)=0$. The former have $Z_\lambda(\ko_x)=-1$. For the vector bundles, using that $s_1(E)^2-2s_0(E)s_2(E)\geq-2$ we can write
\[ Z_\lambda(F[1]) \leq \frac{\ch_1(F)^2}{2\ch_0(F)}+\ch_0(F)-\frac{1}{\ch_0(F)}. \]
The first term is $\leq 0$ by the Hodge index theorem (with equality only for $\ch_1(F)=0$), while the sum of the last two is $\leq 0$ with equality only for $\ch_0(F)=1$. Thus, if $Z_\lambda(F[1])=0$, we must have $\ch_0(F)=1$, $\ch_1(F)=0$, and $s_1^2-2s_0s_2=-2$, i.e. $s_2(F)=1$. From here it follows that $E=\ko_S[1]$, and so $\ko_S$ is the only massless simple object of phase 1. 

To conclude, note that \ref{axiom4} holds for all $\nu>0$ by Lemma~\ref{L:qcpathsA4_massless}. We get that $\sigma_\bullet$ is strongly quasi-convergent by Theorem \ref{T:strongquasiconvergence}, with filtration $0\subset \langle \cO_S\rangle \subset \DCoh(S)$.
\end{proof}

For the second result, we recall the following fact about spherical vector bundles on K3 surfaces of Picard rank 1.

\begin{prop} \label{prop:pathsK3_sphericalvb}
Let $S$ be a K3 surface of Picard rank 1, with $\Pic(S)=\bZ H$ and $H$ ample. Let $(r, dH, \ell) \in \Roots^+(S)$ be an element with $r>0$. Then (i) $\gcd(r, d)=1$, (ii) the value $\ell$ is determined by $(r, d)$, (iii) there exists a unique spherical sheaf $F$ with $v(F)=(r, dH, \ell)$, and (iv) $F$ is $H$-stable and locally free.
\end{prop}

\begin{proof}
    By assumption, $(dH)^2 - 2r\ell =-2$, i.e. $d^2\frac{H^2}{2}-r\ell=-1$. Parts (i) and (ii) follow immediately from this computation. Parts (iii) and (iv) are discussed in \cite{HuybrechtsK3}*{Rem. 10.3.3}.
\end{proof}

\begin{proof}[Proof of Theorem \ref{thm:pathsK3_main}(2)]
Write $\Pic(S)=\bZ H$, where $H$ is ample. Let $s(G)=(r, dH, \ell)$ be the Mukai vector of the given spherical vector bundle $G$.

We consider the heart $\cA=\cA_{bH, H}$, where $b=d/r$, and the central charge
\begin{equation} \label{eq:pathsK3_cchargesph}
Z_t(E) = -s_2(E) +bH.s_1(E) + \frac{1}{2}s_0(E)(t^2H^2-b^2H^2) + \mathtt{i}H.(s_1(E)-bH s_0(E)).
\end{equation}
As before, $\sigma_t:=(Z_t, \cA_t)$ is in the $G$-orbit of $\sigma^\Bridgeland_{bH, tH}$, and so it is a Bridgeland stability condition once $(bH, tH) \in \mathscr{L}(S)$.

We claim that $(bH, tH) \in \mathscr{L}(S)$ for all $t>\lambda =:\sqrt{2/(r^2H^2)}$. By Theorem \ref{thm:pathsK3_geochamber}, we need to check that $Z_{bH, tH}(r',d'H,\ell') \notin \bR_{\leq 0}$ for all $(r', d'H, \ell') \in \Roots^+(S)$. We argue by contradiction: if $Z_{bH, tH}(r', d'H, \ell') \in \bR_{\leq 0}$, then 
\[ 0 = \Im Z_t(r',d'H, \ell')=H^2(d'-br'). \]
Thus, we must have $d'/r'=b=d/r$. But $\gcd(d',r')=\gcd(d,r)=1$ and $r,r'>0$, so $d=d'$ and $r=r'$. Thus, $\ell=\ell'$ as well by Proposition \ref{prop:pathsK3_sphericalvb}. But
\[ \Re Z_t(r, dH, \ell)=-\ell+bd H^2 + \frac{r}{2}(t^2H^2-b^2H^2) > -\ell + bdH^2+\frac{r}{2}\left( \frac{2}{r^2}-b^2H^2\right) =0. \]
The equality on the right hand side follows from $b=d/r$ and $d^2H^2-2r\ell=-2$, as $(r, dH, \ell) \in \Roots(\Lambda)$. This contradicts the assumption that $Z_{bH, tH}(r, dH, \ell) \in \bR_{\leq 0}$.

Now, let us show that $\sigma_\bullet$ is quasi-convergent as $t \to \lambda^+$. Conditions \ref{axiom1} and \ref{axiom2} hold by construction. If $(s_0,s_1,s_2) \in \ker Z_{\lambda}$, then by \eqref{eq:pathsK3_cchargesph}
\begin{align*}
((s_0,s_1,s_2),(s_0,s_1,s_2)) &= s_1^2-2s_0s_2 \\
&= (s_1-bHs_0)^2 - 2s_0(s_2-bHs_1 + \frac{1}{2}(bH)^2 s_0) \\
&= (s_1-bHs_0)^2 - \frac{2}{r^2}s_0^2.
\end{align*}
As $H.(s_1-bHS_0)=0$, the Hodge index theorem ensures that this expression is non-positive. Equality holds if $s_1-bHs_0=0$ and $s_0=0$, which in turn implies that $s_0=s_1=s_2=0$. This way, $(-,-)|_{\ker Z_{\lambda}}$ is negative definite, and so the path $\{\sigma_t\}_{t \to \lambda}$ satisfies \ref{axiom4} for $\nu =0$ by Proposition \ref{prop:gen2cy_axiom4}.

Finally, we determine the massless objects for $\sigma_\bullet$. By Lemma \ref{L:masslesscharacterization}, they all have constant phase 1. By Remark \ref{remark:pathsK3_stableph1}, the $\sigma_t$-stable objects of phase 1 are skyscraper sheaves and shifts $F[1]$ of $\mu_H$-stable vector bundles $F$ with $\mu_H(F)=bH^2$. We have $Z_t(\ko_x) = -1$, hence skyscraper sheaves are not massless. Thus, we let $E=F[1]$, for $F$ a $\mu_H$-stable vector bundle with $\mu_H(F)=bH^2$. Write $v(F)=(r',d'H,\ell')$, so that $\mu_H(E)=\mu_H(F)=d'H^2/r'$. As $b=d/r$ and $\gcd(r,d)=1$, we get that $d'=ad$ and $r'=ar$ for some $a \geq 1$. We also have that $(d'H)^2 - 2r'\ell'= -2+\dim \Ext^1(F,F)$. This way, for $t\geq \lambda$ we have
\begin{align*}
Z_t(F[1]) &= \ell' - bH.d'H - \frac{1}{2} r'(t^2H^2-b^2H^2) \\
&= \left(\frac{1}{ar}+ \frac{ad^2H^2}{2r} - \frac{1}{2ar}\dim \Ext^1(F,F)\right) - \frac{ad^2}{r}H^2 - \frac{1}{2}ar\left(t^2H^2-\frac{d^2}{r^2}H^2\right) \\ 
&= \frac{1}{ar} - \frac{1}{2}art^2H^2 - \frac{1}{2ar}\dim \Ext^1(F,F).
\end{align*}
If $F[1]$ is massless, then $Z_\lambda(F[1])=0$, i.e.
\[ 0 = \frac{1}{ar} - \frac{a}{r} - \frac{1}{2ar}\dim \Ext^1(F,F) = \frac{1-a^2}{ar} - \frac{1}{2ar}\dim \Ext^1(F, F). \]
It follows that $a=1$ and $\Ext^1(F, F)=0$, hence $F$ is spherical with $s(F)=(r,dH,\ell)$. We get that $F=G$ by Proposition \ref{prop:pathsK3_sphericalvb}. 

This way, the only massless stable object of phase 1 is $G[1]$. We get that \ref{axiom4} holds for all $\nu>0$ by Lemma \ref{L:qcpathsA4_massless}, and so $\sigma_\bullet$ is strongly quasi-convergent by Theorem \ref{T:strongquasiconvergence}. The filtration associated to the quasi-convergent path is $0\subset \langle G \rangle \subset \DCoh(S)$.
\end{proof}

We prove Theorem \ref{thm:pathsK3_main}(3) in several steps. To start, we fix notation. Let $f\colon S \to T$ be a birational map, where $S$ is a smooth, projective K3 and $T$ is a normal, projective surface with ADE singularities. We denote by $C_1, \dots, C_n$ the exceptional curves of $f$. Each $C_i$ is a smooth, rational curve of self-intersection $-2$. In particular,
\[ \ch(\ko_{C_i}(-1))=(0,C,0) = s(\ko_{C_i}(-1)). \]
We have a well-defined pullback map $f^\ast\colon \NS(T)_\bQ \to \NS(S)_\bQ$.

Fix $\eta \in \Amp(T)_\bQ$ with $\eta^2>2$. The pullback $f^\ast\eta$ is a nef $\bQ$-divisor in $S$, satisfying $(f^\ast\eta)^2=\eta^2>2$. We will approximate $f^\ast\eta$ by a path $\omega_t$ of ample divisors on $S$, in such a way that $\omega_t^2$ remains constant.

\begin{lem} \label{lemma:pathsK3_ordervanishing}
For any decomposition $\{1, \dots, n\}=I_1 \sqcup \dots \sqcup I_r$, there exists an analytic path $(\omega_t)_{t \in (0, \epsilon)}$ in $\Amp(S)_\bR$ such that: \vspace{-2mm}
\begin{enumerate}
    \item $\omega_t \to f^\ast\eta$ as $t \to 0$;\vspace{-2mm}
    \item for each $1\leq k\leq r$ and any $i \in I_k$, we have that $\omega_t.C_i \sim t^k$;\vspace{-2mm}
    \item $\omega_t^2$ is constant.\vspace{-2mm}
\end{enumerate}
\end{lem}

\begin{proof}
Recall that there exists a neighborhood $U \subset \NS(S)_\bR$ of $f^\ast\eta$, such that
\[ \Amp(S)_\bR \cap U = \bigcap_{i=1}^n \{\lambda: \lambda.C_i > 0\} \cap U. \]
The curves $C_1, \dots, C_n$ are linearly independent in $\NS(S)_\bR$ since the matrix $\{C_i.C_j\}_{i, j}$ is negative definite, and so $\Amp(S)_\bR$ is an intersection of $n$ independent half-spaces in a neighborhood of $f^\ast\eta$. This way, it is clear that we can take a path $\omega'_\bullet$ satisfying the first two conditions. In particular, $\eta^2/(\omega'_t)^2$ is analytic and converges to 1.

Finally, set $\omega_t = \omega'_t \cdot\sqrt{\eta^2/(\omega'_t)^2}$. By construction, $\omega_t$ is analytic and $\omega_t^2=\eta^2$. Finally, $\omega_t.C_i$ and $\omega'_t.C_i$ have the same order of vanishing, hence the second condition holds as well.
\end{proof}

We pick a partition $\{1, \dots, n\} = I_1 \sqcup \dots \sqcup I_r$, and fix $\omega_\bullet\colon (0, \epsilon) \to \Amp(S)_\bR$ as in Lemma \ref{lemma:pathsK3_ordervanishing}. We construct a path $\sigma_\bullet\colon (0, \epsilon) \to \Stab(S)$ as follows: take $\sigma_t:= \sigma^\Bridgeland_{0,\omega_t}\cdot \Phi$ where $\Phi = (T,f) \in G$ is the element given by 
\[
    T=
    \begin{pmatrix} 
    0 & 1 \\ 
    -1 & 0 
    \end{pmatrix} 
    \text{ and } 
    f(\phi)=\phi-1/2.
\] 
In other words, $\sigma_t = (Z_t, \cP_t)$ is a stability condition satisfying $\cP_t(1/2, 3/2] = \cA_{0,\omega_t}$ and 
\begin{equation} \label{eq:pathsK3_Cchargecurves}
\qquad Z_t(s_0, s_1, s_2) = -\omega_t\cdot s_1 + \mathtt{i} \left( -s_2+\frac{\omega_t^2}{2}s_0\right) = -\omega_t \cdot s_1 + \mathtt{i} \left( -s_2 + \frac{\eta^2}{2}s_0 \right).
\end{equation}
Our assumption that $\omega_t^2=\eta^2>2$ ensures that these stability conditions are defined for $t \in (0, \epsilon)$, i.e. that $(0, \omega_t) \in \mathscr{L}(S)$. Note that $Z_t(\ko_{C_i}(-1)) = -\omega_t.C_i \in \bR_{<0}$ by construction.

\begin{lem} \label{lemma:pathsK3_birqconv}
The path $\sigma_\bullet$ satisfies \ref{axiom1}, \ref{axiom2}, and \ref{axiom4} for $\nu=0$. 
\end{lem}

\begin{proof}
The path $\sigma_\bullet$ satisfies \ref{axiom1} by \eqref{eq:pathsK3_Cchargecurves}. By construction $\omega_t$ depends analytically on $t$, so $Z_t$ satisfies \ref{axiom2}. We claim that $(-,-)$ is negative definite on $\ker Z_0$. If $(s_0,s_1,s_2) \in \ker Z_0$, then $s_1 \cdot f^\ast\eta=0$ and $s_2=\frac{\eta^2}{2}s_0$, so we have
\[ 
    ((s_0,s_1,s_2),(s_0,s_1,s_2)) = s_1^2-2s_0s_2 = s_1^2 - \eta^2 s_0^2 \leq 0
\]
by the Hodge index theorem. Equality holds when $s_1=0$ and $s_0=0$, i.e. only for $(s_0,s_1,s_2)=0$. By Proposition \ref{prop:gen2cy_axiom4}, the path $\sigma_\bullet$ satisfies \ref{axiom4}.
\end{proof}

As in the other parts of the theorem, we need to classify the objects $E \in \cP_t(1)$ that are massless for $\sigma_\bullet$. The proof is divided into three steps. We first check that the massless objects of phase 1 are supported on the exceptional locus of $f$. Then, we prove that each $\ko_{C_i}(-1)$ is $\sigma_t$-stable of phase 1. Finally, we use this description to show that $\{\ko_{C_i}(-1)\}$ are the only massless $\sigma_t$-stable objects of phase 1.

\begin{lem}[cf. \cite{VilchesCurves}*{Lem. 3.3}] \label{lemma:pathsK3_stabletorsion}
Let $E$ be a $\sigma_t$-stable object of phase 1 for some $t$, and assume that $Z_t(E) \to 0$ as $t \to 0^+$. Then $E=H^0(E)$ is a pure dimension 1 sheaf, supported set-theoretically on the exceptional locus of $f$.
\end{lem}

\begin{proof}
Let $E$ be a $\sigma_t$-stable object of phase 1 for some $t$. By Lemma \ref{L:phase_one_independent}, we have that $E$ is $\sigma_t$-stable of phase 1 for all $t$. In particular, we have that $E \in \cA_{0,\omega_t}$ for all $t>0$. Then $E$ has at most two non-zero cohomology sheaves, namely
\[ H^0(E) \in \cT_{0, \omega_t}, \quad H^{-1}(E) \in \cF_{0, \omega_t}, \quad H^i(E)=0 \text{ otherwise}. \]

Assume that $H^{-1}(E) \neq 0$. Given $t$, let $F_t \subset H^{-1}(E)$ be the first Harder--Narasimhan factor of $H^{-1}(E)$. Since $H^{-1}(E) \in \cF_{0, \omega_t}$, we have that $\mu_{\omega_t}(F_t) \leq 0$. By construction, we get an injection $F_t[1] \hookrightarrow H^{-1}(E)[1]$ in $\cA_{0,\omega_t}$, hence an injection $F_t[1] \hookrightarrow E$. But $E \in \cP_{\sigma_\bullet}(1)$, hence $F_t[1]$ must have phase equal to 1.

It follows that $\Im Z_t(F_t[1])\geq 0$, and so
\[ 0 \leq \Im Z_t(F_t[1]) = s_2(F_t)-\frac{\eta^2}{2}s_0(F_t).\]
But $F_t$ is $\omega_t$-semistable, hence $\ch_1(F_t)^2\geq 2\ch_0(F_t)\ch_2(F_t)$ by the Bogomolov--Gieseker inequality. Putting these two inequailities together yields
\[ \ch_1(F_t)^2 \geq (\eta^2-2) \ch_0(F_t)^2 \geq (\eta^2-2). \]
The Hodge index theorem allows us to bound $\eta^2\ch_1(F_t)^2 =\omega_t^2\ch_1(F_t)^2\leq (\ch_1(F_t).\omega_t)^2$, and so 
\[ (\ch_1(F_t).\omega_t)^2 \geq \eta^2(\eta^2-2) \implies \ch_1(F_t).\omega_t \leq - \sqrt{\eta^2(\eta^2-2)}, \]
where we used that $\ch_1(F_t).\omega_t\leq 0$. 

Note that $F_t[1] \hookrightarrow E$ in $\cA_{0,\omega_t}$ implies that $0\geq\Re Z_t(F_t[1]) \geq \Re Z_t(E)$, and thus $\Re Z_t(E) \leq -\sqrt{\eta^2(\eta^2-2)}$. This bound holds for all $t$, which contradicts the convergence $\lim Z_t(E)=0$. Thus, we must have $H^{-1}(E)=0$.

A similar argument shows that $H^0(E)$ must be torsion. Note that $E$ is pure of dimension 1, as any skyscraper subsheaf will destabilize $E$ (and skyscraper sheaves are not massless). Finally, from $0 = \lim \Re Z_t(E) = -f^\ast\eta.\ch_1(E)$ we get that the support of $E$ is contained in $C_1 \cup \dots \cup C_n$.
\end{proof}

\begin{lem} \label{lemma:pathsK3_masslessbir}
For each $1 \leq i \leq n$, the object $\ko_{C_i}(-1)$ is $\sigma_\bullet$-stable of phase 1. 
\end{lem}

\begin{proof}
Note that $\ko_{C_i}(-1) \in \cA_{0, \omega_t}$ for all $t$. By Lemma \ref{L:phase_one_independent}, it suffices to check that $\ko_{C_i}(-1)$ is $\sigma_t$-stable for $t$ close to $0$. Thus, we fix $t$ such that $\omega_t.C_i< \sqrt{\eta^2(\eta^2-2)}$. 

We argue by contradiction. Assume that $\ko_{C_i}(-1)$ is not $\sigma_t$-stable, and let $A$ be its first Jordan--Hölder factor\footnote{Note that $A$ might depend on $t$.}. Note that $A \in \cA_{0, \omega_t}=\cP_{\omega_t}(1/2,3/2]$, hence $A$ has phase in $[1, 3/2]$. We get a short exact sequence
\[ 0 \to A \to \ko_{C_i}(-1) \to B \to 0 \]
in $\cA_{0, \omega_t}$. Taking cohomology with respect to the standard t-structure, we get that $A=H^0(A)$ and an exact sequence of sheaves
\[ 0 \to H^{-1}(B) \to H^0(A) \to \ko_{C_i}(-1) \xrightarrow{\xi} H^0(B) \to 0. \]

If $\xi$ were an isomorphism, then $H^{-1}(B) \cong H^0(A)$, which implies $H^{-1}(B) \in \cF_{0,\omega_t}$ and $H^0(A) \in \cT_{0, \omega_t}$, and so $A=0$, a contradiction. As a consequence, $H^0(B)$ has zero-dimensional or empty support. We get that $\ch_0(B)=-\ch_0(H^{-1}(B))=-\ch_0(A)$, and 
\begin{equation}
\ch_1(B)=-\ch_1(H^{-1}(B)) = -\ch_1(A)+C.
\end{equation}
Also $H^{-1}(B)\neq 0$, as otherwise $A$ and $B$ would be coherent sheaves; however, $B=H^0(B)$ would be stable of phase $3/2$ by being an extension of skyscraper sheaves. This implies that $\ch_0(H^{-1}(B))>0$, as $\cF_{0,\omega_t}$ only contains torsion-free sheaves.

By assumption, $A$ destabilizes $\ko_{C_i}(-1)$, and so $\Im Z_t(A) \leq 0$ as $A$ has phase in $[1, 3/2]$. It follows that
\begin{equation} \label{eq:pathsK3_curvesboundIm}
-s_2(A) + \frac{\eta^2}{2}s_0(A) \leq 0 \implies \ch_2(A) \geq \frac{\eta^2-2}{2}\ch_0(A).
\end{equation}
On the other hand, $A$ is $\sigma_t$-stable, and so $\chi(A,A)\geq -2$ by Lemma \ref{lemma:gen2cy_roots}. This implies that
\begin{equation} \label{eq:pathsK3_curvesboundChi}
\begin{split}
    \ch_1(A)^2 =s_1(A)^2  & \geq -2+2s_0(A)s_2(A) \\
    & = -2 + 2\ch_0(A)^2 + 2\ch_0(A)\ch_2(A) \geq 2\ch_0(A)\ch_2(A).
\end{split}
\end{equation}
Combining \eqref{eq:pathsK3_curvesboundIm} and \eqref{eq:pathsK3_curvesboundChi} shows that
\[ \ch_1(A)^2 \geq (\eta^2-2)\ch_0(A)^2 \geq (\eta^2-2). \]
Multiplying by $\eta^2=\omega_t^2$ and using the Hodge index theorem gives us
\[ (\ch_1(A).\omega_t)^2 \geq \ch_1(A)^2 \eta^2 \geq \eta^2(\eta^2-2), \]
hence $\ch_1(A).\omega_t \geq \sqrt{\eta^2(\eta^2-2)}$. 

Finally, note that $\ch_1(H^{-1}(B)).\omega_t\leq 0$, as $H^{-1}(B) \in \cF_{0, \omega_t}$. Putting these two inequalities together yields that
\[ \omega_t.C=\omega_t.(\ch_1(A)-\ch_1(H^{-1}(B))) \geq \sqrt{\eta^2(\eta^2-2)}, \]
contradicting our initial choice of $t$. It follows that $\ko_{C_i}(-1)$ is $\sigma_t$-stable.
\end{proof}

\begin{cor}
Let $E$ be a massless $\sigma_\bullet$-stable object of phase 1. Then $E\cong \ko_{C_i}(-1)$ for some $1 \leq i \leq n$.
\end{cor}

\begin{proof}
Assume that $E \neq \ko_{C_i}(-1)$ for all $1 \leq i \leq n$. By Lemma~\ref{lemma:pathsK3_stabletorsion}, $E=H^0(E)$ is a torsion sheaf, pure of dimension 1, supported set-theoretically on the exceptional locus of $f$. Thus, $\ch_1(E)=\sum a_i C_i$ for some $a_i \geq 0$, with at least one $a_i$ non-zero.

By assumption, $E$ is not isomorphic to $\ko_{C_i}(-1)$. But $E$ and $\ko_{C_i}(-1)$ are $\sigma_t$-stable of phase 1, hence
\[ 0 \geq \chi(E, \ko_{C_i}(-1)) =-((s_0(E), s_1(E), s_2(E), (0, C_i, 0)) = -s_1(E).C_i=-\ch_1(E).C_i. \]
But then $\ch_1(E)^2\geq 0$, as the $a_i$ are non-negative. This implies that $\ch_1(E)=0$, as the matrix $\{C_i.C_j\}_{1 \leq i,j\leq n}$ is negative definite. But this contradicts that $a_i \neq 0$ for some $i$.
\end{proof}

\begin{proof}[Proof of Theorem \ref{thm:pathsK3_main}(3)]
We consider the path $\sigma_\bullet$ constructed before. We proved in Lemma \ref{lemma:pathsK3_birqconv} that $\sigma_\bullet$ satisfies \ref{axiom1}, \ref{axiom2}, and \ref{axiom4} for $\nu=0$. The massless stable objects of phase 1 are given by $\{\ko_{C_i}(-1): 1 \leq i \leq n\}$, by Lemma \ref{lemma:pathsK3_masslessbir}. As there are finitely many massless stable objects, Lemma \ref{L:qcpathsA4_massless} guarantees that \ref{axiom4} holds for $\nu>0$, and so $\sigma_\bullet$ is strongly quasi-convergent by Theorem \ref{T:strongquasiconvergence}. Finally, by \eqref{eq:pathsK3_Cchargecurves}, the construction of Lemma \ref{lemma:pathsK3_ordervanishing}, and the formula $v(\ko_{C_i}(-1))=(0,C_i,0)$, we get
\[ Z_t(\ko_{C_i}(-1))=-\omega_t.C_i \sim t^k \]
for any $1 \leq k \leq r$ and any $i \in I_k$. Thus, the filtration associated to the path $\sigma_\bullet$ is
\begin{equation*} 
    0\subset \langle \cO_{C_i}(-1) \rangle_{i \in I_r} \subset \langle \cO_{C_i}(-1) \rangle_{i \in I_r \cup I_{r-1}} \subset \dots \subset \langle \cO_{C_i}(-1) \rangle_{i \in I_r \cup \dots \cup I_1} \subset \DCoh(S)
\end{equation*} 
and the proof of Theorem~\ref{thm:pathsK3_main} is complete.
\end{proof}

\subsection{The boundary stratum of a contraction}

In this brief section, we explain how the contraction $f\colon S\to T$ considered in Theorem~\ref{thm:pathsK3_main} gives rise to boundary strata of $\Astab(S)$. We let $T^{\rm{sing}}$ denote the singular locus of $T$, which consists of finitely many points. The functor $\bR f_\ast \colon \DCoh(S) \to \DCoh(T)$ induces an equivalence $\DCoh(S)/ \ker \bR f_\ast \simeq \DCoh(T)$, where $\ker \bR f_\ast = \langle\cO_{C_i}(-1)\rangle_{1 \leq i \leq n}$ by \cite{KuznetsovSextic}*{Lem. 2.3, Cor. 2.5}.

A decomposition $\{1,\ldots, n\} = I_1\sqcup \cdots \sqcup I_r$, denoted $\chi$, is \emph{compatible} with $f$ if $f(C_i) = f(C_j)$ implies that $i,j \in I_p$ for some $1\le p\le r$. In particular, for each $1\le j \le r$ the set $\bigcup_{i\in I_j}C_i$ is a union of fibers over singular points and so for any $y \in T^{\rm{sing}}$ there is a unique $j \in \{1, \dots, r\}$ such that $f^{-1}(y) \subset \bigcup_{i \in I_j} C_i$. We write $y\subset j$ to indicate this relation.

Fix $\chi$ as above and a quasi-convergent path $\sigma_\bullet$ as in Theorem~\ref{thm:pathsK3_main}(3). The resulting filtration is \eqref{E:filtration_thm} and we denote $\cK_{\le j} := \langle \cO_{C_i}(-1)\rangle_{i \in I_r\cup \cdots \cup I_{r-(j-1)}}$ for each $j =1,\ldots, r$. In this case, the corresponding multi-scale decomposition has two levels (see Definition~\ref{D:multiscaledecomp}) and $r+1$ nodes on the root component. Thus, it is specified by $r+1$ points in $\bC$ with the same imaginary part:
\[
    \begin{tikzpicture}[scale = .75]
        \fill[gray!25] (0,0) rectangle (8,4);

        \fill (1,2) circle (2pt);
        \fill (2.5,2) circle (2pt);
        \fill (5.5,2) circle (2pt);
        \fill (7,2) circle (2pt);

        \node[below=5pt] at (2.5,2) {$\cK_{\le 2}$};
        \node[below=3.5pt] at (7,2) {$\DCoh(S)$};
        \node at (4,2) {$\cdots$};
        \node[below=5pt] at (1,2) {$\cK_{\le 1}$};
        \node[below = 5pt] at (5.5,2) {$\cK_{\le r}$};
    \end{tikzpicture}
\]
Note that here $\cK_{\le r} = \ker \bR f_*$ is the category of massless objects. For each $y \in T^{\rm{sing}}$, let $\cF_{y} := \langle \cO_{C_i}(-1)\rangle_{f(C_i)=y}$. By disjointness of support, we have an orthogonal decomposition $\ker \bR f_* = \bigoplus_{y \in T^{\rm{sing}}} \cF_y$. Note that the subcategory $\ker \bR f_*$ is not admissible in $\DCoh(S)$.

\begin{defn}
    Let $f\colon S\to T$ be as above and fix a compatible partition $\chi$. We let $\cM_f^{\chi}$ denote the subspace of $\Astab(S)$ comprised of augmented stability conditions with underlying multi-scale decomposition 
    \[
        \DCoh(S) = \langle \cK_{\le 1},\ldots, \cK_{\le r},\DCoh(S)\rangle_{\Sigma},
    \] 
    where $\Sigma$ is 2-level multi-scale line with terminal components $v_1,\ldots, v_{r+1}$ such that $\mathfrak{p}(v_i,v_j) = 1$ for all $1\le i<j\le r+1$, as illustrated above. 
\end{defn}

The works \cites{Br09,ThomasStabBraid} study the stability manifolds of the 2-Calabi--Yau categories associated to ADE surface singularities. Given an ADE Dynkin diagram $\Gamma$, we denote by $\cD_\Gamma$ the corresponding category, which admits a description as a 2-Calabi--Yau completion of modules over the quiver $\Gamma$ with some/any choice of orientation. Given $y\in T^{\rm{sing}}$, the type of its singularity corresponds to an ADE Dynkin diagram, which we denote by $\Gamma(y)$. It follows that $\cF_y \simeq \cD_{\Gamma(y)}$ for each $y\in T^{\rm{sing}}$.

\begin{thm}
\label{T:stratum_description}
    The subspace $\cM_f^\chi$ is a topologically trivial bundle 
    \[
        \cM_f^\chi \longrightarrow (\Stab(T)/\bC)\times U
    \]
    with fiber $(\prod_{j=1}^r \prod_{y\subset j} \Stab(\cF_{y}))/\bC^r$, where $\bC^r$ acts diagonally and $U$ is homeomorphic to $\bR_{>0}^{r-1}$. Further\-more, the quasi-convergent paths $\sigma_\bullet$ in Theorem~\ref{thm:pathsK3_main}(3) with partition $\chi$ converge to points of $\cM_f^\chi$. 
\end{thm}

\begin{proof}
    By definition of the strong topology on $\Astab(\DCoh(S))$ in \cite{augmented}*{\S 5}, the subspace $\cM_f^\chi$ is homeomorphic to $(\prod_{j=1}^r \Stab(\cK_{\le j}/\cK_{\le j-1}))/\bC^r \times \Stab(T)/\bC \times U$, where $\bC^r$ acts diagonally and $U$ is a space of parameters describing deformations of the underlying real multi-scale line $\Sigma$ --- see \cite{admissibleaugmented}*{\S 3.2}. More precisely, $U$ is the subspace of the moduli space of $(r+1)$-marked multi-scale lines up to real oriented isomorphism $\cA_{r+1}^{\bR}$ as in \cite{augmented}*{\S 4} of the type depicted above. Thus, $U$ is diffeomorphic to the moduli of labeled points $x_1< \cdots< x_{r+1}$ in $\bR$ where $x_1 = 0$ and $x_2 = 1$; therefore, $U \cong \bR_{>0}^{r-1}$. 
    
    Next, by definition of compatibility of $\chi$ with $f$, the natural functor $\bigoplus_{y\subset j} \cF_y \to \cK_{\le j}/\cK_{\le j-1}$ is an equivalence. The description of the fiber now follows from the fact that $\Stab(\bigoplus \cF_y) \cong \prod \Stab(\cF_y)$ by \cite{augmented}*{Lem. 6.12}. The statement about convergence is a tautological consequence of Theorem~\ref{thm:pathsK3_main}.
\end{proof}

By the results of \citelist{\cite{Br09} \cite{ThomasStabBraid}}, a distinguished connected component of $\Stab(\cF_y) \cong \Stab(\cD_{\Gamma(y)})$ is isomorphic to a connected cover of $\mathfrak{h}^{\rm{reg}}/W$, where $\mathfrak{h}$ is the Cartan subalgebra of the Lie algebra $\mathfrak{g}$ associated to $\Gamma(y)$ and $W$ the corresponding Weyl group. The stability conditions produced by Theorem~\ref{thm:pathsK3_main}(3) on the associated graded categories are in (the product of) these connected components.

\begin{rem}
\label{R:parameters}
    As mentioned in the introduction, mirror to the contraction $f\colon S\to T$ is a smoothing of a nodal K3 surface $X_0$ mirror to $T$. For the purposes of illustration, let us consider the simplest case of Theorem~\ref{T:stratum_description} where each fiber of $f$ is irreducible and where the partition $\chi$ of $\{1,\ldots, n\}$ is trivial, consisting of one block. In this case, we write $\cM_f$ for the corresponding subspace of $\Astab(S)$. By Theorem~\ref{T:stratum_description}, $\cM_f$ is a trivial bundle over $\Stab(T)/\bC$ with fiber the quotient of
    \begin{equation}
            \Stab(\ker \bR f_*) \cong \bigoplus_{y\in T^{\rm{sing}}} \Stab(\cF_y) \xrightarrow{\sim} \prod_{i=1}^n \bC^n
    \label{E:StabkerRf}
    \end{equation}
    by $\bC$. The fiber coordinates describe the residual central charges $Z_\infty(\cO_{C_i}(-1))$ of the $\sigma_\bullet$-massless spherical objects, which should correspond to the smoothing parameters of the nodes on $X_0$. Indeed, these parameters, or rather their logarithms, are recorded by the isomorphism $\Stab(\ker \bR f_*) \to \bC^n$ given in \eqref{E:StabkerRf}, which sends $\sigma \mapsto (\logZ_\sigma(\cO_{C_i}(-1)))_{i=1}^n$. Thus, in view of Heuristic~\ref{Heur:CY}, Theorem~\ref{T:stratum_description} suggests that (certain) points of $\cM_f$ record a choice of complex structure on $X_0$ along with some infinitesimal deformation parameters remembered by stability conditions on the massless subcategory $\ker \bR f_*$. 
\end{rem}

\section{Enriques surfaces} \label{sec:Enriques}

In this section, we construct strongly quasi-convergent paths in the space of stability condition of an Enriques surface, thereby proving Theorem \ref{thm:intro_Enriques}. The key insight is that the numerical K-theory of an Enriques surface behaves similarly to that of a K3 surface; in particular, an analogue of Proposition \ref{prop:gen2cy_axiom4} holds. 

We review general properties of Enriques surfaces in Subsection \ref{subsec:genEnriques}, and study their numerical K-theory in Subsection \ref{subsec:latticeEnriques}. We then recall the construction of stability conditions in Subsection \ref{subsec:stabEnriques}. Finally, we use these results to construct quasi-convergent paths in Subsection \ref{subsec:pathsEnriques}.

Through this section we work over an algebraically closed field $k$ of characteristic $\neq 2$.

\subsection{Generalities} \label{subsec:genEnriques}

An \emph{Enriques surface} is a smooth, projective surface $T$ satisfying  $\omega_T^{\otimes 2} \cong \ko_T$, $\omega_T \not\cong \ko_T$, and $\dim H^1(T, \ko_T)=0$. If $T$ is an Enriques surface over $\bC$, then its non-zero Hodge numbers are $h^{0,0}=h^{2,2}=1$ and $h^{1,1}=10$, its fundamental group is $\bZ/2\bZ$, and its universal cover is a K3 surface. We refer to \citelist{\cite{CDL25} \cite{BHPV04}*{VIII.15}} for basic properties of Enriques surfaces. See also the short summary in \cite{LNSZ21}.

We point out that there is a short exact sequence of abelian groups
\[ 0 \to \bZ/2\bZ \to \NS(T) \to \Num(T) \to 0, \]
where $\NS(T)$ is the group of divisors modulo algebraic equivalence\footnote{Since $H^1(T, \ko_T)=0$, algebraic equivalence agrees with linear equivalence.}, and $\Num(T)$ is the group of divisors modulo linear equivalence. We usually write elements of $\NS(T)$ with uppercase letters and elements of $\Num(T)$ with lowecase letter. Given $F \in \NS(T)$, we write $[F]$ for its image in $\Num(T)$. The inclusion map $\bZ/2\bZ \to \NS(T)$ maps the non-zero element to the class of the canonical divisor. In particular, any $f \in \Num(T)$ admits two lifts $F^+, F^- \in \NS(T)$, satisfying $F^+-F^- \sim K_T$.

\begin{prop}[\cite{BHPV04}*{VII.15.2--3}] \label{prop:genEnriques_form}
An Enriques surface has $\Pic(T) \cong H^2(T, \bZ) \cong \bZ^{10} \oplus (\bZ/2\bZ)$. Moreover, the intersection form on $\Num(T) \cong H^2(T, \bZ)_{\mathrm{tf}}$ is even and non-degenerate.
\end{prop}

The moduli space of (marked) Enriques surfaces is a quasi-projective 10-dimensional variety \cite{CDL25}*{Thm. 5.3.9}. Enriques surfaces containing a $(-2)$-curve are called \emph{nodal}\footnote{Note that nodal does \emph{not} mean that they are singular in this context.}, and the locus of nodal surfaces is a divisor in the moduli of Enriques surfaces \cite{CDL25}*{Thm. 5.3.2}. This way, the \emph{unnodal} Enriques surfaces, i.e. those not containing $(-2)$-curves, form a dense open set in the moduli space. We point out that an unnodal Enriques surface has no curves of negative self-intersection. In particular, any effective divisor is automatically nef. From now on we will focus on unnodal Enriques surfaces. Note however that many of the results below can be adapted for nodal Enriques surfaces.

\begin{prop}[\cite{CDL25}*{\textsection 3.5}, cf. \cite{LNSZ21}*{Defn. 3.2}]
Let $T$ be an unnodal Enriques surface. Then there exists an ample divisor $\Delta$ satisfying: \vspace{-2mm}
\begin{enumerate}
\item $\Delta^2 =10$; \vspace{-2mm}
\item $\Delta.F\geq 3$ for any nef divisor $F$ satisfying $F^2=0$. \vspace{-2mm}
\end{enumerate}
We call $\Delta$ a \emph{Fano polarization}.
\end{prop}

\begin{lem}[\cite{CDL25}*{Lem. 3.5.2}] \label{lemma:genEnriques_tenLB}
Let $\Delta$ be a Fano polarization. Then there exists a canonical collection $f_1, \dots, f_{10} \in \Num(T)$ of nef classes, satisfying $3[\Delta]=f_1+\dots+f_{10}$. The classes $f_i$ are characterized by
\[ 
    \{f_1, \dots, f_{10}\} = \{f \in \Num(T): f \text{ nef}, f.[\Delta]=3\}. 
\]
\end{lem}

Given a Fano polarization $\Delta$ on an unnodal Enriques surface $T$, consider one of the classes $f_i\in \Num(T)$ given by Lemma \ref{lemma:genEnriques_tenLB}. The class $f_i$ admits two lifts $F_i^+, F_i^- \in \NS(T)$, satisfying $F_i^+-F_i^-\sim K_T$. The two classes can be represented by effective divisors and are known as \emph{half-fibers}.

\begin{thm}[\cite{Zub97}*{\textsection 5}]
Let $T$ be an unnodal Enriques surface, and fix a Fano polarization $\Delta$, with $f_1, \dots, f_{10}$ the associated nef classes, and $F_1^\pm, \dots, F_{10}^\pm$ the half-fibers.

For any choice of $\epsilon_1, \dots, \epsilon_{10} \in \{+, -\}$, the ten line bundles $\{\ko_T(F_1^{\epsilon_1}), \dots, \ko_T(F_{10}^{\epsilon_{10}})\}$ form an orthogonal exceptional sequence in $\DCoh(T)$.
\end{thm}

\subsection{Numerical K-theory} \label{subsec:latticeEnriques}

Fix $T$ an Enriques surface. We consider the index two lattice $\Lambda \subset \bZ \oplus \Num(T) \oplus \frac{1}{2} \bZ$ given by the formula
\[ \Lambda = \left\{ (s_0, s_1, s_2): \frac{1}{2}s_0 + s_2 \in \bZ \right\}. \]
We endow $\Lambda$ with the quadratic form
\[ ((s_0,s_1,s_2), (s_0',s_1',s_2')) = s_1\cdot s_1'-s_0s_2'-s_2s_0'. \]
Note that the quadratic form is integral and has signature $(2,10)$. We point out that $(\Lambda, (-,-))$ is \emph{not} even, as $((1,0,1/2),(1,0,1/2))=-1$. 

On the other hand, we endow $\Knum(T)$ with the pairing induced by the negative of the Euler pairing: 
\[ \langle [E], [F]\rangle := -\chi(E,F). \]
As the following result shows, the two lattices are related by a map $s\colon \Ob(\DCoh(T)) \to \Lambda$ similar to the Mukai vector.

\begin{prop}
The map $s\colon \Ob(\DCoh(T)) \to \bZ \oplus \Num(T) \oplus \frac{1}{2}\bZ$ given by
\[ s(E) = (\ch_0(E), \ch_1(E), \ch_2(E)+\frac{1}{2}\ch_0(E)) \]
induces an isometry $(\Knum(T), (-,-)) \xrightarrow{\sim} (\Lambda, \langle-,-\rangle)$.
\end{prop}

\begin{proof}
We first show that the image of $s$ is $\Lambda$. If $E$ is a vector bundle and $s(E) =(s_0,s_1,s_2)$, then
\[ \frac{1}{2}s_0+s_2=\ch_2(E) + \ch_0(E) = \frac{1}{2}c_1(E)^2-c_2(E)+\ch_0(E) \]
is an integer, since $c_1(E)^2$ is even by Proposition \ref{prop:genEnriques_form}. This shows that the image of $s$ is contained in $\Lambda$; the reverse inclusion is direct.

Now, recall that the Todd class of $T$ is $1+[pt]$, where $[pt]$ is the class of a point. This way, given $E, F \in \DCoh(T)$, the Hirzebruch--Riemann--Roch theorem yields
\begin{align*}
\chi(E, F)&= \int_T \ch(E)^\vee \cdot \ch(F) \cdot (1+[pt]) \\
&= \int_T (\ch_0(E)-\ch_1(E) + \ch_2(E)) (\ch_0(F)+\ch_1(F)+\ch_2(F))(1+[pt]) \\
&= \ch_0(E)\ch_0(F) + \ch_0(E) \ch_2(F) -\ch_1(E)\ch_1(F) + \ch_2(E) \ch_0(F) \\
&= -\langle s(E), s(F)\rangle.
\end{align*}
It follows that the map $s$ factors through $\Knum(T)$, and the induced map is an isometry.
\end{proof}

\subsection{Stability conditions} \label{subsec:stabEnriques}

We will briefly recall the construction of Bridgeland stability conditions on $\DCoh(T)$ by Arcara--Bertram \cite{ABLStabilitysurf}. Using their construction, we will adapt the results of Proposition \ref{prop:gen2cy_axiom4} to Enriques surfaces.

Let $\omega \in \Amp(T)_\bQ$ and $\beta \in \NS(T)_\bQ$ be given. We consider the central charge
\[ 
    Z_{\beta, \omega}^{\ArcaraBertram}(E) = -\ch_2(E) + \beta.\ch_1(E)+ \frac{\omega^2-\beta^2}{2}\ch_0(E) + \mathtt{i}\omega\cdot(\ch_1(E)-\beta\ch_0(E)), 
\]
together with the heart $\cA_{\beta,\omega}$ as defined in \eqref{eq:pathsK3_heart}. The pair $\sigma_{\beta,\omega}^\ArcaraBertram:=(Z_{\beta,\omega}^\ArcaraBertram, \cA_{\beta,\omega})$ defines a pre-stability condition by \cite{ABLStabilitysurf}. The support property holds by adapting the arguments of \cite{Tod13}*{\textsection 3.6}. A deformation argument gives a continuous map
\[ \Sigma^\ArcaraBertram\colon \NS(T)_\bR \times \Amp(T)_\bR \to \Stab(T) \]
extending the assignment $(\beta,\omega) \mapsto \sigma_{\beta,\omega}^\ArcaraBertram$ for rational classes. We denote by $\Stab^\dagger(T) \subset \Stab(T)$ the connected component containing the image of $\Sigma^\ArcaraBertram$. 

For future reference, we point out the following fact about stable objects of phase 1.

\begin{lem}[\cite{ABLStabilitysurf}*{p. 8}] \label{lemma:stabEnriques_stphase1}
If $E$ is a $\sigma_{\beta,\omega}^\ArcaraBertram$-stable object of phase 1, then $E=\ko_x$ for some $x \in T$, or $E=F[1]$ for some $\mu_\omega$-stable vector bundle $F$ on $T$ satisfying $\mu_\omega(F)=\beta.\omega$.
\end{lem}

\begin{lem} \label{lemma:stabEnriques_invariance}
For any $\sigma \in \Stab^\dagger(T)$, we have $\sigma \otimes \ko(K_T) = \sigma$.
\end{lem}

\begin{proof}
Since $\ko_T(K_T)$ is a numerically trivial, $-\otimes \ko_T(K_T)$ acts trivially on $\Knum(T)$. This way, it suffices to prove the statement for a single $\sigma \in \Stab^\dagger(T)$ by Bridgeland's deformation theorem.

Pick $\beta \in \NS(T)_\bR$ and $\omega \in \Amp(T)_\bR$. For any torsion-free sheaf $E$, we have $\ch_i(E \otimes \ko_T(K_T))=(\ch(E).\ch(\ko_T(K_T))_i=\ch_i(E)$ for all $i$. It follows directly that $E$ is $\mu_\omega$-semistable if and only if $E \otimes \ko_T(K_T)$ is $\mu_\omega$-semistable, and that the slopes are preserved. Thus the heart $\cA_{\beta, \omega}$ is preserved by $-\otimes \ko_T(K_T)$. Since the central charge $Z_{\beta,\omega}^\ArcaraBertram$ is also preserved, we see that $\sigma_{\beta,\omega}^\ArcaraBertram \otimes \ko_T(K_T) = \sigma_{\beta,\omega}^{\ArcaraBertram}$. 
\end{proof}

\begin{cor}[cf. Lemma \ref{lemma:gen2cy_roots}] \label{cor:stabEnriques_stablebound}
Let $E$ be a $\sigma$-stable object for some $\sigma \in \Stab^\dagger(T)$. Then $(s(E), s(E))\geq -2$. 
\end{cor}

\begin{proof}
Note that $\Hom^i(E,E)=0$ for $i <0$ and $\Hom^0(E,E) \cong k$, as $E$ is $\sigma$-stable. On the other hand, $E \otimes \ko_T(K_T)$ is $\sigma$-stable of the same phase by Lemma \ref{lemma:stabEnriques_invariance}. Together with Serre duality, we get that $\Hom^i(E, E) \cong \Hom^{2-i}(E, E\otimes \ko_T(K_T))=0$ for $i=2$ and zero or $\bC$ for $i=1$. Combining these observations yields
\begin{align*}
-\chi(E, E) &= -\dim \Hom(E, E) + \dim \Hom^1(E, E) - \dim \Hom(E, E\otimes \ko_T(K_T)) \\ &\geq \dim \Hom^1(E, E)-2.
\end{align*}
We conclude as the left hand side equals $(s(E), s(E))$. 
\end{proof}

Corollary \ref{cor:stabEnriques_stablebound} implies an analog of Proposition \ref{prop:gen2cy_axiom4} for $\Stab^\dagger(T)$. We denote by $\Roots(T)$ the collection of elements $v \in \Lambda$ such that $(v, v) =-1$ or $-2$.

\begin{prop} \label{prop:stabEnriques_axiom4}
Let $\sigma_\bullet$ be a path of stability conditions in $\Stab^\dagger(T)$. If $\sigma_\bullet$ satisfies \ref{axiom1} and \ref{axiom2}, and $(-,-)|_{\ker Z_\infty}$ is negative definite, then \ref{axiom4} holds. 
\end{prop}

\begin{proof}
    Apply the proof of Proposition \ref{prop:gen2cy_axiom4}, substituting Lemma \ref{lemma:gen2cy_roots} with Corollary \ref{cor:stabEnriques_stablebound}.
\end{proof}

\subsection{Quasi-convergent paths} \label{subsec:pathsEnriques}

The goal of this subsection is to prove the following result:

\begin{thm} \label{thm:pathsEnriques_main}
Let $T$ be an Enriques surface.\vspace{-2mm}
\begin{enumerate}
\item There exists a strongly quasi-convergent path $\sigma_\bullet$ with associated filtration 
\[
    0\subset \langle \cO_T, \cO_T(K_T) \rangle \subset \DCoh(T).
\]
\item Assume that $T$ is unnodal, fix a Fano polarization $\Delta$, and let $F_1^\pm, \dots, F_{10}^\pm$ be the half-fibers associated to the nef classes $f_1, \dots, f_{10}$. Then there exists a strongly quasi-convergent path $\sigma_\bullet$ with associated filtration 
\[
    0\subset \langle \ko_T(F_1^+), \ko_T(F_1^-), \dots, \ko_T(F_{10}^+), \ko_T(F_{10}^-)\rangle \subset \DCoh(T).
\]
\end{enumerate}
\end{thm}

For (1), the construction follows the idea of Theorem \ref{thm:pathsK3_main}, by considering a $G$-translate of $\sigma_{0, aH}^\ArcaraBertram$ for each $a>0$. For (2), we use that $\Delta.\ch_1(\ko_T(F_i^\pm))=3$ for every $i$. This way, the shifts $\ko_T(F_i^\pm)[1]$ are $\sigma_{\beta, t\Delta}^\ArcaraBertram$-stable, provided that $\beta.\Delta=3$. This motivates our choice of stability conditions in the proof of (2).

\begin{proof}[Proof of Theorem \ref{thm:pathsEnriques_main}(1)]
Choose an ample divisor $H$ on $T$. We define a path 
\[
    \sigma_\bullet = (Z_\bullet, \cA)\colon (1/2,\infty) \to \Stab(T)
\]
as follows. Set $\cA = \cA_{0, H}$, and
\begin{align*}
Z_t(E) &= -s_2(E)+ts_0(E)+\mathtt{i}H.s_1(E) \\
&= -\ch_2(E) + \left( t-\frac{1}{2}\right)\ch_0(E) + \mathtt{i}H.s_1(E).
\end{align*}
Note that $(-,-)|_{\ker Z_t}$ is negative definite for all $t>0$: indeed, if $v=(s_0,s_1,s_2) \in \ker Z_t$, then $s_2=ts_0$ and $H.s_1=0$, hence $(v,v)=s_1^2-2ts_0^2 \leq 0$ by the Hodge index theorem. Equality holds only when $s_1=0$ and $s_0=0$, which implies $v=0$.

We claim that for $t>1/2$, $\sigma_t \in \Stab^\dagger(T)$. In fact, after rescalling the imaginary part by $a(t)=\sqrt{(2t-1)/H^2}$, $\sigma_t$ agrees with $\sigma_{0, a(t)H}^\ArcaraBertram$. Thus, $\sigma_t$ is in the $G$-orbit of the Arcara--Bertram stability conditions. 

The path $\sigma_\bullet$ satisfies \ref{axiom1} and \ref{axiom2} by construction. By Proposition \ref{prop:stabEnriques_axiom4}, the path satisfies \ref{axiom4} for $\nu=0$. 

Let us determine the massless objects of phase 1. By Lemma \ref{lemma:stabEnriques_stphase1}, the $\sigma_t$-stable objects of phase 1 are skyscraper sheaves and shifts $F[1]$ of $\mu_H$-stable vector bundles with $\mu_H(F)=0$. We have $Z_t(\ko_x)=-1$ for all $t$, hence skyscraper sheaves are not massless. Let $E=F[1]$ be a shift of a $\mu_H$-stable vector bundle and assume that $E$ is massless, so that $s(E) \in \ker Z_{1/2}$ and hence $H.s_1(E) =0$ and $s_2(E)=\frac{1}{2}s_0(E)$. By Corollary \ref{cor:stabEnriques_stablebound},
\[ -2 \leq s_1(E)^2 -2s_0(E)s_2(E) = s_1(E)^2 - s_0^2(E). \]
The term $s_1(E)^2$ is either $0$ or $\leq -2$, by the Hodge index theorem and evenness of the intersection form on $\Num(T)$. The second term is $s_0(E)^2=\ch_0(F)^2$, which is $1$ or $\geq 4$.

Thus, $s_1(E)=0$ and $s_0(E)=1$, and it follows that $F$ is a vector bundle with $\ch_0(F)=1$, $\ch_1(F)=0$, and $\ch_2(F)=0$. Here $F$ is a line bundle with $\ch_1(F)$ either $0$ or $K_T$, hence $F=\ko_T$ or $F=\ko_T(K_T)$. To conclude, \ref{axiom4} holds for $\nu\geq 1$ by Lemma~\ref{L:qcpathsA4_massless}, and so $\sigma_\bullet$ is strongly quasi-convergent by Theorem \ref{T:strongquasiconvergence}.
\end{proof}

The proof of Theorem \ref{thm:pathsEnriques_main}(2) is divided into several steps. We fix an unnodal Enriques surface $T$ and a Fano polarization $\Delta$. We let $f_1, \dots, f_{10}$ be the nef classes satisfying $\Delta.f_i =3$, and $F_i^\pm$ be the associated half-fibers, as in Lemma \ref{lemma:genEnriques_tenLB} and the discussion thereafter. 

Let us write $\beta=\frac{3}{10}\Delta$,  $\cA:=\cA_{\beta, \Delta}$, 
\begin{align*}
Z_t(E) &= -s_2(E) + \beta.s_1(E) + ts_0(E) + \mathtt{i}\Delta.\left(s_1(E)-\beta s_0(E)\right) \\
&= -\ch_2(E) + \beta.\ch_1(E) + \left( t-\frac{1}{2}\right) \ch_0(E) + \mathtt{i}\Delta.\left(\ch_1(E)-\beta \ch_0(E)\right).
\end{align*}
Note that $v(\ko_T(F_i^\pm)) =(1,f_i,1/2)$ and so $Z_t(\ko_T(F_i^\pm)[1])=-\frac{2}{5}-t$.

We first show that $\sigma_t$ is a stability condition for all $t>-2/5$, and describe the massless objects of $\sigma_\bullet$ as $t \to (-2/5)^-$. Finally, we use these results to conclude strong quasi-convergence. We begin with the following two observations:

\begin{lem}
For any $t>-\frac{9}{20}$, $(-,-)|_{\ker Z_t}$ is negative definite.
\end{lem}

\begin{proof}
If $v=(s_0,s_1,s_2) \in \ker Z_t$, then
\begin{align*}
(v,v) &= s_1^2-2s_0s_2 = (s_1-\beta s_0)^2 - 2s_0(s_2-\beta.s_1 + \frac{\beta^2}{2}s_0) \\
&= (s_1-\beta s_0)^2 - 2s_0^2(\frac{9}{10}+t) \leq 0,
\end{align*}
by the Hodge index theorem and $t>-\frac{9}{20}$. Equality gives $s_0=0$ and $s_1-\beta s_0=0$, so $v=0$.
\end{proof}

\begin{lem}
For any $t>1/20$, $\sigma_t\in \Stab^\dagger(T)$.
\end{lem}

\begin{proof}
The condition $t>\frac{1}{20}$ implies that $t-\frac{1}{2}-\frac{\beta^2}{2}>0$, so we can write $t-\frac{1}{2} = \frac{(s\Delta)^2-\beta^2}{2}$ for some $s>0$. Thus, $\sigma_t$ is a $G$-translate of $\sigma_{\frac{3}{10}\Delta, s\Delta}^\ArcaraBertram$, and is in particular in $\Stab^\dagger(T)$.
\end{proof}

To extend the path $\sigma_t$ to all $t\in (-2/5, +\infty)$, we analyze the $\sigma_t$-stable objects of phase 1, and prove that none of them becomes massless before $-2/5$. To do so, we will need:

\begin{lem} \label{lemma:pathsEnriques_massless10}
Let $F$ be a $\mu_\Delta$-stable vector bundle on $T$ with $\mu_\Delta(F)=3$ and let $t\geq-2/5$ be given. Then $\Re Z_t(F[1]) <0$, unless $F \cong \ko_T(F_i^\pm)$ and $t=-2/5$.  
\end{lem}

\begin{proof}
We write $s_i=s_i(F)$, so that $Z_t(F[1]) = s_2-\beta.s_1-ts_0 + i\Delta(\beta s_0-s_1)$.
By assumption, $\mu_\Delta(F)=3$, i.e. $\Delta.s_1=3s_0$. This gives us $\Im Z_t(F[1])=0$ and
\[ 
    \Re Z_t(F[1]) = s_2-\frac{9}{10}s_0 - ts_0. 
\]
Assume that $F \not\cong \ko_T(F_i^\pm)$ for all $i$. We have that $F$ and $\ko_T(F_i^\pm)$ are pairwise non-isomorphic $\mu_\Delta$-stable vector bundles of the same slope, hence there are no maps between them. Thus,
\begin{align*} 
-\chi(\ko_T(F_i^+), F) &= -\dim \Hom(\ko_T(F_i^+), F) + \dim \Ext^1(\ko_T(F_i^+), F) - \dim \Ext^2(\ko_T(F_i^+), F) \\
&= \dim \Ext^1(\ko_T(F_i^+), F) - \dim \Hom(F, \ko_T(F_i^++K_T)) \\
&=\dim \Ext^1(\ko_T(F_i^+), F) - \dim \Hom(F, \ko_T(F_i^-)) \geq 0,
\end{align*}
for all $i$. This gives us
\[ 
    0 \leq -\chi(\ko_T(F_i^+), F) = ((1,f_i, 1/2), (s_0, s_1, s_2)) = s_1.f_i-s_2-\frac{1}{2}s_0. 
\]
Summing over $i$ and using that $3[\Delta] = \sum_i f_i$, we get
\[ 
    3\Delta. s_1 - 10 s_2 - 5s_0 \geq 0 \implies s_2 \leq \frac{2}{5}s_0. 
\]
Therefore,
\[ 
    \Re Z_t(F[1]) \leq \frac{2}{5}s_0 - \frac{9}{10}s_0 - ts_0 = \left( -\frac{1}{2}-t\right) s_0, 
\]
which is negative for $t\geq -2/5>-1/2$. Finally, note that $\Re Z_t(\ko_T(F_i^\pm)[1]) = - \frac{2}{5}-ts_0$, which is negative for $t>-2/5$ and zero for $t=-2/5$.
\end{proof}

Let us write $a = \inf \{ s>-2/5: \sigma_t \in \Stab^\dagger(T) \text{ for all }t>s\}$. By construction, $\sigma_t \in \Stab^\dagger(T)$ for all $t>a$, and that $a \leq 1/20$. Note also that if $a>-2/5$, then $\sigma_a$ is \emph{not} a stability condition, as otherwise Bridgeland's deformation theorem would imply that $\sigma_{a-\epsilon} \in \Stab^\dagger(T)$ for $0<\epsilon \ll 1$.

\begin{proof}[Proof of Theorem \ref{thm:pathsEnriques_main}(2)]
Note that $\{\sigma_t: t>a\}$ satisfies \ref{axiom1} and \ref{axiom2} by construction, and Proposition \ref{prop:gen2cy_axiom4} shows that \ref{axiom4} holds for $\nu=0$.

Let us describe the massless subcategory. Recall that if $E$ is $\sigma_t$-semistable of phase 1, then either $E=\ko_x$ for some $x \in T$, or $E=F[1]$ for some $\mu_\Delta$-stable vector bundle $F$ with $\mu_\Delta(F)=3$ by Lemma \ref{lemma:stabEnriques_stphase1}. Here $Z_a(\ko_x) = -1$, hence skyscraper sheaves are not massless, and $\Re Z_a(F[1])<0$ by Lemma \ref{lemma:pathsEnriques_massless10}, unless $a=-2/5$ and $F=\ko_T(F_i^\pm)$.

Thus, if $a>-2/5$, then there are no massless objects, and so $\{\sigma_t: t>a\}$ converges to a genuine stability condition with central charge $Z_a$ and heart $\cA$ by Theorem \ref{T:strongquasiconvergence}. This contradicts that $\sigma_a$ is not a stability condition for $a>-2/5$.

We get that $a=-2/5$, and that $\{\ko_T(F_i^\pm)\}$ are the only massless stable objects of phase 1. As there are finitely many of them, Lemma~ \ref{L:qcpathsA4_massless} guarantees that \ref{axiom4} holds for $\nu>0$. By Theorem \ref{T:strongquasiconvergence} the path $\sigma_\bullet$ is strongly quasi-convergent. Since $Z_t(\ko_T(F_i^\pm)[1]) = -2/5-t$, the filtration is $0\subset \langle \ko_T(F_i^\pm): i=1, \dots, 10\rangle \subset \DCoh(T)$.
\end{proof}

\printendnotes

\bibliography{refs}{}
\bibliographystyle{plain}

\end{document}